\documentclass[11pt]{amsart}

\usepackage{amssymb,amsthm,amsfonts,latexsym,mathtools}

\usepackage{amsmath}
\usepackage{mathrsfs}
\usepackage{stmaryrd}
\usepackage[utf8]{inputenc}
\usepackage{color}
\usepackage{enumitem}
\usepackage{tabularx}
\usepackage{graphicx}
\usepackage{accents}
\usepackage{bbm}
\usepackage{comment}

\newlength{\dhatheight}
\newcommand{\doublehat}[1]{%
    \settoheight{\dhatheight}{\ensuremath{\hat{#1}}}%
    \addtolength{\dhatheight}{-0.35ex}%
    \hskip1pt\widehat{\vphantom{\rule{1pt}{\dhatheight}}%
    \smash{\hskip0pt\widehat{#1}}}}
\input{xy}
\xyoption{all}
\xyoption{poly}
\usepackage[all]{xy}
\allowdisplaybreaks

\usepackage[cal=boondoxo]{mathalfa}
 
\theoremstyle{plain}
\newtheorem{theorem}{Theorem}[section]
\newtheorem*{theorem*}{Theorem}
\newtheorem{lemma}[theorem]{Lemma}
\newtheorem{corollary}[theorem]{Corollary}
\newtheorem*{corollary*}{Corollary}
\newtheorem{proposition}[theorem]{Proposition}

\theoremstyle{definition}
\newtheorem{example}[theorem]{Example}
\newtheorem{definition}[theorem]{Definition}

\theoremstyle{remark}
\newtheorem{remark}[theorem]{Remark}

  \newenvironment{proofofcase}[1][\claimproofname]
  {\begin{proof}[#1]}
  {\end{proof}}

 \newcommand\CC{{\mathbb{C}}}
 \newcommand\FF{{\mathbb{F}}}
 \def\KK{{\mathbb{K}}}
 \def\kk{{\mathbbm{k}}}
 
 \def\sd{\mathrm{sd}}

 \newcommand\ZZ{{\mathbb{Z}}}
 
 \def\A{\mathcal{A}}

 \def\D{{\mathcal D}}
 
 \def\F{{\mathcal F}}
 \def\G{{\mathcal G}}
 
 \def\I{{\mathcal I}}
 
 \def\K{{\mathcal K}}

 \def\P{{\mathcal P}}
 \def\Q{{\mathcal Q}}
 
 \def\S{{\mathcal S}}

 \DeclareMathOperator\Lk{Lk}

 \def\op{\mathrm{op}}

 \newcommand{\GF}[1]{\mathbb{F}_{#1}}
 \DeclareMathOperator{\chara}{char}

 \DeclareMathOperator\Tr{Trace}
 
 \DeclareMathOperator\GL{GL}
 \DeclareMathOperator\PSL{PSL}
 \DeclareMathOperator\PGL{PGL}
 \DeclareMathOperator\SL{SL}

 \DeclareMathOperator{\GU}{GU}
 \DeclareMathOperator{\SU}{SU}
 \DeclareMathOperator{\PGU}{PGU}
 \DeclareMathOperator{\PSU}{PSU}
 
 \DeclareMathOperator{\Sym}{Sym}

 \def\redF{\widehat{\mathcal{F}}}

 \def\redS{\mathring{\mathcal{S}}} 
 \DeclareMathOperator\Rad{Rad}

 \DeclareMathOperator\Aut{Aut}

 \DeclareMathOperator\Id{Id}

 \DeclareMathOperator{\codim}{codim}
 \DeclareMathOperator\diam{diam}

 \DeclareMathOperator\Fix{Fix}

 \newcommand{\tq}{\mathrel{{\ensuremath{\: : \: }}}}

 \DeclareMathOperator\Stab{Stab}
 
 \def\Ap{\mathcal{A}_p}
 
 \def\groupiso{\cong}

 \newcommand\gen[1]{\left\langle#1\right\rangle}

\begin{document}

\title [Homotopy properties of the complex of frames of a unitary space]{Homotopy properties of the complex of frames of a unitary space}
   \author{Kevin I. Piterman}
   \address{Philipps-Universit\"at Marburg \\
   Fachbereich Mathematik und Informatik \\
   35032 Marburg, Germany \\
   and\\
   Departamento de Matem\'atica-IMAS\\
   FCEyN\\
   Universidad de Buenos Aires\\
   Buenos Aires, Argentina}
\email{kpiterman@dm.uba.ar}
  \author{Volkmar Welker}
   \address{Philipps-Universit\"at Marburg \\
   Fachbereich Mathematik und Informatik \\
   35032 Marburg, Germany}
   
 \email{welker@mathematik.uni-marburg.de}

\begin{abstract}
Let $V$ be a finite dimensional vector space equipped with a non-degenerate Hermitian form over a field $\KK$.
Let $\G(V)$ be the graph with vertex set the $1$-dimensional non-degenerate subspaces of $V$ and adjacency relation given by orthogonality.
We give a complete description of when $\G(V)$ is connected in terms of the dimension of $V$ and the size of the ground field $\KK$.
Furthermore, we prove that if $\dim(V) > 4$ then the clique complex $\F(V)$ of $\G(V)$ is simply connected.
For finite fields $\KK$, we also compute the eigenvalues of the adjacency matrix of $\G(V)$.
Then by Garland's method, we conclude that $\tilde{H}_m(\F(V);\kk) = 0$ for all $0\leq m\leq \dim(V)-3$, where $\kk$ is a field of characteristic $0$, provided that $\dim(V)^2 \leq |\KK|$.
Under these assumptions, we deduce that the barycentric subdivision of $\F(V)$ deformation retracts to the order complex of the certain rank selection of $\F(V)$ which is Cohen-Macaulay over $\kk$.

Finally, we apply our results to the Quillen poset of elementary abelian $p$-subgroups of a finite group and to the study of geometric properties of the poset of non-degenerate subspaces of $V$ and the poset of orthogonal decompositions of $V$.
\end{abstract}



\maketitle

\section{Introduction}

Let $V$ be a vector space of dimension $n$ over a field $\KK$. 
We equip $V$ with a non-degenerate Hermitian form $\Psi(\bullet,\bullet)$ with respect to a non-trivial involution $\tau\in \Aut(\KK)$, and we say that $(V,\Psi)$, or just $V$, is a unitary space.

We consider the (undirected) simple graph $\G(V)$ whose vertices are the non-degenerate $1$-dimensional subspaces of $V$, and two vertices are adjacent 
if and only if they are orthogonal with respect to the Hermitian form. 
The \textit{frame complex} $\F(V)$ is the clique complex of $\G(V)$, i.e., the simplicial complex on the vertex set of $\G(V)$ whose simplices, which we call \textit{partial frames}, are the subsets 
of the vertex set which form a clique in $\G(V)$. 
Note that $\F(V)$ is a pure simplicial complex of dimension $n-1$. 

\medskip

Our study of $\G(V)$ and $\F(V)$ is motivated from two directions. 

First, showing that for $\KK=\GF{q^ 2}$ the 
homology of $\F(V)$ in dimension $n-2$ is non-zero is a step on the path outlined by Aschbacher and Smith \cite[Section 4]{AS93} to remove 
a restriction on unitary components in the groups covered by their proof of a conjecture by Quillen \cite{Qui78} on the poset of non-trivial (elementary abelian) $p$-subgroups of a finite group for primes $p > 5$.
By the recent proof of the conjecture for $p=3,5$ under identical restrictions \cite{PS}, the same applies to these primes. Hence, Theorem \ref{coroHomologyCohenMacaulay} can be seen as a step in this program towards eliminating the hypothesis on unitary groups.
While our results do not yet fully suffice to yield new 
cases of the original conjecture, they provide new insights into the study of the Quillen dimension property for unitary groups.
In particular, we establish the Quillen dimension property for $\PGU_7(3)$ and $p=2$ in Example \ref{ex:qdp}, which is a case not covered by the methods of \cite{AS93}.

The work of Aschbacher and Smith \cite{AS93} indicates the connection between the poset of
non-trivial orthogonal decompositions of an $n$-dimensional unitary space over $\GF{q^2}$ and the poset of non-trivial $p$-subgroups of the associated unitary group $\GU_n(q)$.
They show that, for $p$ dividing $q+1$, the decomposition poset is homotopy equivalent 
to the poset of elementary abelian $p$-subgroups of the (general) unitary group strictly containing the order $p$ subgroup of the center. For odd $p$, or odd $n$ if $p=2$, this then provides an injection of top homology groups of the decomposition poset and the poset of elementary abelian $p$-subgroups of the projective unitary group.
Here the implications on the Quillen conjecture are even more direct than in the case of
$\F(V)$.
By Proposition \ref{prop:CMimplicationFrames}, 
the Cohen-Macaulay property is transported from $\F(V)$ to the poset of non-degenerate subspaces, and thus, by Corollary \ref{coro:connectionCM}, to the decomposition poset and to this homotopy equivalent $p$-subgroup poset when in addition $p\mid q+1$.
From known and new results on this poset (see next paragraph), we deduce a unitary-analog of the Cohen-Macaulayness of the poset of elementary abelian $p$-subgroups in the general linear group from Quillen's original paper \cite{Qui78}.
Even beyond unitary spaces, $\F(V)$, the decomposition poset and the poset of non-trivial 
non-degenerate subspaces are closely related. This is also demonstrated in Section \ref{sec:forms}. 

Indeed, the poset of non-trivial non-degenerate subspaces of a unitary space provides our second motivation, and
$\F(V)$ can be used to analyze its order complex. 
In \cite{DGM} Devillers, Gramlich and M\"uhlherr determine the homotopy type of this complex for infinite fields and certain finite fields. In Proposition \ref{coro:applicationOtherPosets}, we exhibit the relation between $\F(V)$ and the poset of non-degenerate subspaces, and extend the results from \cite{DGM} for certain finite fields. Finally, the unitary analog of the
Stiefel complex studied by Vogtmann in \cite{Vo82} for non-degenerate bilinear forms can be analyzed using $\F(V)$ (this will be exploited in an upcoming work \cite{PitWel2}).

\medskip

Our main results are the following. In Section \ref{sec:walk} we prove:

\begin{theorem}[Connectivity]
  \label{connectedFramePoset}
  Let $V$ be a unitary space of dimension $n\geq 2$ over some field $\KK$.
  Then the following are equivalent:
  \begin{itemize}
    \item[(i)] $\G(V)$ is connected.
    \item[(ii)] $\F(V)$ is connected. 
    \item[(iii)] One of the following conditions holds:
      \begin{itemize}
        \item[(a)] $n = 2$ and $\KK = \GF{2^2}$,
        \item[(b)] $n = 3$ and $\KK\neq \GF{2^2}$,
        \item[(c)] $n \geq 4$.
      \end{itemize}
  \end{itemize}

  The diameter of $\G(V)$ is 
  $1$ in case (a) and $2$ in case (c).
  Also in case (b), the diameter is $3$ if $V$ contains non-zero isotropic vectors, and it is $2$ otherwise.
\end{theorem}

The theorem implies that $\F(V)$ is disconnected if and only if 
either $V \cong \GF{2^2}^3$ or
$V \cong \KK^2$ for $\KK \neq \GF{2^2}$. 
In particular, for $n \leq 3$ we can easily determine the homotopy type
of $\F(V)$.

In Section \ref{sec:simple} we study the simple connectivity of $\F(V)$
and prove:

\begin{theorem}[Simple connectivity]
  \label{simplyConnectedFramePoset}
  Let $V$ be a unitary space of dimension $n\geq 2$ over $\KK$.
  The following assertions hold:
  \begin{itemize}
    \item[(i)] If $n\geq 5$ and $(n,\KK)\neq (6,\GF{2^2})$, then $\F(V)$ is simply connected.
    \item[(ii)] If $n = 6$ and $\KK = \GF{2^2}$, then $\pi_1(\F(V)) \groupiso C_2\times C_2$.
    \item[(iii)] If $n = 4$ and $\KK=\GF{2^2}$ or $\GF{3^2}$, then $\F(V)$ is not simply connected. In the former case $\F(V)$ is a wedge of $81$ spheres of dimension $1$, and in the latter case $\tilde{H}_*(\F(V),\ZZ)$ is non-zero in degrees $2$ and $1$.
    \item[(iv)] If $n = 4$ and $\KK = \GF{q^ 2}$ for $q \geq 4$, then $\pi_1(\F(V))$ has Kazhdan's property (T). In particular $\pi_1(\F(V))$ is not a non-trivial free group, and its abelianization is a finite group.

    \item[(v)] If $n = 3$ then $\F(V)$ is homotopy equivalent to a non-trivial wedge of $1$-spheres if $\KK \neq \GF{2^ 2}$, and to a non-trivial wedge of $0$-spheres if $\KK= \GF{2^ 2}$.
  \end{itemize}
\end{theorem}

It remains open the question of whenever $\F(V)$ is simply connected if $\dim(V) = 4$ and $|\KK| \geq 4^2$.

\medskip

We study in Section \ref{sec:eigen} the eigenvalues of the adjacency matrix of the graph $\G(V)$ when $\KK$ is a finite field.
Let $d_n:=\frac{q^{n-2}(q^{n-1}-(-1)^{n-1})}{q+1}$, which is the number of $1$-dimensional non-degenerate subspaces of a unitary space of dimension $n-1$ over $\GF{q^2}$.
Then, for an $n$-dimensional unitary space $V$ over $\GF{q^2}$, $\G(V)$ is a $d_n$-regular graph with $d_{n+1}$ vertices.

\begin{theorem}
[{Eigenvalues of $\G(V)$}]
\label{theoremEigenvalues}
Let $V$ be a unitary space of dimension $n\geq 2$ over $\GF{q^2}$.
Let
\[ \mu_1 := d_n, \quad \mu_2 := q^{n-2}, \quad \mu_3 := (-1)^n q^{n-3}, \quad \mu_4 := -q^{n-2}.\]
The eigenvalues of the adjacency matrix of the graph $\G(V)$ are as follows.
\begin{itemize}
\item[(i)] If $n = 2$, the eigenvalues are $\mu_1$ and $\mu_4$.
\item[(ii)] If $n \geq 3$ and $q\neq 2$, the eigenvalues are $\mu_1$, $\mu_2$, $\mu_3$ and $\mu_4$.
\item[(iii)] If $n = 3$ and $q = 2$, the eigenvalues are $\mu_1$ and $\mu_3$.
\item[(iv)] If $n > 3$ and $q = 2$, the eigenvalues are $\mu_1$, $\mu_3$ and $(-1)^n \mu_4$.
\end{itemize}
\end{theorem}

We compute the multiplicities of these eigenvalues in Corollary \ref{multiplicitiesEigenvaluesDimAtLeast3} for $n > 2$, and in Proposition \ref{walksAndEigenvaluesLowDim2} for $n = 2$.
Note that the eigenspaces are $\GU(V)$-modules.

In Section \ref{sec:garland} we use Theorem \ref{theoremEigenvalues} to conclude
by Garland's method (see \cite{BS,Garland, Zuk}):

\begin{theorem}
\label{coroHomologyCohenMacaulay}
  Let $V$ be a unitary space of dimension $n\geq 2$ over $\GF{q^2}$.
  Let $\kk$ be a field of characteristic $0$.
  If $n < q+1$, then $\tilde{H}_i(\F(V),\kk)=0$ for every $0\leq i \leq n-3$.
  In particular, $\widehat{\F}(V)$ is Cohen-Macaulay over $\kk$ if $n < q+1$.
\end{theorem}

Here $\widehat{\F}(V) \subseteq \F(V)$ is the poset of all partial frames of size $\neq 0,n-2$ ordered by inclusion.
Section \ref{sec:garland} also contains more general results
about $\kk$-homological connectivity of $\F(V)$.

Finally, in Remark \ref{rk:nonCM} we prove that $\tilde{H}_{n-3}(\F(V),\kk)\neq 0$ for every $q\geq 3$ if, for instance,  $n = q(q-1)+1$. In particular, $\redF(V)$ is not Cohen-Macaulay over any field $\kk$ if $n\geq q(q-1)+1$. Therefore, the upper bound for $n$ provided in Theorem \ref{coroHomologyCohenMacaulay} cannot be extended beyond $q(q-1)+1$.

\section*{Acknowledgments}
This work was carried out during different research stays of the first author in IMAS-CONICET as CONICET Postdoctoral Fellow,
the Mathematisches Forschungsinstitut Oberwolfach as Oberwolfach Leibniz Fellow, 
Institut Mittag-Leffler as Junior Postdoctoral Fellow, and, currently, in Philipps-Universität Marburg supported by the Alexander von Humboldt Foundation.
The first author is very grateful to these institutions for their support.

\section{Spaces with forms and associated posets and simplicial complexes}
\label{sec:basics}

\subsection{Spaces with forms.} \label{sec:forms}
Our main results are for finite-dimensional unitary spaces or, equivalently, non-degenerate Hermitian forms in finite-dimensional vector spaces. Nevertheless, we will formulate each intermediate result in the 
generality which is allowed when following the proof in the
Hermitian case.

Many of our main objects will be posets or simplicial complexes. Indeed
we will identify a poset $\P$ with the simplicial complex of all chains in $\P$
which is called the \textit{order complex} of $\P$. In particular, we can
talk about homology or homotopy of posets, and, hence, about the Cohen-Macaulay property. We refer to 
\cite{Bjo95} for the definition and a general introduction to combinatorial methods for studying the 
topology and homology of simplicial complexes and order complexes of posets.
For a poset $\P$ and
$p \in \P$, we will write $\P_{\leq p}$ for $\{ q\in \P\,|\,q \leq p\}$. 
Analogously defined are $\P_{< p}$, $\P_{\geq p}$ and $\P_{> p}$.
We also set $[p,p'] = \P_{\leq p'} \cap \P_{ \geq p}$ and $(p,p')= \P_{< p'} \cap \P_{ > p}$ for $p \leq p'$.
For a simplicial complex
$\K$ and a simplex $\sigma \in \K$ we write $\Lk_\K(\sigma)$ for 
the link of $\sigma$ in $\K$. For a number $m \geq 0$ we write $\K^m$
for the $m$-skeleton of $\K$ and $\K_m$ for the set of $(m-1)$-dimensional
simplices in $\K$. For a poset $\P$ and a simplicial complex $\K$ we
write $|\P|$ and $|\K|$ to denote the geometric realization of the order
complex of $\P$ and $\K$ respectively. Recall also that a simplicial complex
is called pure if all its maximal faces have the same dimension.

We write $(V,\Psi)$ to denote a finite-dimensional vector space $V$ over a field $\KK$ together with a form $\Psi : V \times V \rightarrow \KK$.
For a fixed field automorphism $\tau : \KK \rightarrow \KK$ of order $\leq 2$ and $\epsilon \in \{ \pm 1\}$, the form $\Psi$ is 
$\epsilon$-Hermitian if $\Psi(v,w) = \epsilon 
\tau(\Psi(w,v))$ and $\Psi(\alpha v + \beta v' ,w) = 
\alpha\Psi(v,w)+\beta \Psi(v',w)$.
In addition, if $\chara(\KK) = 2$ and $\tau$ is the identity, we also require $\Psi(v,v) = 0$ for all $v\in V$.
From now on we assume:
\[ \text{{\bf (Forms)} $\Psi$ is a } \epsilon \text{-Hermitian form for a field automorphism } \tau \text{ of order } \leq 2 \text{ and }  \epsilon \in \{ \pm 1 \} . \]
It follows that all our forms are reflexive; i.e., for all $v,w\in V$ we have $\Psi(v,w) = 0$ if and only if 
$\Psi(w,v) = 0$.
Note that if $\tau$ has order $2$ and $\epsilon = 1$ then 
$\Psi$ is a Hermitian form in the sense used in the introduction.
The assumptions also imply that Witt's lemma holds (see \cite[p.81]{AscFGT}) and so every isometry between two
subspaces of $V$ can be extended to an isometry of $V$.

The symbol $\perp$ denotes the orthogonality relation between vectors or subspaces.
That is, for $v,w\in V$, $v\perp w$ if and only if $\Psi(v,w) = 0$.
For $S,T\subseteq V$, write $S\perp T$ if and only if $v\perp w$ for all $v\in S$ and $w\in T$.
Note that since $\Psi$ is a reflexive form, the relation $\perp$ is symmetric. 

For a subspace $S\leq V$, let 
\[ \Rad(S) := \{v\in S\tq v\perp w \,\, \text{for all } w\in S\}\]
denote the radical of $S$.
We say that $S$ is a non-degenerate subspace if $\Rad(S) = 0$, and $\Psi$ is a non-degenerate form if $V$ is non-degenerate.
Let $S^\perp:=\{v\in V\tq~v\perp~w~\text{ for all } w\in S\}$.
Then $\Rad(S) = S\cap S^\perp$.
Thus, $\Psi$ is non-degenerate if and only if $\Rad(V)  = V^\perp = 0$.

A vector $v\in V$ is called degenerate or isotropic if $\Psi(v,v) = 0$.
A totally isotropic subspace is a subspace $S\leq V$ such that $S = \Rad(S)$.
We say that $V$ is anisotropic if it contains no non-zero isotropic vectors.

We write $|v|:=\Psi(v,v)$ for a 
vector $v\in V$. This notation is non-standard and conflicts with 
standard conventions, but it will be very useful for our considerations. 

Note that a vector $v\in V$ spans a non-degenerate $1$-dimensional subspace if and only if $|v|\neq 0$.
If $\{S_i\}_{i\in I}$ is a set of pairwise orthogonal non-degenerate subspaces of $V$, then its span is the direct sum $\bigoplus_{i\in I} S_i$, which is also non-degenerate.

An isometry between two vector spaces with forms $(V,\Psi)$ and $(W,\Psi')$ is an isomorphism $\alpha:V\to W$ such that $\Psi'(\alpha(v),\alpha(w)) = \Psi(v,w)$ for all $v,w\in V$.

We will use the following well-known facts about subspaces of a finite-dimensional vector space with a non-degenerate form (see
for example \cite[Section 19]{AscFGT}).

\begin{proposition}
\label{dimensionTheorem}
  Let $(V,\Psi)$ be a finite-dimensional vector space equipped with a non-degenerate form $\Psi$, and let $S\leq V$ be a subspace.
  The following assertions hold:
  \begin{itemize}
    \item[(i)] $\dim(S^\perp) = \dim(V) - \dim(S) = \codim(S)$;
    \item[(ii)] $S$ is non-degenerate if and only if $S^\perp$ is     non-degenerate, if and only if $S\oplus S^\perp = V$;
    \item[(iii)] $(S^\perp)^\perp = S$;
    \item[(iv)] $\Rad(S^\perp) = \Rad(S)$;
    \item[(v)] $\dim(\Rad(S)) \leq \dim(V)/2$;
    \item[(vi)] there exists an orthogonal decomposition $S =   
      \Rad(S) \oplus S'$, where $(S',\Psi|_{S'})$ is non-degenerate. In particular $S/\Rad(S)$ is (naturally) isometric to $S'$;
    \item[(vii)] $\codim(S) > \dim(\Rad(S))$, $\epsilon=1$ and $(|\tau|,\chara(\KK))\neq (1,2)$, if and only if $S^\perp$  
      contains non-isotropic vectors.
  \end{itemize}
\end{proposition}

Next, we define combinatorial objects associated to 
vector spaces with forms.

By $\S(V,\Psi)$ we denote the poset of non-degenerate 
subspaces of $V$ ordered by inclusion.
Note that $0 \in \S(V,\Psi)$ is the unique minimal element 
of $\S(V,\Psi)$ and if $V$ is non-degenerate then $V \in \S(V,\Psi)$ is the unique maximal element of $\S(V,\Psi)$.
We write $\mathring{\S}(V,\Psi)$ for $\S(V,\Psi) \setminus \{0,V\}$. 
The next proposition is a simple fact from linear algebra (see also Lemma \ref{lm:dimensions}).

\begin{proposition} \label{jordanhoelder}
    Let $(V,\Psi)$ a finite-dimensional vector space equipped
    with a possibly degenerate form $\Psi$.
    If we have two decompositions
    $$\Rad(V) \oplus V_1 \oplus \cdots \oplus V_r = V =
    \Rad(V) \oplus W_1 \oplus \cdots \oplus W_s,$$ 
    where $V_1,\ldots, V_r$ and $W_1,\ldots, W_s$ are
    two sets of mutually orthogonal minimal elements of
    $\mathring{\S}(V,\Psi)$, then $r=s$ and, after reindexing,
    $\dim(V_i) = \dim(W_i)$ for $i=1,\ldots, r=s$.
\end{proposition}

Note that in a decomposition as in Proposition
\ref{jordanhoelder}, the conclusion that after reindexing
one can achieve that $V_i$ and $W_i$ are isometric is false in general. 
For example, when $\Psi$ is an orthogonal form over a finite field of odd characteristic this conclusion is
wrong already for a non-degenerate space of dimension $2$. 

The conclusions of Proposition \ref{jordanhoelder} are also false if we drop the assumption $\Psi(v,v) = 0$ for all $v\in V$ when $\chara(\KK)=2$ and $\tau=1$, as the following  example shows.

\begin{example}
    Let $V = \KK^3$ with $\chara(\KK)=2$, and let $\Psi(v,w) = \sum_{i=1}^3 v_i w_i$, for $v,w\in V$.
    Let $e_i$ denote the canonical vectors.
    Then $\Psi(e_i,e_i)=1$ and $\Psi(e_i,e_j) = 0$ for $i\neq j$.
    If $E_i = \gen{e_i}$, we see that $V = E_1\oplus E_2\oplus E_3$ is a decomposition of $V$ into mutually orthogonal minimal elements of $\redS(V,\Psi)$.

    On the other hand, let $v = (1,1,1)$ and $S = \gen{v}$.
    Then $\Psi(v,v) = 1\neq 0$.
    Now $w\in S^\perp$ if and only if $\sum_i w_i = 0$.
    Since the characteristic is $2$, for all $w\in S^\perp$ we have $\Psi(w,w) = \sum_i w_i^2 = \left(\sum_i w_i\right)^2 =0$.
    Hence $S^\perp$ is non-degenerate but contains no $1$-dimensional non-degenerate subspace.
    Thus $V = S \oplus S^\perp$ is also a decomposition of $V$ into mutually orthogonal minimal elements of $\redS(V,\Psi)$.
\end{example}

Next, we collect decompositions of $V$ into 
mutually orthogonal subspaces in a poset. 
Let $\D(V,\Psi)$ be the
poset of all direct sum 
decompositions $\Rad(V)\oplus V_1\oplus \cdots \oplus V_s$ for 
$V_i \in {\S}(V,\Psi) \setminus \{0\}$, such that
$V_1,\ldots, V_r$ are pairwise orthogonal and the order of the $V_i$ 
does not matter. We order the elements
by refinement, that is, we set
$\Rad(V) \oplus V_1\oplus \cdots \oplus V_s \preceq \Rad(V) \oplus W_1 \oplus \cdots \oplus W_r$ if for each $V_i$ there is a $W_j$ with
$V_i \subseteq W_j$.  We write $\mathring{\D}(V,\Psi) = \D(V,\Psi) 
\setminus \{ V\}$. Thus $\mathring{\D}(V,\Psi) \neq \D(V,\Psi)$ if and
only if $V$ is non-degenerate.

Denote by $\Pi_r$ the partially ordered set
of all set-partitions of $[r]:=\{1,\ldots, r\}$ ordered by refinement. The following proposition
is a simple consequence of the definitions and Proposition \ref{jordanhoelder}.

\begin{proposition} \label{prop:elementary}
  Let $(V,\Psi)$ be a finite-dimensional vector space equipped with a non-degenerate form $\Psi$.
  \begin{itemize}
  \item[(i)]
  The order complex of the poset $\D(V,\Psi)$ is pure of
  dimension $\dim(\S(V,\Psi))-1$.
  \item[(ii)] For 
  $\pi = V_1\oplus \cdots \oplus V_r \in \D(V,\Psi)$ we have that
  \[ \D(V,\Psi)_{\succeq \pi}  \cong \Pi_r,\]
  \[ \D(V,\Psi)_{\preceq \pi}  
  \cong \D(V_1,\Psi|_{V_1}) \times\cdots \times \D(V_r,\Psi|_{V_r}).\]
  \end{itemize}
\end{proposition}

Assume that $\Psi$ is non-degenerate.
For an element $\pi = V_1 \oplus\cdots \oplus V_r \in \D(V,\Psi)$, we write $|\pi|=r$ and 
let $\S_\pi = \{ \bigoplus_{i\in I} V_i : \emptyset \neq I \subsetneq [r]\} \subseteq\mathring{\S}(V,\Psi)$.
It is easily seen that $\S_\pi$ is
isomorphic to the proper part of the lattice of subsets of an $r$-element set.
We can then define the map 
$g$ from the order complex of
$\mathring{\S}(V,\Psi)$ to the poset $\mathring{\D}(V,\Psi)$
which sends a simplex $S_0 < \cdots < S_r$ to 
$S_0 \oplus S_{1}' \oplus \cdots \oplus S_r' \oplus S_{r+1}'$
for $S_i = S_{i-1} \oplus S_i'$, where $S_i' = S_{i-1}^\perp \cap S_i$ and $S_{r+1} := V$.
We call $g$ the decomposition map.

\begin{lemma}  \label{lem:fiber}
  Let $V$ be a finite-dimensional vector space with a non-degenerate form $\Psi$.
  Then the decomposition map
  $g$ from the order complex of 
  $\mathring{\S}(V,\Psi)$ to $\mathring{\D}(V,\Psi)^\op$  is a poset map.
  
  Furthermore, for $\pi = V_1 \oplus\cdots \oplus V_r \in \mathring{\D}(V,\Psi)$
  we have that 
  $g^{-1}(\mathring{\D}(V,\Psi)_{\succeq \pi})$ is the
  order complex of 
  $\S_\pi$. 
 \end{lemma}
 \begin{proof} 
   The decomposition map is easily seen to be a poset map to $\mathring{\D}(V,\Psi)^\op$.
   
   If $\sigma = W_1 \oplus\cdots \oplus W_{s+1} \in \D(V,\Psi)_{\succeq \pi}$
   and $S_1 < \cdots < S_s$ is a
   chain in $\mathring{\S}(V,\Psi)$ mapped to $\sigma$ then, after possible renumbering, we can assume
   $S_i = W_1 \oplus \cdots \oplus W_i$. But each $W_i$ is a sum of some of the $V_j$, so $S_1 < \cdots < S_s$ is a chain in $\S_\pi$.
   Conversely, each chain in $\S_\pi$ is mapped to some
   $\sigma \succeq \pi$. 
 \end{proof}
 
We use the symbol ``$\ast$'' to denote the join of simplicial complexes or posets, and ``$\bigvee$'' to denote the wedge of spaces.
 
 \begin{proposition} \label{prop:decomp}
Let $\Psi$ be a non-degenerate form on a finite-dimensional vector space $V$ over a finite field. Then
   $\mathring{\S}(V,\Psi)$ 
   is homotopy equivalent to
   $$\bigvee_{\pi \in \D(V,\Psi)}
   S^{|\pi|-2} * \D(V,\Psi)_{\prec \pi},$$ where $S^m$ denotes the $m$-sphere.
 \end{proposition}
 \begin{proof}
 By Lemma \ref{lem:fiber}, the decomposition map $g$ from the order complex of
   $\mathring{\S}(V,\Psi)$ to $\D(V,\Psi)^{\op}$ fulfils the assumptions of \cite[Theorem 2.5]{BWW}, so the assertion follows from the theorem.
 \end{proof}
 
The previous proposition is formulated for vector spaces over finite fields since its proof relies on \cite[Theorem 2.5]{BWW} which is proved for finite posets.
Nevertheless, \cite[Theorem 2.5]{BWW} extends to a certain family of infinite posets, making Proposition \ref{prop:decomp} hold for infinite fields.
We will provide further details in \cite{PitWel2}.

In Section \ref{sec:further} we will compare results that follow from part (i) of the following corollary with results by
Devillers, Gramlich and M\"uhlherr in \cite[Main Theorem]{DGM}.

It will be convenient to use the following terminology.
A poset or simplicial complex $X$ is said to be spherical (resp. spherical over $\kk$) if $X$ is $(\dim(X)-1)$-connected or, equivalently, homotopy equivalent to a wedge of spheres of dimension $\dim(X)$ (resp. all homology groups over $\kk$ of degree $\leq \dim(X)-1$ vanish).
 
 \begin{corollary}
 \label{coro:decomp}
 Let $\Psi$ be a non-degenerate form on a finite-dimensional vector space $V$ over
a finite field and let $\kk$ be a ring. Then

$$\S(V,\Psi) \text{ is homotopy Cohen-Macaulay (resp., Cohen-Macaulay over } \kk \text{) }$$ $$ \Updownarrow $$ 
$$\D(V,\Psi) \text{ is homotopy Cohen-Macaulay (resp., Cohen-Macaulay over } \kk \text{)}.$$
 \end{corollary}
 \begin{proof}
   First, we recall some facts about intervals.
   
   \noindent (D) ~~ Any interval 
   in $\D(V,\Psi)$, which is not 
   of the form $\D(V,\Psi)_{\prec \pi}$
   for some $\pi \in \D(V,\Psi)$, is
   an interval in a set partition lattice and hence homotopically
   Cohen-Macaulay (see \cite{Bjo95}). 

   \noindent (S)~~ The interval between $U$ and $U'$ in $\S(V,\Psi)$ is
       isomorphic to $\S(U_1,\Psi|_{U_1})$ for the orthogonal complement $U_1$ of $U$ in $U'$. In particular, if $\S(V,\Psi)$ is Cohen-Macaulay (resp. Cohen-Macaulay over $\kk$), then every interval is.
       
   \begin{itemize}
       \item[($\Downarrow$)] By (D) it suffices to show that any interval $\D(V,\Psi)_{\prec \pi}$ is spherical (resp., spherical over $\kk$). 
       
       Let $\pi = V_1 \oplus \cdots \oplus V_r$. Then by (S) we have that $\S(V_i,\Psi|_{W_i})$ is isomorphic to an the interval between 
       $U = V_1\oplus \cdots V_{i-1} \oplus V_{i+1} \oplus \cdots \oplus V_r$
       and $V$ in $\S(V,\Psi)$ and hence homotopically Cohen-Macaulay (resp., Cohen-Macaulay over $\kk$). From Proposition \ref{prop:decomp} it then follows that $\mathring{\D}(V_i,\Psi|_{V_i})$ is spherical (resp. spherical over $\kk$). Since $$\D(V,\Psi)_{\prec \pi} \cong \Big(~\D(V_1,\Psi|_{V_1}) \times \cdots \times \D(V_r,\Psi|_{V_r})~\Big) \setminus \{(V_1,\ldots, V_r)\}$$ it follows from \cite[(9.7)]{Bjo95} that 
       $\D(V,\Psi)_{\prec \pi}$ is also spherical (resp. spherical over $\kk$). 

       This shows that $\D(V,\Psi)$ is homotopically Cohen-Macaulay (resp., Cohen-Ma\-cau\-lay over $\kk$).
       
       \item[($\Uparrow$)]
       Consider an interval between $U$ and $U'$ in $\S(V,\Psi)$. Then by (S) this interval is
       isomorphic to $\S(U_1,\Psi|_{U_1})$ for the orthogonal complement $U_1$ of $U$ in $U'$.
       
       By Proposition \ref{prop:decomp}, $\S(U_1,\Psi|_{U_1})$ is spherical (resp. spherical over $\kk$) if each lower interval $\D(U_1,\Psi|_{U_1})_{\prec\pi}$ is, for $\pi \in \D(U_1,\Psi|_{U_1})$.
       Fix an orthogonal decomposition $V = U_1\oplus U_2\oplus \ldots\oplus U_k$ such that $U_i$ is minimal in $\mathring{\S}(V,\Psi)$ for $i\geq 2$.
       For $\pi \in \D(U_1,\Psi|_{U_1})$, let $\pi' = \pi \oplus U_2\oplus\ldots\oplus U_k \in \D(V,\Psi)$.
       Then $\D(U_1,\Psi|_{U_1})_{\prec\pi} = \D(V,\Psi)_{\prec \pi'}$ is spherical (resp. spherical over $\kk$).
\end{itemize}
    
 \end{proof}

The \textit{subspace orthogonality graph} $\G(V,\Psi)$ of $(V, \Psi)$ is the (undirected) graph whose vertices are the minimal elements of $\mathring{\S}(V,\Psi)$ with edges 
between vertices which are orthogonal subspaces.
We sometimes write $S \in \G(V,\Psi)$ and mean that  $S$ is in the vertex set of $\G(V,\Psi)$ or, equivalently, a minimal non-zero non-degenerate subspace of $V$.

The \textit{frame complex} $\F(V,\Psi)$ of $(V,\Psi)$ is the clique complex of $\G(V,\Psi)$. That is, the simplicial complex whose simplices are the subsets of the vertex set of
$\G(V,\Psi)$ which form a clique.

We call a simplex $\sigma = \{S_1,\ldots,S_r\}$ of $\F(V,\Psi)$
of dimension $r-1$ a \textit{partial frame} or \textit{partial $r$-frame}. The subspaces 
from the simplex 
$\sigma$ span a non-degenerate subspace of dimension $\dim(S_1)+\cdots + \dim(S_r) \leq \dim(V/\Rad(V))$.
If $\sigma$ 
spans a complement to $\Rad(V)$ then we call $\sigma$ a \textit{frame}. 
It follows that $\sigma$ is 
maximal if and only if $V = S_1\oplus \cdots \oplus S_r \oplus \Rad(V)$ or 
equivalently, $\dim(S_1) + \cdots +  \dim(S_r) = \dim(V/\Rad(V))$.

The next fact is an immediate consequence of Proposition \ref{jordanhoelder}.

\begin{corollary} 
\label{coro:pureFrameComplex}
  $\F(V,\Psi)$ is a pure simplicial complex.
\end{corollary}

The following lemma follows from the simple fact that  
the graph $\G(V,\Psi)$ 
can be considered as the $1$-skeleton of $\F(V,\Psi)$.

\begin{lemma} \label{lem:conn}
  $\G(V,\Psi)$ is connected as a graph if and only if $\F(V,\Psi)$ is connected as a simplicial complex.
\end{lemma}

In 
$\widehat{\F}(V,\Psi)$ we collect all non-empty simplices $\sigma = \{S_1,\ldots,S_r\}$ for which there
is no minimal element $S$ of $\mathring{\S}(V,\Psi)$ such that
$V = S_1 \oplus \cdots \oplus S_r \oplus S \oplus \Rad(V)$. We consider $\widehat{\F}(V,\Psi)$ as
a poset ordered by inclusion and
call it the \textit{reduced frame poset} of $V$.

\begin{lemma}\label{lemmaHatPoset}
  Let $(V,\Psi)$ be finite-dimensional 
  and non-degenerate. Then 
  \begin{itemize}
      \item[(i)] 
  if $\dim(\F(V,\Psi)) \geq 1$ then $\F(V,\Psi)$ deformation retracts onto a subcomplex of
  dimension $\dim(\F(V,\Psi))-1$;
      \item[(ii)] if $\sigma$ is a simplex in $\F(V,\Psi)$ 
      such that $\dim(\Lk_{\F(V,\Psi)}(\sigma)) \geq 1$ 
      then $\Lk_{\F(V,\Psi)}(\sigma)$ is homotopy equivalent to the order complex
      of all $\sigma' \in \widehat{\F}(V,\Psi)$ which properly contain $\sigma$. In particular, if $\dim(\F(V, \Psi)) \geq 1$ then $\F(V,\Psi)$ is
      homotopy equivalent to the order complex of $\widehat{\F}(V,\Psi)$. 
  \end{itemize}
\end{lemma}

\begin{proof}
\begin{itemize}
  \item[(i)] 
  Let $\sigma\in \F(V,\Psi)$ be a simplex of dimension $\dim(\F(V,\Psi))-1$. 
  Then it is easily checked that $\hat{\sigma} = \sigma \cup \gen{\sigma}^\perp$ is the unique maximal simplex of
  $\F(V,\Psi)$ containing $\sigma$. Therefore $\sigma$ is a free face of $\sigma'$. 
  
This allows us to apply elementary collapses to $\F(V,\Psi)$ (see \cite{Bjo95} for terminology and details). Choosing for every maximal face a unique face of dimension
$\dim(\F(V,\Psi))-1$ contained in it, we can collapse $\F(V,\Psi)$ to a $(\dim(\F(V,\Psi))-1)$-dimensional subcomplex.
We can perform all the collapses at the same time, and clearly, this produces a (strong) deformation retract. The assertion then follows.
\item[(ii)] Since $\Lk_{\F(V,\Psi)}(\sigma)$ is isomorphic to 
$\F(\gen{\sigma}^\perp,\Psi|_{\gen{\sigma}^\perp})$ it suffices to show the assertion for
$\sigma = \emptyset$.
Consider $\F(V,\Psi)$ as a poset ordered by inclusion.  The map which
sends a co-maximal simplex $\sigma$ to the unique maximal simplex $\hat{\sigma}$ containing $\sigma$
and which is the identity on all other simplices is a closure operator with image $\widehat{\F}(V,\Psi)$. By \cite[Corollary 10.12]{Bjo95}
it follows that $\F(V,\Psi)$
and the order complex of $\widehat{\F}(V,\Psi)$ are homotopy equivalent.
\end{itemize}
\end{proof}

The subcomplex constructed in Lemma \ref{lemmaHatPoset}(i) is not unique since for each maximal face, any of its maximal proper 
subfaces can be used for the collapsing step.
We will mostly work with $\widehat{\F}(V,\Psi)$ 
since it avoids this non-canonical choice.

The following lemma determines the dimensions of $\F(V,\Psi)$, $\S(V,\Psi)$ and $\D(V,\Psi)$.

\begin{lemma}
\label{lm:dimensions}
Let $(V,\Psi)$ be finite-dimensional over $\KK$. Then
\begin{enumerate}
    \item A minimal element of $\redS(V,\Psi)$ has dimension $1$ or $2$.
    \item All minimal elements of $\redS(V,\Psi)$ have the same dimension.
    \item $\dim(\F(V,\Psi)) = \dim(\redS(V,\Psi))+1=\dim(\mathring{\D}(V,\Psi))+1$.
\end{enumerate}
\end{lemma}

\begin{proof}
Items (1) and (2) are straightforward from the basic properties of the $\epsilon$-Hermitian forms and Proposition \ref{jordanhoelder}.
Item (3) follows from Corollary \ref{coro:pureFrameComplex} and Proposition \ref{jordanhoelder} since the full frames correspond to the minimal elements of $\D(V,\Psi)$.
\end{proof}

We analyse next the Cohen-Macaulay property on $\widehat{\F}(V,\Psi)$ and the relation with $\S(V,\Psi)$.

\begin{remark}
    \label{rk:intervalsofFramesCM}
    Assume that $\widehat{\F}(V,\Psi)$ is homotopically Cohen-Macaulay.
    We show that $\widehat{\F}(S,\Psi|_S)$ is also Cohen-Macaulay for all $S\in \S(V,\Psi)$.
    Fix a full frame $\sigma_0$ of $S^\perp$.
    Since lower intervals in $\widehat{\F}(S,\Psi|_S)$ are always spherical, this poset is Cohen-Macaulay if and only if the upper intervals are spherical.
    Thus, for a frame $\rho\in\widehat{\F}(S,\Psi|_S)$, let $\rho':=\rho\cup \sigma_0$, which is also a frame.
    Then it is easy to see that $\sigma\in \widehat{\F}(S,\Psi|_S)_{>\rho}\mapsto \sigma \cup \sigma_0\in \widehat{\F}(V,\Psi)_{>\rho'}$ is an isomorphism with inverse $\sigma\mapsto \sigma \setminus \sigma_0$.
    Since $\widehat{\F}(V,\Psi)$ is Cohen-Macaulay, the interval $\widehat{\F}(V,\Psi)_{>\rho'}$ is spherical, and hence $\widehat{\F}(S,\Psi|_S)_{>\rho}$ is spherical.
    This shows that $\widehat{\F}(S,\Psi|_S)$ is also Cohen-Macaulay.

    The same proof shows that if $\widehat{\F}(V,\Psi)$ is Cohen-Macaulay over $\kk$, then $\widehat{\F}(S,\Psi|_S)$ is Cohen-Macaulay over $\kk$ for all $S\in \S(V,\Psi)$.
\end{remark}

\begin{proposition}
\label{prop:CMimplicationFrames}
Let $(V,\Psi)$ be finite-dimensional and non-degenerate.
If $\widehat{\F}(V,\Psi)$ is homotopy Cohen-Macaulay (resp. Cohen-Macaulay over $\kk$), then $\S(V,\Psi)$ and $\D(V,\Psi)$ are.
\end{proposition}

For the proof we use the following generalized Quillen's fiber-Theorem (see \cite[Proposition A.1]{PS}).

\begin{proposition}
\label{quillensFiberTheoremGeneral}
Let $f:\P\to \Q$ be a map between posets of finite height, $\kk$ a ring, and $m\geq -1$.
If for all $y\in Q$ we have that $f^{-1}(\Q_{\leq y}) * \Q_{>y}$ is ($\kk$-homologically) $m$-connected, then $f$ is a ($\kk$-homologically) $(m+1)$-equivalence.
\end{proposition}

\begin{proof}[Proof of Proposition \ref{prop:CMimplicationFrames}]
We assume that $\redF(V,\Psi)$ is homotopy Cohen-Macaulay, and show that $\S(V,\Psi)$ is also homotopy Cohen-Macaulay.
The proof for Cohen-Macaulay over $\kk$ is analogous.
Moreover, the conclusions for $\D(V,\Psi)$ will follow easily by Corollary \ref{coro:decomp}.

Any interval of $\S(V,\Psi)$ is isomorphic to $\redS(S,\Psi|_S)$, for some $S\in \S(V,\Psi)$.
Since by Remark \ref{rk:intervalsofFramesCM}, $\redF(S,\Psi|_S)$ is homotopy Cohen-Macaulay,  then we can conclude by induction that $\S(S,\Psi|_S)$ is homotopy Cohen-Macaulay when $\dim(S)<\dim(V)$.
Therefore, it remains to show that $\redS(V,\Psi)$ is spherical.

Let $m = \dim(\redS(V,\Psi))$.
Then $\F(V,\Psi)$ is $(m-1) = (\dim(\widehat{\F}(V,\Psi))-1)$-connected and $m+2$ is the size of any full frame of $V$.
Let $X$ denote the face poset of the $m$-skeleton of $\F(V,\Psi)$, that is, the poset of non-empty and non-full frames.
In particular, $X$ is spherical of dimension $m$.

Consider the poset map $\phi:X\to \redS(V,\Psi)$ that sends a (non-full) frame  to its span.
We apply Quillen's fiber-Theorem \ref{quillensFiberTheoremGeneral} to conclude that $\S(V,\Psi)$ is spherical.

Let $S \in \redS(V,\Psi)$.
Then $\phi^{-1}(\redS(V,\Psi)_{\leq S}) = \F(S,\Psi|_S)\simeq \redF(S,\Psi|_S)$, and the latter is spherical of dimension $m_S$, where $m_S+2$ is the size of a full frame of $S$ by Lemma \ref{lm:dimensions}.
On the other hand, $\redS(V,\Psi)_{>S}$ is isomorphic to $\redS(S^\perp, \Psi|_{S^\perp})$.
Since $\redF(S^\perp,\Psi|_{S^\perp})$ is Cohen-Macaulay by Remark \ref{rk:intervalsofFramesCM}, by induction we conclude that $\redS(V,\Psi)_{>S}$ is spherical of dimension $m_{S^\perp}$.
As before, $m_{S^\perp}+2$ is the size of a full frame in $S^\perp$.
Thus $\phi^{-1}(\redS(V,\Psi)_{\leq S})*\redS(V,\Psi)_{>S}$ is $m_S + m_{S^\perp}  = (m_S-1)+(m_{S^\perp}-1)+2$ connected.
Now, note that the $m - 2 = m_S + m_{S^\perp}$.
By Quillen's fiber-Theorem \ref{quillensFiberTheoremGeneral} we conclude that $\phi$ is an $(m-1)$-equivalence.
Since $X$ is $(m-1)$-connected, $\redS(V,\Psi)$ is $(m-1)$-connected and thus spherical.
This concludes the proof of the proposition.
\end{proof}

From now on, the form $\Psi$ will be determined from the context, so we can just write $\S(V)$ for $\S(V,\Psi)$ and $\F(V)$ for $\F(V,\Psi)$.

\subsection{Unitary spaces.}
In this section we work with Hermitian forms $\Psi$ over a finite-dimensional vector space $V$, i.e., $\epsilon = 1$ and $\tau$ is an automorphism of $\KK$
of order $2$. We call a vector space equipped with a non-degenerate Hermitian form $\Psi$ a 
\textit{unitary space}.

For a given field $\KK$ the automorphisms $\tau$ of order $2$ are in bijection with subfields $\FF$ of $\KK$ such that $\KK\slash \FF$ is a separable field extension of degree $2$. Note that $\tau$ is not determined by $\KK$ in general. 
Nevertheless, this is the case for important examples.
A finite field $ \KK$ has an automorphism $\tau$ of order $2$ if and only if $|\KK|$ is a square. In that case, we will usually write $\KK = \GF{q^2}$, where $|\KK|=q^2$, and $\tau$ is uniquely determined: it is given by $\tau(x) = x^q$.
If $\KK$ is the field of complex numbers then the unique $\tau$ of order $2$ is the complex conjugation.
It follows that finite unitary spaces are unique up to isometry, and complex unitary spaces are determined by their signature (see \cite[(21.6)]{AscFGT}).

Sometimes we will work with the following condition:
$$\text{(ON) ~~ For all } v \in V \text{ there is } x \in \KK \text{ such that } \Psi(v,v) = x \tau(x).$$

Under this assumption, $V$ has an orthonormal basis and thus the unitary structure on $V$ is uniquely determined by $\tau$ (see (19.9) and (19.11) in \cite{AscFGT}). For example, (ON) holds for unitary spaces over finite fields.
If $\KK = \CC$ then (ON) is satisfied if and only if $\Psi$ is equivalent to the standard scalar product.

When (ON) holds, after fixing the dimension $n$, the field $\KK$ and the
automorphism $\tau$, every unitary space of dimension $n$ over $\KK$ is isometric to the unitary space $\KK^n$ equipped with the canonical non-degenerate Hermitian form
\begin{equation*}
\Psi(v,w) := \sum_{i=1}^n v_i \,\tau(w_i).
\end{equation*}

For a unitary space $V$, we denote by $\GU(V)$ the isometry group of $V$. The following discussion summarizes a few simple facts and consequences of Witt's extension theorem (see \cite[p.81]{AscFGT}).

\begin{remark}
\label{rem:elemproperties}
Let $V$ be an $n$-dimensional vector space equipped with a (possibly degenerate) Hermitian form $\Psi$.
Then:
\begin{enumerate}[label=(\roman*)]
    \item The minimal elements
    of $\mathring{\S}(V)$ have dimension $1$.
    \item $\dim(\F(V)) = 
    n-\dim(\Rad(V))-1$. In particular, if $V$ is non-degenerate then $\dim (\F(V)) = n-1$ and $\dim(\widehat{\F}(V)) = n-2$.
\end{enumerate}
If in addition (ON) holds, then
\begin{enumerate}[label=(\roman*)]
    \item The group $\GU(V)$ acts flag-transitively on $\S(V)$.
    \item $\GU(V)$ acts transitively on the set of frames of a fixed size.
    \item If $S\in \G(V)$ is a vertex, $\Stab_{\GU(V)}(S)$ is transitive on the set of vertices adjacent to $S$.
\end{enumerate}
\end{remark}

We saw that the frame complex of a vector space with a non-degenerate form collapses to a subcomplex of one dimension less.
For unitary spaces over $\KK = \GF{2^2}$, we can collapse one extra dimension.
Due to this particular property, most of our results on finite unitary spaces will need a special treatment for the case $\KK = \GF{2^2}$ and they usually have different answers.
Let $\doublehat{\F}(V)$ be the poset of non-empty frames of size $\neq \dim(V)-1,\dim(V)-2$.

\begin{lemma}
\label{doubleJumpDimensionInField2}
Let $V$ be a unitary space of dimension $n\geq 3$ over $\GF{2^2}$.
Then $\doublehat{\F}(V)$ has dimension $n-3$ and it is homotopy equivalent to $\F(V)$. In particular,
for $n=3$ we have that
$\F(V)$ and $\G(V)$ are disconnected with $4$ connected components corresponding to each $3$-frame.
\end{lemma}

\begin{proof}
From Lemma \ref{lemmaHatPoset} we know that $\F(V)$ is homotopy equivalent to $\widehat{\F}(V)$.
If $\sigma \in \widehat{\F}(V) - \doublehat{\F}(V)$, then $\sigma$ is an $(n-2)$-frame.
The elements above $\sigma$ are the maximal simplices (i.e., frames) containing $\sigma$.
Hence they arise from $\sigma$
by adding a frame of
$\gen{\sigma}^\perp$. Now $\gen{\sigma}^\perp$ is a 
unitary space of dimension $n - (n-2) = 2$ over $\GF{2^2}$.
It is easily checked that there is a unique frame (see e.g., Equation \eqref{eqNumberMFrames}) for $\gen{\sigma}^\perp$. 
Thus $\sigma$ is contained in a unique maximal frame.
Again, using \cite[Corollary 10.12]{Bjo95}, this implies that  $\widehat{\F}(V)$ and $\doublehat{\F}(V)$ are 
homotopy equivalent.

For $n=3$ the poset $\doublehat{\F}(V)$ is a set of $4$ points corresponding to the $4$ 
$3$-frames of $V$ (see Appendix \ref{appendixCountingVectorsAndSubspaces})
and hence disconnected. The first part
of the claim implies that $\F(V)$ is disconnected and hence $\G(V)$ is disconnected.
\end{proof}

We close this section with a reduction of the connectedness of $\F(V,\Psi)$ for a general Hermitian form $\Psi$ to the unitary case.

\begin{proposition}
\label{cor:degenerateConnected}
Let $V$ be a vector space of dimension $n$ and $\Psi$ a Hermitian form on $V$ such that $\dim(V / \Rad(V)) \geq 2$.
Then $\F(V)$ is connected if and only if $\F(V/\Rad(V))$ is connected.
\end{proposition}

\begin{proof}
By Proposition \ref{dimensionTheorem}, we can write $V = \Rad(V) \oplus V'$ where $V'$ is non-degenerate.
Then $V/\Rad(V)$ is a unitary space isomorphic to $V'$.
From now on, we identify $V/\Rad(V)$ with $V'$.

The projection map $p:V\to V'$ extends to a surjective simplicial map $p_*:\F(V)\to \F(V')$.
In particular, if $\F(V)$ is connected then $\F(V')$ is connected.

Now suppose that $\F(V')$ is connected, which by Lemma \ref{lem:conn} is equivalent to $\G(V')$ being connected. 
Analogously, $\F(V)$ is connected if and only if $\G(V)$ is connected.
Since $\G(V')$ is a connected subgraph of $\G(V)$, it suffices to prove that every vertex of $\G(V)$ can be connected to a vertex of $\G(V')$.
Let $S\in \G(V)$ be a vertex, and write $S = \gen{x}$.
Then $x = r + v$ with $r\in\Rad(V)$ and $v\in V'$.
Since $S$ is non-degenerate, $0\neq |x| = |r| + |v| = |v|$.
Hence $v$ is a non-degenerate vector of $V'$.
On the other hand, $\dim(V')\geq 2$ by hypothesis, so $\gen{v}^\perp\cap V'$ is non-degenerate of dimension $\geq 1$.
Let $y\in\gen{v}^\perp\cap V'$ be a non-degenerate vector.
Then $S \perp \gen{y}$ and $\gen{y}\perp \gen{v}$, that is, $S$ and $\gen{v}\in\G(V')$ are connected in $\G(V)$.
This shows that every vertex $S\in \G(V)$ can be connected to a vertex in $\G(V')$.
Therefore $\G(V)$ is connected.
\end{proof}

\section{Counting walks in $\G(V)$
and connectedness of $\G(V)$}
\label{sec:walk}

In this section, we count the number of walks of length $\leq 4$ between two vertices in $\G(V)$. By a walk we understand a sequence of vertices
such that two consecutive vertices are joined by an edge.
In particular, the repetition of vertices and edges is allowed.
The length of a walk is the number of its edges.

Our count of walks will yield
a proof of Theorem \ref{connectedFramePoset} and allow us to clarify the homotopy type of $\F(V)$ for $V$ of dimension $\leq 3$ in Proposition \ref{lowDimensionalCases}. 
The calculations will also play a key role in Section \ref{sec:eigen} when the
eigenvalues of the adjacency matrix of $\G(V)$ are determined for $\KK= \GF{q^2}$.
For the explicit computation of the number of walks, we will work under condition (ON).

For the whole section, $V$ will be an $n$-dimensional unitary space over a field $\KK$.
We write $\Psi(\bullet,\bullet)$ for its non-degenerate Hermitian form and
$\tau$ for the underlying field
automorphism of order $2$.

We introduce first some useful counting functions, which may take the value $\infty$.
Let $S$ be a subspace of $V$.

\begin{itemize}
\item For $k \geq 0$, let $l_k(T,W)$ be the number of walks of length $k$ between vertices $T,W$ in $\G(V)$. 

\item For the set $\F(S^\perp)_1$ of $1$-dimensional non-degenerate subspaces of $S^\perp$, let $d(S) = \#\F(S^\perp)_1$.
In particular, if $S \in \G(V)$, $d(S)$ is the degree of $S$ in the graph $\G(V)$.

\item Let $\I(S)$ be the set of non-zero isotropic vectors in $S$ and set $I(S) = \# \I(S)$.
\end{itemize}

Any of the numbers above can be infinite. We will use the convention 
$\infty + \infty = \infty$, $\infty \cdot \infty = \infty$ and
$a \cdot \infty = \infty$ for any real $a > 0$.

\bigskip

If in addition (ON) holds and $S$ has dimension $m$ with radical of dimension $R$, then:
\begin{itemize}
\item We write instead $d_{m+1}^R := \# \F(S^\perp)_1$. When $S$ is non-degenerate (i.e., $R = 0$), we let $d_{m+1}:=d_{m+1}^0$.
The shift from dimension $m$ to index $m+1$ is motivated by the fact that $d_{m}^R$ is the degree of each vertex in $\G(S)$.
Hence $\G(S)$ is a $d_m^R$-regular graph.
Note that $d^1_1 = d^1_2= 0$.
\item Also let $I_m^R = \# \I(S)$. When $R = 0$, we just write $I_m$ for $I_m^0$.
\end{itemize}
By (ON), the values of $d^R_m$ and $I_m^R$ are independent of the subspace $S$ chosen in their definition and just depend on $m$ and $R$.
Moreover, the values of $d^R_m$ and $I_m^R$ for a finite field $\KK = \GF{q^ 2}^n$ can be found in the Appendix \ref{appendixCountingVectorsAndSubspaces}, where we also list some useful equalities.

It will turn out that for $S,W \in \G(V)$, the values of $l_k(S,W)$ will frequently depend on the following cases.

\begin{itemize}
\item[($=$)] $S=W$,
\item[($\perp$)] $S \perp W$,
\item[(ND)] $S\neq W$, $S\not\perp W$ and $S+W$ non-degenerate,
\item[(D)] $S+W$ degenerate.
\end{itemize}
We will prove this assertion when (ON) is satisfied.

\bigskip

Now, we proceed with the computation of walks of length $1$ and $2$.

\begin{lemma}
  \label{walksLength1and2}
  Let $V$ be a unitary space of dimension $n\geq 2$ over a field $\KK$.
  Let $S,W\in \G(V)$ be two vertices.
  Then the following hold:
  \begin{equation*}
    l_1(S,W) = 
    \begin{cases}
        1 & S \text{ and } W \text{ satisfy ($\perp$)},\\
        0 & \text{otherwise},
    \end{cases}
  \end{equation*}
  \begin{equation*}
    l_2(S,W) = \# \F( (S + W)^\perp )_1 = d(S+W).
  \end{equation*}
  In particular, $l_2(S,W) > 0$ if and only if $S=W$, or $n\geq 3$ and $S+W$ is non-degenerate, or $n\geq 4$ and $S+W$ is degenerate.

  Moreover, if (ON) holds, then
  \begin{equation*}
    l_2(S,W) = \begin{cases}
      d_n & S = W,\\
      d_{n-1} & S \text{ and } W \text{ satisfy ($\perp$) or (ND)},\\
      d^1_{n-1} & S \text{ and } W \text{ satisfy (D)}.\\
    \end{cases}
  \end{equation*}

\end{lemma}

\begin{proof}
  The value of $l_1(S,W)$ follows immediately from the definition of $\G(V)$.
  
  For $l_2(S,W)$, note that a vertex $T$ is adjacent to both $S$ and $W$ if and only if $T\in \F( (S + W)^\perp)_1$.
  Thus $l_2(S,W) = \#  \F( (S + W)^\perp)_1 =d(S+W)$.
  Now, the size of this latter set depends on the dimension of $S + W$ and if the sum is degenerate or not.
 In view of Proposition \ref{dimensionTheorem}(vii), 
  \begin{equation}
      \label{eq:positivel2}
      l_2(S,W) > 0 \text{ if and only if } \codim(S+W) > \dim(\Rad(S+W)).
  \end{equation}
  We will use this equivalence to show the In particular part of the lemma.
  \begin{itemize}
    \item[($=$)] If $S=W$ then $\dim(S + W) = 1$ and $l_2(S,W)$ is just the degree of $S=W$ in $\G(V)$.
    This number is positive for all $n\geq 2$ by Eq. (\ref{eq:positivel2}).
    
    \item[($\perp)$, (ND)] If $S+W$ is non-degenerate of dimension $2$, then $(S + W)^\perp$ is a unitary space of dimension $n-2$.
    By Eq. (\ref{eq:positivel2}), $l_2(S,W) > 0$ if and only if $n \geq 3$.

    \item[(D)] If $S+W$ is degenerate, $(S + W)^\perp$ is a vector space of dimension $n-2$ equipped with a Hermitian form whose radical has dimension $1$.
    By Eq. (\ref{eq:positivel2}), $l_2(S,W) > 0$ if and only if $n \geq 4$.
\end{itemize}
If $V$ verifies (ON), the values for $l_2(S,W)$ follow by definition.
This finishes the proof of the lemma.
\end{proof}

Next, we aim at counting walks of length $3$ and $4$.
For that, the following obvious recursive formula will be useful. 
We use the convention that a summand equal to $\infty$
yields the sum $\infty$.

\begin{equation}
\label{eqInductiveCounting}
l_k(S,W) = \sum_{T\in \G(S^\perp) } l_{k-1}(T,W).
\end{equation}

In order to apply the recursion, we need to know all $l_{k-1}(T,W)$ as $T$ runs over the $1$-dimensional non-degenerate subspaces of $V$ orthogonal to $S$.
Under condition (ON), we will show that for given $S, W \in \G(V)$,
the number of $T$'s such that the relative position of
$W$ and $T$ is described by a fixed case from ($=$), ($\perp$), (D) or (ND)
only depends on which case ($=$), ($\perp$), (D), (ND) describes the relative position of $S$ and $W$, 
and, that among each case, the number
$l_{k-1}(T,W)$ is fixed.

We start by investigating the dependence of the number of $T$'s
on the relative position of $S$ and $W$.
We introduce the following notation:

\begin{itemize}
    \item $E_0(S,W) = \{T \in \G(S^\perp) \tq T + W \text{ non-degenerate}\}$,
    \item $E_1(S,W) = \{T \in \G(S^\perp) \tq T\perp W \}$,
    \item $E_2(S,W) = \{T \in \G(S^\perp) \tq T\neq W,T\not\perp W, T + W \text{ non-degenerate}\}$,
    \item $E_3(S,W) = \{T \in \G(S^\perp) \tq T+ W \text{ degenerate}\}$.
\end{itemize}
For $i=0,1,2,3$ we set $\eta_i(S,W) = \# E_i(S,W)$. 
The following lemma considers a case where $E_3(S,W)$ is empty.

\begin{lemma} \label{lem:tau-1}
  Let $V$ be a unitary vector space of dimension $n\geq 2$ over a field $\KK$.
  Then $V$ is anisotropic if and only if there are no orthogonal vectors $v,w\in V$ such that $|v| |w| = -1$.
  In particular, in such cases, every subspace is non-degenerate.
  
  Moreover, if $V$ satisfies (ON), $V$ is anisotropic if and only if there is no $\zeta \in \KK$ such that $\zeta \tau(\zeta) =-1$.
\end{lemma}

\begin{proof}
Suppose that there exist orthogonal vectors $v,w\in V$ such that $|v| |w| = -1$.
Equivalently, $|v|^{-1} + |w| = 0$.
Let $x = \frac{v}{|v|} + w$.
Then $x\neq 0$ and $|x| = \frac{|v|}{|v|^2} + |w| = 0$, that is, $x$ is a non-zero isotropic vector, so $V$ is not anisotropic.

Conversely, assume that $x\in V$ is a non-zero isotropic vector.
Then we can pick two non-zero orthogonal vectors $u,w\in V$ such that $x = u + w$.
If $v = |u|^{-1}u$ then $v,w$ are two orthogonal vectors such that $|v||w| = -1$.
This proves the equivalence of the lemma.

The assertion on the non-degenerate subspaces is immediate from this.

If $V$ satisfies (ON), then $|v| |w| = x\tau(x) y\tau(y) = (xy)\tau(xy)$ for suitable elements $x,y\in \KK$, so we can take $\zeta = xy$ in the statement.
\end{proof}

In the following lemma we provide additional information on the $E_i(S,W)$, $i=0,1,2,3$, needed for the calculation of
the corresponding $\eta_i(S,W)$.

\begin{lemma}
\label{mainlemmaCountingDegAndNonDeg}
Let $V$ be a unitary space of dimension $n\geq 2$ over a field $\KK$ satisfying (ON), and let $S,W \in \G(V)$.
Then we have:

\begin{align}
    \nonumber E_0(S,W) &\cap E_3(S,W) = \emptyset, E_1(S,W) \cap E_3(S,W) = \emptyset \text{ and } E_1(S,W) \cap E_2(S,W) = \emptyset, \\
    \label{eq:set1} \F(S^\perp)_1 & = E_0(S,W) \cup E_3(S,W),\\
    \label{eq:set2} E_0(S,W) & = E_1(S,W) \cup E_2(S,W) \cup  \begin{cases}
    \{ W\} & S \perp W,\\
    \emptyset & S \not\perp W,
    \end{cases} \\
    \nonumber E_1(S,W) & = \F(S^ \perp)_1 \text{ if } S,W \text{ satisfy } (=),\\
    \nonumber E_1(S,W) & = \F((S+W)^ \perp)_1 \text{ if } S,W \text{ satisfy $(\perp)$, (ND) or (D)},\\
    \nonumber E_3(S,W) & = \emptyset \text{ if } S,W \text{ satisfy } (=), \\
    \nonumber E_3(S,W) & \cong \I((S+W)^\perp) \text{ if } S,W \text{ satisfy } (\perp), \\
    \nonumber E_3(S,W) & \cong \Big( \I(S^ \perp) \setminus \I((S+W)^ \perp\Big)/\KK  \text{ if } S,W \text{ satisfy (ND) or (D)}.
\end{align}

In particular (for $n\geq 3$) we have:
\begin{itemize}
    \item[(i)] $E_0(S,W)$ and $E_1(S,W)$ are non-empty except when $n=3$ and $S,W$ satisfy (D).
    \item[(ii)] $E_3(S,W)$ is 
empty if $\I(S^\perp)$ is empty.
\end{itemize}
\end{lemma}

\begin{proof}
The assertions about the intersections follow immediately from the definitions.
The same holds for \eqref{eq:set1} and \eqref{eq:set2}, 
the assertions on $E_i(S,W)$, $i=1,2,3$, in the case ($=$), and the 
identity $E_1(S,W) = \F( (S+W)^\perp)_1$ in cases ($\perp$), (ND) and (D). 

Since in \eqref{eq:set1} we have a disjoint union, it follows that
once $\F(S^\perp)_1$ is fixed, $E_3(S,W)$ and $E_0(S,W)$ determine each other.
Since we already know $E_1(S,W)$ again by
disjointness in \eqref{eq:set2}, we can calculate $E_2(S,W)$ from $E_0(S,W)$.
Thus given $E_1(S,W)$ and $\F(S^\perp)_1$ we have that $E_3(S,W)$ determines $E_2(S,W)$. 

Thus it remains to determine $E_3(S,W)$ in case ($\perp$), (D) and (ND). 
Also $E_3(S,W) = \emptyset$ if $V$ is anisotropic.
Thus, by Lemma \ref{lem:tau-1}, we can assume there is $\zeta \in \KK$
such that $\zeta \tau(\zeta) = -1$.

From now on, since we are working under (ON) and to simplify the notation, we assume that $\Psi$ is the canonical non-degenerate Hermitian form of $V = \KK^n$.
By the transitivity of the action of $\GU(V)$ (see Remark \ref{rem:elemproperties}) we may also assume that $S = \gen{e_1}$, where $e_i$ denotes the $i$th vector of the canonical basis of $\KK^n$.

\medskip

\noindent\textbf{Case 1.} If $S \perp W$ then $E_3(S,W)$ is in bijection with the set $\I((S+W)^\perp)$.

\begin{proofofcase}[Proof of Case 1.]

Write $W = \gen{w}$.
Let $T\in E_3(S,W)$.
Since $T\not\perp W$ and $T\perp S$, there is a unique generator $t$ of $T$ such that $t = w + v'_T$, with $v'_T$ a non-zero vector orthogonal to $S+W$.
In particular, since $T+W$ is degenerate, $v'_T$ is isotropic.
Thus $T+W = \gen{w,v'_T}$ and its radical is exactly $\gen{v'_T}$.
This defines a bijective map $T\in E_3(S,W) \mapsto v'_T \in \I((S+W)^\perp)$,  with inverse given by $v' \mapsto \gen{w+v'}$.
\end{proofofcase}

\noindent\textbf{Case 2.} If $S, W$ satisfy (ND) or (D), then $E_3(S,W) \cong \left( \I(S^ \perp) \setminus \I((S+W)^ \perp) \right)/\KK$.

\begin{proofofcase}[Proof of Case 2.]
Assume $W = \gen{w}$ for $w = 
w_1e_1 + \tilde{w}$ and some 
$\tilde{w} \in S^\perp$. Since $S\neq W$ and $S\not\perp W$ we have $w_1\neq 0$ and $\tilde{w}\neq 0$.

Let $T\in \G(S^\perp)$. 
Observe that, if $T = \gen{v}$, then $T + W = \gen{v,w} = \gen{v, w_1e_1+\tilde{w}}$ is degenerate if and only if there exists a non-zero vector $u = \alpha v + \beta w$, $\alpha,\beta\in\KK$, such that:
\[\begin{cases}
0 = \Psi(u,v) = \alpha |v| + \beta \Psi(w,v),\\
0 = \Psi(u,w) = \alpha \Psi(v,w) + \beta |w|.
\end{cases}\]
Such non-zero vector $u$ exists if and only if the determinant of the underlying matrix is zero, that is,
\begin{equation}
\label{eq1}
|v||w| = \Psi(v,w)\tau(\Psi(v,w)) = \Psi(v,\tilde{w}) \tau(\Psi(v,\tilde{w})).
\end{equation}
Therefore, we need to determine the vectors $v$ which yield \eqref{eq1}.

By (ON), without loss of generality, we can assume $|w|=1$.
Then, using $\zeta\in\KK$ with $\zeta\tau(\zeta) = -1$, equation \eqref{eq1} takes the following equivalent form:
\begin{align}
\label{eq2}
0 & = (\tau(\zeta)  \Psi(v,\tilde{w}))\,\tau( \tau(\zeta)  \Psi(v,\tilde{w})) + |v| \\
 \nonumber & = (\Psi(v,\zeta\tilde{w})\,\tau(\Psi(v,\zeta\tilde{w})) + \sum_{i=2}^n v_i\tau(v_i).
\end{align}

We will show that the solutions 
$v$ for \eqref{eq2} satisfying $|v| \neq 0$ are in one-to-one correspondence with the solutions of the following system of equations.
For $x\in \KK^n$, write $x = x_1e_1 + \tilde{x}$, $\tilde{x}\in \{0\}\times \KK^{n-1}$.
With this notation, $v = \tilde{v}$.
Consider the equations:
\begin{equation}
\label{eq3}
\begin{cases}
0 = |x| = x_1\tau(x_1) + \sum_{i=2}^n x_i \tau(x_i) = x_1\tau(x_1) + |\tilde{x}|,\\
0 = \Psi(x, e_1 - \zeta\tilde{w} ) = x_1 - \Psi(\tilde{x},\zeta\tilde{w}) ,\\
|\tilde{x}| \neq 0.
\end{cases}
\end{equation}
Note that if $v$ satisfies \eqref{eq2} then we get a solution $x$ to the system \eqref{eq3} by defining $x_1=\Psi(v,\zeta\tilde{w})$ 
and $\tilde{x} = \tilde{v}$.
Conversely, if $x$ satisfies \eqref{eq3}, then the vector $v = \tilde{v} =\tilde{x}$ satisfies $|v|=|\tilde{v}|=|\tilde{x}|\neq 0$ and
\[ \Psi(v,\zeta\tilde{w})\tau(\Psi(v,\zeta\tilde{w})) + |v| = \Psi(\tilde{x},\zeta\tilde{w})\tau(\Psi(\tilde{x},\zeta\tilde{w})) + |\tilde{x}| = x_1\tau(x_1) + |\tilde{x}| = 0.\]
This implies that \eqref{eq2} holds.

It remains to determine the solutions to the system  \eqref{eq3}.
First, we determine the non-zero solutions of the first two equations
allowing the possibility $|\tilde{x}|=0$.
These are the isotropic vectors 
in the orthogonal complement of $e_1 - \zeta\tilde{w}$.
Therefore, the solutions are given by $\I(\gen{e_1 - \zeta\tilde{w}}^\perp)$.
Moreover, from
\[ |e_1 - \zeta\tilde{w}| = 1 - |\tilde{w}| = 1 - ( |w| - w_1\tau(w_1)) = 1 - ( 1 - w_1\tau(w_1))=w_1\tau(w_1)\neq 0,\]
we see that $e_1 - \zeta\tilde{w}$ is a non-degenerate vector.

Next, we determine non-zero solutions with $|\tilde{x}|= 0$.
Note that, since $|x|=0$, we have that:
\begin{align*}
 |\tilde{x}| = 0 \, \Leftrightarrow \, x_1\tau(x_1) = 0 \, \Leftrightarrow \, x_1 = 0 \, \Leftrightarrow \, \Psi(\tilde{x},\tilde{w}) = 0.
\end{align*}
Then the solutions $x$ with $|\tilde{x}| = 0$ are given by the isotropic vectors lying in the orthogonal complement of $\tilde{w}\neq 0$ in the non-degenerate space $S^\perp = \{0\} \times \KK^ {n-1}$. This is the set $\I(\gen{e_1,\tilde{w}}^\perp) = \I((S+W)^\perp)$.
A unitary transformation mapping $\gen{e_1 - \zeta\tilde{w}}$ to $\gen{e_1}$ and $\gen{e_1,\tilde{w}}$ to itself establishes then a bijection between $E_3(S,W)$ and the set $\left( \I(S^ \perp) \setminus \I((S+W)^ \perp) \right)/\KK$.
This finishes the proof of Case 2.
\end{proofofcase}

Assertion (ii) is a simple consequence of the definition.
We briefly indicate how to deduce (i).
If $E_0(S,W)$ is empty then so is $E_1(S,W)$.
Moreover, $E_1(S,W)$ is empty if and only if $(S+W)^\perp$ is totally isotropic.
Since $n\geq 3$, this can only happen if $n=3$ and $S+W$ is degenerate.
\end{proof}

\begin{remark}
\label{rk:characterisationE2}
Let $V$ be a unitary space of dimension $n\geq 2$ over $\KK$ and let $S,W \in \G(V)$.
Suppose in addition that $S\not\perp W$.
Write $S = \gen{s}$ and $W = \gen{s+\tilde{w}}$ for some $\tilde{w}\perp s$.
Then we have the following characterization of $E_2(S,W)$: a $1$-dimensional subspace $\gen{v}$ belongs to $E_2(S,W)$ if and only if the following conditions hold:
\begin{enumerate}
    \item (Non-isotropic) $|v|\neq 0$,
    \item ($\perp S$) $\Psi(v,s) = 0$,
    \item ($\not\perp W$) $\Psi(v,\tilde{w})\neq 0$,
    \item (ND sum) $\Psi(v,\tilde{w})\Psi(\tilde{w},v)\neq |v||s+\tilde{w}|$.
\end{enumerate}
To see the equivalence, note that the first three conditions are equivalent to $\gen{v}\in \G(S^\perp)$ with $\gen{v}\not\perp W$.
It remains to show that $\gen{v}+W$ is non-degenerate if and only if the fourth condition (ND sum) holds.
But this is exactly the negation of the condition given in equation (\ref{eq1}) for the sum to be degenerate (note that (ON) is not needed for this equivalence).
\end{remark}

The following corollary restates the results from Lemma \ref{mainlemmaCountingDegAndNonDeg} in enumerative terms 
for the case that $\KK$ is a finite field.

\begin{corollary}
\label{lemmaCountingDegAndNonDeg}
Let $S,W\in \G(V)$, where $V$ is a unitary space of dimension $n\geq 3$ over a field $\KK$ satisfying (ON).
Then we have the equalities:
\begin{align*}
    d_n & = \eta_0(S,W)+ \eta_3(S,W),\\
    \eta_0(S,W) & = \eta_1(S,W)+\eta_2(S,W) + \begin{cases}
    1 & S \perp W,\\
    0 & S \not\perp W.
    \end{cases}
\end{align*}

Moreover, if $\KK = \GF{q^2}$ is a finite field, Table \ref{tab:etaValues} records the values of $\eta_1(S,W)$, $\eta_2(S,W)$ and $\eta_3(S,W)$ and determines $\eta_0(S,W)$.

\begin{table}[ht]
\renewcommand{\arraystretch}{1.8}
    \centering
    \begin{tabular}{|c|c|c|c|}
    \hline
        & $\eta_1(S,W)$ & $\eta_2(S,W)$ & $\eta_3(S,W)$ \\
    \hline
        $S,W$ satisfy ($=$) & $d_n$ & $0$ & $0$ \\
        \hline
        $S, W$ satisfy ($\perp$) & $d_{n-1}$ & $d_n - d_{n-1}-I_{n-2}-1$ & $I_{n-2}$\\
        \hline
        $S,W$ satisfy (ND) & $d_{n-1}$ & \parbox{6cm}{\centering $d_n - d_{n-1}(q+2) = $\\
        $q^{2n-4} - 2 q^{2n-5} + 2q^{n-3}(-1)^n$} & $d_{n-1}(q+1)$ \\
        \hline
        $S,W$ satisfy (D)& $d^1_{n-1}$ & $q^{2n-4}-2q^{2n-5}$ & $q^{2n-5}$ \\
    \hline
    \end{tabular}
    
    \medskip
    
    \caption{Values of $\eta_1,\eta_2,\eta_3$ for a fixed pair $S,W\in \G(V)$.}
    \label{tab:etaValues}
\end{table}

\end{corollary}
\begin{proof}
The equalities for $d_n$ and $\eta_0$ follow immediately from the definitions and Lemma \ref{mainlemmaCountingDegAndNonDeg}.

Suppose then that $\KK = \GF{q^2}$ is a finite field.
Except for the entries for $\eta_2(S,W)$ and $\eta_3(S,W)$ in cases (ND) and (D),
  the entries in Table \ref{tab:etaValues} follow from the definitions and Lemma \ref{mainlemmaCountingDegAndNonDeg}.
  In cases (ND) and (D), the size of the set $E_3(S,W)$ can be computed by using Lemma \ref{mainlemmaCountingDegAndNonDeg}.
 That is, $\eta_3(S,W)$ equals the cardinality of $\left(\I(S^ \perp) \setminus \I((S+W)^ \perp)\right)/ \KK$, which depends on whenever $S+W$ is degenerate or not.
  We get
  \begin{align*}
    \# \Big(\left(\I(S^ \perp) \setminus \I((S+W)^ \perp)\right)/ \KK\Big) & = 
\begin{cases}
(I_{n-1} - I^1_{n-2})\frac{1}{q^{2}-1} & S+ W \text{ degenerate},\\
(I_{n-1} - I_{n-2})\frac{1}{q^{2}-1} & S+ W \text{ non-degenerate}.
\end{cases}
\end{align*}
Now, the values for $\eta_2(S,W)$ and $\eta_3(S,W)$ in rows (ND) and (D) of Table \ref{tab:etaValues} follow from the formulas provided in Appendix \ref{appendixCountingVectorsAndSubspaces}.
\end{proof}

Next, we show that the numbers $l_{k-1}(T,W)$ in \eqref{eqInductiveCounting} only depend on which of the cases
($=$), ($\perp$), (ND) or (D) holds for $T$ and $W$. 

\begin{lemma} \label{lem:welldefined}
  Let $V$ be a unitary space of dimension $n\geq 2$ over $\KK$ satisfying (ON). 
  Let $T,W,T',W' \in \G(V)$. Then for $k \geq 0$ we have:
  \begin{itemize}
      \item[($=$)] $l_{k}(T,T) = l_{k}(T',T')$.
      \item[($\perp$)] If $T \perp W$ and $T' \perp W'$ then 
      $l_{k}(T,W) = l_{k}(T',W')$.
      \item[(ND)] If both $T,W$ and $T',W'$ satisfy (ND) then $l_{k}(T,W) = l_{k}(T',W')$.
      \item[(D)] If $T+W$ and $T'+W'$ are degenerate then 
      $l_k(T,W) = l_k(T',W')$.
  \end{itemize}
\end{lemma}
\begin{proof}
  Since by Remark \ref{rem:elemproperties}(iii) $\GU(V)$ acts transitively on the vertices of $\G(V)$ preserving orthogonality, the cases ($=$) and ($\perp$) are obvious.

  Note that Lemma \ref{walksLength1and2} shows that 
  $l_{k}(T,W) = l_{k}(T',W')$ for 
  $k=0,1,2$.
  If $k\geq 3$, by  \eqref{eqInductiveCounting} and induction, we see that
\begin{align}
    \label{eqInductivewalksAllEtas}
     l_{k}(T,W) = \,& l_{k-1}(W,W)\,\big(\#(E_0(T,W)\setminus( E_1(T,W) \cup E_2(T,W))\big) + \\
     \nonumber & l_{k-1}\,\big(T_{\perp},W) \,\eta_1(T,W\big) + \\ \nonumber & l_{k-1}(T_{ND},W) \, \eta_2(T,W) + l_{k-1}(T_D,W) \, \eta_3(T,W),
\end{align}
for some $T_\perp, T_{ND}, T_D$ in  
$\G(T^\perp)$ such that $T_\perp$ and $W$ satisfy ($\perp$),
$T_{ND}$ and $W$ satisfy (ND), and $T_D$ and $W$ satisfy (D).

Assume that we are in case (D).
Then there exist generators $t,t+\tilde{w},t',t'+\tilde{w}'$ of $T,W,T',W'$ respectively such that $|t| = 1 = |t'|$, and $\tilde{w}$ and $\tilde{w}'$ are non-zero isotropic vectors orthogonal to $t$ and $t'$ respectively; thus we can define a map $t\mapsto t'$ and $\tilde{w}\mapsto \tilde{w}'$ which is an isometry from $T+W$ onto $T'+W'$ and extends (by Witt's lemma) to an isometry of $V$.
Hence the cardinalities of the sets $E_j(T,W)$ do not depend on the choice of $T,W$ such that $T+W$ is degenerate.

Finally, suppose we are in case (ND).
Here it is not always true that we can find a unitary transformation mapping $T$ onto $T'$ and $W$ onto $W'$.
However, we show that the cardinalities of the sets $E_j(T,W)$ and $E_j(T',W')$ coincide for all $j$.

Since we are in case (ND), we can fix generators $t,a\cdot t+\tilde{w},t',a'\cdot t'+\tilde{w}'$ of $T,W,T',W'$ respectively such that $|t| = 1 = |t'|$, $a,a'\in\KK^\times$, and $|\tilde{w}|=1=|\tilde{w}'|$ with $t\perp \tilde{w}$ and $t'\perp \tilde{w}'$.
Let $\varphi$ be a unitary transformation of $V$ such that $\varphi(t) = t'$ and $\varphi(\tilde{w}) = \tilde{w}'$.
Although in general $\varphi(W)\neq \varphi(W')$, we have that $\varphi(E_2(T,W)) = E_2(T',W')$ by using the characterization of Remark \ref{rk:characterisationE2}.
That is, since $\varphi$ is an isometry, for $v\in V$ we have
\begin{enumerate}
    \item $|v| = |\varphi(v)|$,
    \item $\Psi(v,t) = \Psi(\varphi(v),t')$,
    \item $\Psi(v,\tilde{w}) = \Psi(\varphi(v),\tilde{w}')$,
    \item $\Psi(v,\tilde{w})\Psi(\tilde{w},v)-|v| = \Psi(\varphi(v),\tilde{w}')\Psi(\tilde{w}',\varphi(v))-|\varphi(v)|$.
\end{enumerate}
Hence $\gen{v}\in E_2(T,W)$ if and only if $\gen{\varphi(v)}\in E_2(T',W')$.
This shows that $\varphi(E_2(T,W)) = E_2(T',W')$.

Note also that $\varphi(T+W)=T'+W'$, so $\varphi(E_1(T,W)) = E_1(T',W')$.
This allows us to conclude that $\varphi(E_0(T,W)) = E_0(T',W')$.
Since in addition $\varphi(T) = T'$, we get $\varphi(\F(T)_1) = \F(T')_1$.
By definition of the sets $E_j(\bullet,\bullet)$ and Lemma \ref{mainlemmaCountingDegAndNonDeg}, we conclude that $\varphi$ is a bijection between $E_j(T,W)$ and $E_j(T',W')$ for $j=0,1,2,3$.

\end{proof}

The preceding Lemma \ref{lem:welldefined} shows that the following notation is well-defined:

\begin{definition}
Let $V$ be a unitary space of dimension $n\geq 2$ over $\KK$ satisfying (ON).
We define the following symbols for $k\geq 0$.
\begin{itemize}
    \item $l_k(=) := l_k(S,S)$, for $S\in \G(V)$.
    \item $l_k(\perp) := l_k(S,W)$, for $S,W\in\G(V)$ and $S \perp W$.
    \item $l_k(\text{ND}) := l_k(S,W)$, where $S,W\in\G(V)$, $S\neq W$, $S\not\perp W$ and $S+W$ is non-degenerate.
    \item $l_k(\text{D}) := l_k(S,W)$, where $S,W\in\G(V)$ and $S+W$ is degenerate.
\end{itemize}
If no choice of $S$ and $W$ is possible, then we define the corresponding symbol to be $0$.
\end{definition}

Using the relations between the $E_j(S,W)$ from Lemma \ref{mainlemmaCountingDegAndNonDeg} we can bring 
\eqref{eqInductivewalksAllEtas} from the proof of Lemma \ref{lem:welldefined}
into the following form, which we will use in the calculation of
walks of length $3$ and $4$.

\begin{equation}
    \label{eqInductivewalksEta}
     l_k(S,W) = l_{k-1}(=) \, l_1(S,W) + l_{k-1}(\perp)\, \eta_1(S,W) + l_{k-1}(\text{ND}) \,\eta_2(S,W) + l_{k-1}(\text{D}) \,\eta_3(S,W).
\end{equation}

Before we proceed with the computation of the number of walks of length $3$ and $4$ in $\G(V)$, we need to exclude
a certain situation for $\KK = \GF{2^2}$.

\begin{lemma}
\label{remarkCaseqEquals2}
Let $V$ be a unitary space of dimension $n\geq 2$ over $\GF{2^ 2}$.
If $S,W\in\G(V)$ are such that $S + W$ is non-degenerate, then either $S=W$ or $S\perp W$.
\end{lemma}
\begin{proof}
We can assume that $V = \GF{2^2}^n$ with the canonical unitary structure, $S = \gen{e_1}$, $W = \gen{w}$ with $w= w_1e_1 + \tilde{w}$, $\tilde{w}\in \{0\} \times \GF{2^2}^{n-1}$, $|w|=1$, and $S\neq W$.
Then $\tilde{w}\neq 0$.
Now, if $S+ W=\gen{e_1,\tilde{w}}$ is non-degenerate, we have that $|\tilde{w}|\neq 0$.
But $|\tilde{w}|\in\GF{2}$ implies $|\tilde{w}|=1$.
Then $1 = |w| = w_1^{3} + |\tilde{w}|$ forces $w_1=0$ and $W\perp S$.
\end{proof}

Now we determine $l_3(\bullet)$ and $l_4(\bullet)$.

\begin{theorem}
\label{theoremwalks3Dim3orMore}
Let $V$ be a unitary space of dimension $n\geq 3$ over $\KK$.
For any $S,W \in \G(V)$ the following are equivalent:
\begin{itemize}
    \item[(i)] $l_3(S,W) > 0$,
    \item[(ii)] $(n,\KK) \neq (3,\GF{2^2})$ or $S+W$ is non-degenerate.
    \end{itemize}
For $\KK = \GF{q^2}$ the values of $l_3(\bullet)$ are given in Table \ref{tab:walks3}.

\begin{table}[ht]
    \centering
    \begin{tabular}{|c|c|}
        \hline
         & $l_3(\bullet)$\\
        \hline
        ($=$) & $d_nd_{n-1}$ \\
        ($\perp$) & $d_n d_{n-1} + q^{3n-8}(-1)^n -q^{2n-6} + q^{2n-4}$\\
        (ND), $q\neq 2$ & $d_n d_{n-1}  + q^{3n-8}(-1)^n - q^{2n-6}$\\
        (D) & $d_n d_{n-1}  + q^{3n-8}(-1)^n$\\ 
         \hline
    \end{tabular}
    
    \medskip
    
    \smallskip
    
    \caption{Walks of length $3$ for $\KK = \GF{q^2}$.}
    \label{tab:walks3}
\end{table}
\end{theorem}

\begin{proof}
First, suppose that $(n,\KK) = (3,\GF{2^2})$ and $S+W$ is degenerate.
By Lemma \ref{doubleJumpDimensionInField2} it follows that $l_3(S,W) = 0$.
This shows that if (ii) fails, then (i) fails.

Now assume that (ii) holds.
For $T\in \G(S^\perp)$, we have $l_3(S,W) \geq l_2(T,W)$.
By Lemma \ref{walksLength1and2}, $l_2(T,W)>0$ unless $n = 3$ and $T+W$ is degenerate.
Therefore, we assume that $n = 3$ and show that $l_3(S,W) > 0$ by exhibiting $T\in \G(S^\perp)$ such that $T+W$ is non-degenerate.
We prove that this is always possible unless $\KK = \GF{2^2}$ and $S+W$ is degenerate, which is excluded by hypothesis (ii).

Write $S = \gen{s}$.
If $S = W$ or $S\perp W$, then clearly such $T$ exists.
Also, if $S+W$ is non-degenerate, we can take $T\in \G( (S+W)^\perp)$ by Lemma \ref{walksLength1and2}, and for such $T$ the sum $T+W$ is non-degenerate.
Hence we suppose that $S+W$ is degenerate.
This means that we can take $w\in W$ such that $w = s + \tilde{w}$ and $\tilde{w} \in S^\perp$ is a non-zero isotropic vector.
By Remark \ref{rk:characterisationE2}, we need a non-isotropic vector $v\in V$ such that $\Psi(v,s) = 0$ and
\begin{equation}
\label{eq:ND_dim3_proof}
|v| |w| \neq \Psi(v,\tilde{w})\Psi(\tilde{w},v).
\end{equation}
From now on, we work in $S^\perp$.
Since $S^\perp$ is a $2$-dimensional unitary space, we can take a basis of non-isotropic orthogonal vectors $v_1,v_2$, and write $a := |v_1| \neq 0$ and $b := |v_2| \neq 0$.
Also let $u_1,u_2\in \KK$ be such that $\tilde{w} = u_1 v_1 + u_2v_2$ (note that $u_1,u_2\neq 0$ since $\tilde{w}$ is non-zero isotropic).
Write $x^*=x\tau(x)$ for $x\in \KK$.
Then $|xv_1+yv_2| = x^*a+y^*b$ for all $x,y\in \KK$.

We prove that there exists $v = xv_1+yv_2\in S^\perp$ such that:
\begin{enumerate}
    \item $x^*a+y^*b |v| \neq 0$,
    \item $1 = \Psi(v,\tilde{w}) = x\tau(u_1)a + y\tau(u_2)b$,
    \item $|v| |w| \neq 1$.
\end{enumerate}
Suppose that $x = 0$, and let $y\neq 0$ be such that $\Psi(v,\tilde{w}) =1$ and $y^*b|v|\neq 0$.
Thus (1) and (2) above hold and
\[ 1 = \Psi(v,\tilde{w}) = y\tau(u_2)b \] 
implies $y = \tau(u_2^{-1})b^{-1}$.
If (3) fails then
\[ 1 =   |v||w| = y^*b |w| = \tau(u_2^{-1})^*b^{-1} |w|.\]
Since $\tau(u_2)^* = u_2^*$, this vector $v$ is not a solution provided that
\[ |w| = u_2^* b.\]
If we proceed analogously with $y = 0$ and $x\neq 0$, we conclude that
\[ |w| = u_1^* a.\]
Therefore, assuming that (3) fails for the two choices $x = 0,y\neq 0$ and $x\neq 0,y = 0$ such that (1) and (2) hold, we conclude that
\[ 0 = |\tilde{w}| = u_1^*a + u_2^* b = 2 |w|.\]
Since $|w|\neq 0$, we see that $\KK$ has characteristic $2$ and $u_1^*a=u_2^*b$, that is, $a = (u_2 u_1^{-1})^*b$.

Now, suppose that any arbitrary non-isotropic vector $v = xv_1+yv_2$ with $\Psi(v,\tilde{w}) = 1$ also has $|v||w| = 1$.
Then
\[1 = \Psi(v,\tilde{w}) = x\tau(u_1)a+y\tau(u_2)b\]
implies that
\begin{align*}
   x & = (1-y\tau(u_2)b)\tau(u_1^{-1})a^{-1}\\
   & = (1-y\tau(u_2)b)\tau(u_1^{-1})(u_2^{-1} u_1)^*b^{-1}\\
   & = ( b^{-1}-y\tau(u_2))(u_2^{-1})^* u_1
\end{align*}
In particular,
\[ x^*  = (1-y\tau(u_2)b-\tau(y)u_2b+y^* u_2^*b^2) (u_1^{-1})^* a^{-2}.\]
On the other hand,
\begin{align*}
    |v| & = x^* a + y^*b\\
    & = (1-y\tau(u_2)b-\tau(y)u_2b+y^* u_2^*b^2) (u_1^{-1})^* a^{-1} + y^*b\\
    & = (1-y\tau(u_2)b-\tau(y)u_2b+y^* u_2^*b^2) (u_1^{-1})^* (u_2^{-1}u_1)^*b^{-1} + y^*b\\
    & =  (b^{-1}-y\tau(u_2)-\tau(y)u_2+y^* u_2^*b) (u_2^{-1})^* + y^*b\\
    & =  (b^{-1}-y\tau(u_2)-\tau(y)u_2) (u_2^{-1})^*.
\end{align*}
Therefore, $|v|\neq 0$ if and only if
\[ b^{-1} \neq y\tau(u_2)+\tau(y)u_2.\]
Now, since we are assuming $|v||w| = 1$, we have:
\begin{align*}
1 & = |v||w|\\
  & = (b^{-1}-y\tau(u_2)-\tau(y)u_2) (u_2^{-1})^* u_2^*b\\
  & = 1 - y\tau(u_2) b - \tau(y)u_2 b,
\end{align*}
which holds if and only if
\begin{align*}
     \tau(y) = y\tau(u_2)u_2^{-1}.
\end{align*}
Since by hypothesis $\KK\neq \GF{2^2}$, we have that $|\Fix(\tau)|\geq 3$.
Then there exists $y\in \Fix(\tau)$ such that $y\neq 0$ and
\[ b^{-1} \neq y\tau(u_2) + \tau(y)u_2 = y( \tau(u_2) + u_2)\]
since $|v| = 0$ has at most one solution in $\Fix(\tau)$.
For this element $y$, $|v|\neq 0$, and this implies that $1 = |v||w|$, or equivalently, $y \tau(u_2) = \tau(y)u_2 = y u_2$.
Thus $\tau(u_2) = u_2$.
Therefore, if $y\in \KK \setminus \Fix(\tau)$ and $|v|\neq 0$, $1 = |v||w|$, that is $y = \tau(y)$, which is a contradiction.
Hence, for all $y\in \KK\setminus \Fix(\tau)$ we have $|v| = 0$, that is
\[ b^{-1} = (y+\tau(y))u_2.\]
But then this means that
\[ \tau(y) = y + (u_2b)^{-1}\]
for all $y\in \KK \setminus \Fix(\tau)$.

Let $c := (u_2b)^{-1}$, and note that $c\in \Fix(\tau)$.
Hence $y+c\notin \Fix(\tau)$, which implies that
\[\tau(y^2) + c^2 = \tau(y^2+c^2) = \tau((y+c)^2) = \tau(y+c)^2 = (y+c+c)^2 = y^2.\]
Since $c^2\neq 0$, $\tau(y^2)\neq y^2$, so $\tau(y^2) = y^2+c$.
Then the previous equation implies that $c+c^2 = 0$, so $c = 1$ and $\tau(y) = y + 1$.
Now, by applying $\tau$ to $y^3$ we reach a contradiction:
\begin{itemize}
    \item If $y^3\notin \Fix(\tau)$, then
    \[y^3+1 = \tau(y^3) = \tau(y^2)\tau(y)= (y^2+1)(y+1) = y^3+y^2+y+1.\]
    This implies that $y^2+y=0$, whose only solutions are $y=0,1 \in \Fix(\tau)$.
    
    \item Therefore $y^3\in \Fix(\tau)$ for all $y\in\KK \setminus\Fix(\tau)$.
    In particular,
    \[ y^3 = \tau(y^2)\tau(y)  = y^3+y^2+y+1.\]
    Then $y^2+y+1 = 0$, which has at most $2$ solutions.
    Thus $|\KK\setminus\Fix(\tau)|\leq 2$, a contradiction.    
\end{itemize}

The final contradiction arose from the assumption that no non-isotropic vector $v\in S^\perp$ with $\Psi(v,\tilde{w}) = 1$ and $1\neq |v||w|$ exists.
Therefore such $v$ exists, and this implies that there exists $T\in \G(S^\perp)$ such that $T+W$ is non-degenerate.
Hence $l_3(S,W)\geq l_2(T,W) > 0$.
This completes the proof that (ii) implies (i).

Suppose now that $\KK = \GF{q^2}$ is a finite field. According to \eqref{eqInductivewalksEta} and Lemma \ref{walksLength1and2}, for $S,W\in \G(V)$ we have that
\begin{align}
\label{eq:rec}
    l_3(S,W) & = l_2(=)\cdot l_1(S,W) + l_{2}(\perp) \cdot \eta_1(S,W) + l_{2}(\text{ND}) \cdot \eta_2(S,W) + l_{2}(\text{D}) \cdot \eta_3(S,W)\\ \nonumber
    & = d_n \cdot l_1(S,W) + d_{n-1}\cdot (\eta_1(S,W)+\eta_2(S,W)) + d^1_{n-1}\cdot \eta_3(S,W).
\end{align}
Since the values of $\eta_j(S,W)$ only depend on the relative position of $S$ and $W$, so does the value of $l_3(S,W)$.
Thus, Table \ref{tab:walks3} follows easily from the recursion formula \eqref{eq:rec}, Table \ref{tab:etaValues} and the formulas given in Appendix \ref{appendixCountingVectorsAndSubspaces}:

\begin{itemize}
    \item[($=$)] If $S=W$, then $l_3(=) = l_3(S,W) = d_n d_{n-1}$.
    \item[($\perp$)] If $S\perp W$ then $l_3(\perp) = l_3(S,W) = d_n d_{n-1} + (d^{1}_{n-1}-d_{n-1})I_{n-2} + d_n -d_{n-1}$.
    \item[(ND)] If $S\neq W$, $S\not\perp W$ and $S + W$ is non-degenerate, then
\[ l_3(\text{ND}) = l_3(S,W) = d_n d_{n-1} + d_{n-1}(d^1_{n-1}-d_{n-1})(q+1) = d_n d_{n-1} + q^{3n-8}(-1)^n  - q^{2n-6}.\]
    Note that this situation cannot arise if $q = 2$ by Lemma \ref{remarkCaseqEquals2}.
    \item[(D)] If $S + W$ is degenerate, then
    \[l_3(\text{D}) = l_3(S,W) = d_n d_{n-1} + q^{2n-5}(d^1_{n-1}-d_{n-1})  = d_n d_{n-1} + q^{3n-8}(-1)^n.\]
\end{itemize}
\end{proof}

Next, we turn to the computation of walks of length $4$ in a unitary space of dimension $n \geq 3$.

\begin{theorem}
\label{theoremwalks4}
Let $V$ be a unitary space of dimension $n\geq 3$ over the field $\KK$. 
For any $S,W \in \G(V)$ the following are equivalent:
\begin{itemize}
    \item[(i)] $l_4(S,W) > 0$,
    \item[(ii)] $(n,\KK) \neq (3,\GF{2^2})$ or $S+W$ is non-degenerate.
\end{itemize}
For $\KK = \GF{q^2}$ the values of $l_4(\bullet)$ are given in Table  \ref{tab:walks4}.
 
\begin{table}[ht]
    \centering
    \begin{tabular}{|c|c|}
        \hline
         & $l_4(\bullet)$\\
        \hline
        ($=$) & $d_n l_3(\perp)$ \\
        ($\perp$) & $(d_n+l_3(\perp))d_{n-1} + l_3(\text{ND})(d_n-d_{n-1}-I_{n-2}-1) + l_3(\text{D})I_{n-2}$\\
        (ND), $q\neq 2$ & $l_3(\perp)d_{n-1} + l_3(\text{ND})(d_n-d_{n-1}) + (l_3(\text{D}) - l_3(\text{ND})) d_{n-1}(q+1) $\\ 
        (D) & $l_3(\perp)d^1_{n-1} + l_3(\text{ND})(d_n - d^1_{n-1}) + (l_3(\text{D})-l_3(\text{ND}))q^{2n-5} $\\ 
         \hline
    \end{tabular}
    
    \medskip
    
    \smallskip
    
    \caption{Walks of length $4$ for $\KK = \GF{q^2}$.}
    \label{tab:walks4}
\end{table}

\end{theorem}
\begin{proof}
Let $T \in \G(S^\perp)$.
By Theorem \ref{theoremwalks3Dim3orMore} we have
$l_3(T,W) >0$ if either $(n,\KK) \neq (3,\GF{2^2})$ or $T+W$ is non-degenerate.
It follows that in all those cases $l_4(S,W) > 0$.
Therefore, we reduce to the case $(n,\KK) = (3,\GF{2^2})$ and $T+W$ is degenerate for all $T \in \G(S^\perp)$.
Since $l_4(S,W) \geq l_2(S,W)$, by Lemma \ref{walksLength1and2}, $S+W$ is degenerate.

Conversely, if $(n,\KK) = (3,\GF{2^2})$ and $S+W$ is degenerate, then $l_4(S,W) = 0$ by Lemma \ref{doubleJumpDimensionInField2}.
This proves that (i) and (ii) are equivalent.

For $\KK=\GF{q^2}$, the values for $l_4(\bullet)$ in Table \ref{tab:walks4} can be deduced from the recursive formula \eqref{eqInductivewalksEta}, analogous to the proof of Theorem \ref{theoremwalks3Dim3orMore}.
We leave the details for the reader.
\end{proof}

Finally, we are in position to provide a proof of Theorem \ref{connectedFramePoset}.

\begin{proof}[Proof of Theorem \ref{connectedFramePoset}]
  Lemma \ref{lem:conn} shows the equivalence between (i) and (ii).
  
  Next, we show the equivalence between (i) and (iii).
  To that aim, we split the proof in the cases $n = 2$ and $n \geq 3$.
  
  \medskip
  
  \noindent\textbf{Claim.} Suppose that $n = 2$.
  Then $\G(V)$ is connected if and only if $\KK = \GF{2^2}$.
  In that case, the diameter is $1$.
  
  \begin{proofofcase}[Proof of Claim.]
  By Lemma \ref{lemmaHatPoset}, $\F(V)$ and $\redF(V)$ 
  are homotopy equivalent. Thus $\redF(V)$ has one connected component for each $2$-frame.
  Thus, since $n = 2$, $\G(V)$ is connected if and only if there is a unique $2$-frame (in which case the diameter must be $1$).
  
  Suppose that $\G(V)$ is connected and that $S = \gen{x},T = \gen{y}$ constitute the unique $2$-frame of $V$.
  Let $a\in \KK^{\times}$.
  Note that $0 = |x+ay| = |x|+a\tau(a) |y|$ if and only if $\tau(a) = (-|x||y|^{-1})a^{-1}$.
  Let $c:=|x||y|^{-1}$.
  If there is $a$ with $|x+ay|\neq 0$ then $\gen{x+ay},\gen{x+ay}^\perp$ is a different $2$-frame, contrary to the connectedness assumption.
  Hence no such $a$ exists, which means that $\tau(a) = -ca^{-1}$.
  In particular, for $a\in\Fix(\tau)^{\times}$ we get $a^2 + c = 0$.
  Hence $\Fix(\tau) = \GF{2}$ and $\KK = \GF{2^2}$.
  
  On the other hand, if $\KK = \GF{2^2}$, a direct computation shows that $\G(V)$ consists of two vertices connected
  by an edge. Hence the diameter is $1$ in this case.
  \end{proofofcase}
  
  If $n \geq 3$ then Theorem \ref{theoremwalks3Dim3orMore} shows that 
  $\G(V)$ is connected of diameter $\leq 3$ if and only if $n \neq 3$ or $\KK \neq \GF{2^2}$.
  Now, Lemma \ref{walksLength1and2} implies that the diameter is $2$ if and only if either $n \geq 4$, or
  $n=3$ and $V$ is anisotropic.
  Thus this also shows that the diameter is exactly $3$ if $n = 3$ and $\KK\neq\GF{2^2}$.
  This completes the proof of the equivalence of (i) and (iii).
\end{proof}

As a consequence of Theorem \ref{connectedFramePoset} we can deduce some facts about
the connectedness of $\redS(V)$. 

\begin{corollary}
\label{corollarySVConnected}
Let $V$ be a unitary space of dimension $n\geq 3$ over a field $\KK$.
Then the poset $\redS(V)$ is connected if and only if $n > 3$ or $\KK \neq \GF{2^2}$.
\end{corollary}
\begin{proof}
Consider the $(n-2)$-skeleton $\F(V)^{n-2}$ of $\F(V)$.
Then we have a surjective poset map  $\F(V)^{n-2}\to \redS(V)$.
If $n > 3$ or $\KK \neq \GF{2^2}$ then  $\F(V)$ is connected by Theorem \ref{connectedFramePoset}.
Moreover, $\F(V)^{n-2}\hookrightarrow \F(V)$ is an $(n-2)$-equivalence, so $\F(V)^{n-2}$ is connected.
Therefore $\redS(V)$ is connected.

If $n=3$ and $\KK=\GF{2^2}$ then $\redS(V)$ is disconnected with $4$ connected components.
\end{proof}

Finally, as a last application of Theorem \ref{connectedFramePoset} in this section we determine the homotopy type of $\F(V)$ for $n =2,3$. 

\begin{proposition}
\label{lowDimensionalCases}
Let $V$ be a unitary space of dimension $n$ over a field $\KK$.

\begin{itemize}
\item[(i)] If $n=3$ and $\KK\neq \GF{2^ 2}$ then $\F(V)$ is homotopy equivalent to a (non-trivial) wedge of spheres of dimension $1$. For $\KK=\GF{q^2}$, $q \neq 2$, there are $\frac{1}{3} (q^6 - 2 q^5 - q^4 + 2 q^3 - 3 q^2 + 3)$ spheres in the wedge. 

\item[(ii)] If $n=3$ and $\KK=\GF{2^2}$ then $\F(V)$ collapses to a discrete space of $4$ points, so it is a wedge of $3$ spheres of dimension $0$. 

\item[(iii)] If $n=2$ then $\F(V)$ is homotopy equivalent to a discrete space of points.
If $\KK= \GF{q^2}$ then there are $q(q-1)/2$ points. 
\end{itemize}
\end{proposition}

\begin{proof}
Assertion (i) follows from Theorem \ref{connectedFramePoset}, the fact that $\F(V)$ collapses to a subcomplex of dimension $3-2=1$ by Lemma \ref{lemmaHatPoset} and the computation of the reduced Euler characteristic of $\F(V)$ by using Lemma \ref{eulerCharacteristicFrameComplex} in the finite case.

Assertion (ii) follows from Lemma \ref{doubleJumpDimensionInField2}.

Finally, if $V$ is a unitary space of dimension $2$, then $\F(V)$ collapses to a $0$-dimensional subposet by Lemma \ref{lemmaHatPoset}.
If in addition $\KK=\GF{q^2}$, the number of such points is $(q-1)/2$ by Lemma \ref{eulerCharacteristicFrameComplex}.
In particular, $\F(V)$ is homotopy equivalent to a discrete space, and this is contractible if and only if $\KK = \GF{2^2}$.

\end{proof}

\section{Simple connectivity of the frame complex}\label{sec:simple}

In this section, we work with the minimal closed walks in the graph $\G(V)$ and show that if the dimension of $V$ is \textit{big enough}, then they are homotopically trivial.
Since minimal closed walks generate the fundamental group of $\F(V)$, this will allow us to conclude simple connectivity.
A similar approach was employed by Das in \cite{Das} to study the simple connectivity of the poset $\S(V) \setminus \{0,V\}$ of proper non-zero non-degenerate subspaces of $V$.

As mentioned at the beginning of Section \ref{sec:walk}, in a walk we allow the repetition of edges and vertices, while this is not allowed 
in a cycle (except for the vertex closing the cycle). Clearly, a minimal 
closed walk is a cycle. 

An $m$-gon in a graph $\G$ is a closed walk $C$ in $\G$ such that if $v,v'\in C$ then $d_C(v,v') = d_{\G}(v,v')$, where $d_C$, $d_{\G}$ denote the distance between vertices in $C$ (regarded as a subgraph) and $\G$ respectively.
In particular, an $m$-gon is a cycle.
We say that a closed walk $C$ in the graph $\G$, based on a vertex $x\in \G$ is a \textit{tailed $m$-gon based on $x$} if there exists an $m$-gon $C'$ and a walk $\gamma$ from $x$ to some vertex $y\in C'$ such that $C = \gamma * C' * \gamma^{-1}$, where $*$ denotes the concatenation of paths.
That is, $C$ is the cycle that goes through $\gamma$ until the vertex $y$, then it goes through $C'$ and then it comes back via $\gamma^{-1}$ to $x$.
See Figure \ref{fig:exampleTailedMGon}.
Analogously we can define tailed cycles.

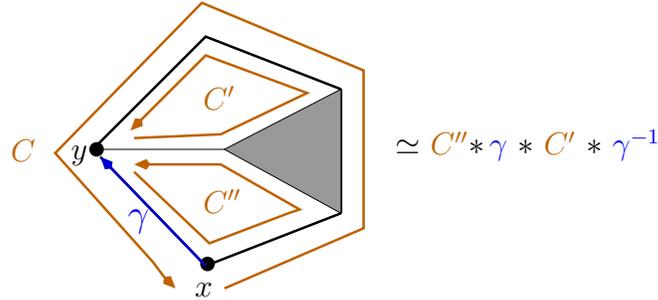
\begin{figure}[ht]

\centering
\begin{picture}(0,0)%
 \includegraphics[scale=0.5]{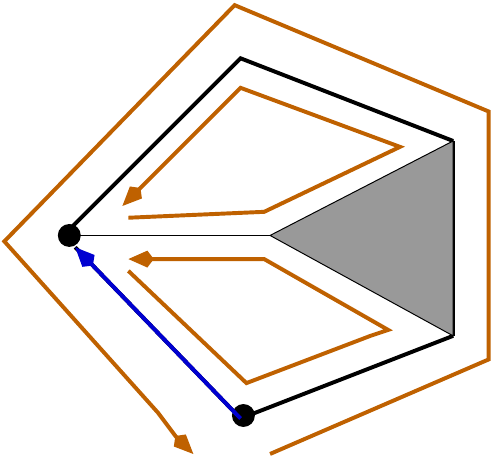}%
\end{picture}%
\setlength{\unitlength}{4144sp}%
\begingroup\makeatletter\ifx\SetFigFont\undefined%
\gdef\SetFigFont#1#2#3#4#5{%
  \reset@font\fontsize{#1}{#2pt}%
  \fontfamily{#3}\fontseries{#4}\fontshape{#5}%
  \selectfont}%
\fi\endgroup%
\begin{picture}(3486,3216)(1948,-3919)
\put(2850,-2830){\makebox(0,0)[lb]{\smash{{\SetFigFont{12}{14.4}{\rmdefault}{\mddefault}{\updefault}{\color[rgb]{.75,.38,0}$C'$}%
}}}}
\put(2850,-3450){\makebox(0,0)[lb]{\smash{{\SetFigFont{12}{14.4}{\rmdefault}{\mddefault}{\updefault}{\color[rgb]{.75,.38,0}$C''$}%
}}}}
\put(2800,-3950){\makebox(0,0)[lb]{\smash{{\SetFigFont{12}{14.4}{\rmdefault}{\mddefault}{\updefault}{\color{black}$x$}%
}}}}
\put(1700,-3140){\makebox(0,0)[lb]{\smash{{\SetFigFont{12}{14.4}{\rmdefault}{\mddefault}{\updefault}{\color[rgb]{.75,.38,0}$C$}%
}}}}
\put(2400,-3530){\makebox(0,0)[lb]{\smash{{\SetFigFont{12}{14.4}{\rmdefault}{\mddefault}{\updefault}{\color{blue} \Large{$\gamma$}}%
}}}}
\put(2060,-3130){\makebox(0,0)[lb]{\smash{{\SetFigFont{12}{14.4}{\rmdefault}{\mddefault}{\updefault}{\color{black} $y$}
}}}}
\put(4000,-3100){\makebox(0,0)[lb]{\smash{{\SetFigFont{12}{14.4}{\rmdefault}{\mddefault}{\updefault}{{$\simeq~$}\color[rgb]{.75,.38,0}$C''$\color{black}$*\,$\color{blue} $\gamma$ \color{black} $*$ \color[rgb]{.75,.38,0}$C'$ \color{black}$*$ \color{blue} $\gamma^{-1}$}%
}}}}
\end{picture}%

\bigskip

\caption{A pentagon loop $C$ based at $x$ is homotopic to a concatenation of two loops, $C''$ and $\gamma * C' * \gamma^{-1}$, where we can regard $C''$ and $C'$ as $4$-gons based on $x$ and $y$ respectively. Then $\gamma * C' * \gamma^{-1}$ is a tailed $4$-gon based on $x$.}
\label{fig:exampleTailedMGon}
\end{figure}

Note that the fundamental group of the clique complex $K(\G)$ of $\G$ is generated by homotopy classes of tailed $m$-gons in $\G$ (for possibly different values of $m$): every loop of $K(\G)$ is homotopic to a closed walk in the graph $\G$, which is in fact homotopic to a concatenation of tailed simple cycles, and each one of these tailed simple cycles can be homotopically carried to a concatenation of tailed $m$-gons.
Moreover, it is not hard to show that if $\diam(\G) = d$, then only $m$-gons with $m\leq 2d+1$ can exist: otherwise we could shorten the $m$-gon by a diagonal walk of length $\leq d$ between vertices at distance $>d$ inside the $m$-gon, contradicting the definition of $m$-gon.
Therefore, we see that, for a connected graph $\G$, $\pi_1(K(\G)) = 1$ if and only if every $m$-gon is null-homotopic, and $m \leq 2\diam(\G) + 1$.
See \cite{ASegev} for more details.
In contrast to \cite{ASegev, Das}, we will focus on showing that $m$-gons are homotopically trivial loops, although we will usually get that they are \textit{triangulable} in the sense of \cite{ASegev}. For the sake of simplicity, we have opted to omit the use of this language.

By Theorem \ref{connectedFramePoset}, if $\dim(V)\geq 4$ then $\G(V)$ is connected with $\diam(\G(V)) = 2$.
Therefore, in order to compute the (tailed) $m$-gons generating $\pi_1(\F(V))$ we only need to worry about $m$-gons with $m\in \{3,4,5\}$.
Moreover, note that a $3$-gon in a clique complex is always homotopically trivial.
Hence, in order to show that $\pi_1(\F(V)) = 1$, we need to prove that $4$-gons and $5$-gons are null-homotopic.

\begin{lemma}
\label{lemmaTriangulableSquares}
Let $V$ be a unitary space of dimension $n\geq 5$ over a field $\KK$ such that $(n,\KK)\neq (6,\GF{2^2})$.
Then every closed cycle of length $4$ in $\G(V)$ is null-homotopic.
\end{lemma}

\begin{proof}
Let $C:v_1,v_2,v_3,v_4,v_1$ be a closed cycle of length $4$ in $\G(V)$.
We can assume that $C$ is a $4$-gon.

Let $S = v_1+v_3$, so $v_2,v_4\leq S^\perp$.
Thus, if $\G(S^\perp)$ is connected, we get a walk in $\G(S^\perp)$ joining $v_2$ with $v_4$.
Since each vertex of this walk is joined with $v_1$ and $v_3$, we conclude that $C$ is null-homotopic.
See Figure \ref{fig:triangulationByConnectedness}.

\begin{figure}[ht]
\includegraphics[scale=0.4]{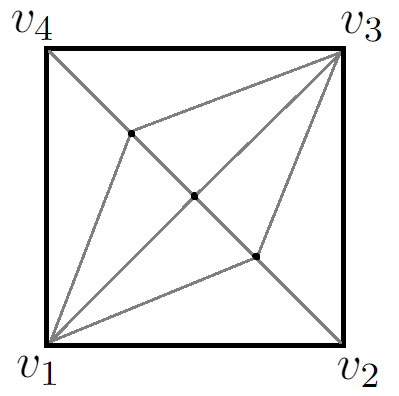}
\caption{Triangulation by diagonal walk if $\G((v_1+v_3)^\perp)$ is connected.}
\label{fig:triangulationByConnectedness}
\end{figure}

Assume then that $\G(S^\perp)$ is disconnected.
Since $C$ is a $4$-gon, $\dim(S) = 2$.
This implies that $\dim(S^\perp) \geq 3$ and also that $\Rad(S^\perp) = \Rad(S)$ has dimension $0$ or $1$.
Then $d:=\dim(S^\perp/\Rad(S^\perp))\geq 2$ and by Proposition \ref{cor:degenerateConnected}, $\G(S^\perp/\Rad(S^\perp))$ is disconnected.
By Theorem \ref{connectedFramePoset}, we have that either $d= 2$ and $\KK \neq \GF{2^2}$ or else $d = 3$ and $\KK = \GF{2^2}$.
We can proceed in a similar fashion with $T := v_2+v_4$ and hence make the same assumptions on $T$ as with $S$.

Then, we split the remaining of the proof into these two cases.
Fix generators $x,y,z,w$ of the subspaces $v_1,v_2,v_3,v_4$ respectively.

\medskip

\noindent\textbf{Case 1.} Assume $\KK \neq \GF{2^2}$. If $\dim(S^\perp/\Rad(S^\perp)) = 2$ and  $\dim(T^\perp/\Rad(T^\perp)) = 2$ then $C$ is null-homotopic.

\begin{proofofcase}[Proof of Case 1.]
Let $r = \dim(\Rad(S))$.
Then $2 = \dim(S^\perp) - \dim(\Rad(S)) = n-2 -r\geq 3-r \geq 2$ implies that $(n,r) = (5,1)$.
Thus we can write $\Rad(S) = \gen{x+bz}$ for some $b\in \KK^{\times}$.
Analogously we have $\Rad(T) = \gen{y+aw}$ for some $a\in \KK^{\times}$.

Consider $W = \gen{x,y,z,w}$.
Note that
\[ W = \gen{x,y} \oplus \gen{x+bz,y+aw},\]
and that indeed $\Rad(W) = \gen{x+bz,y+aw}$, which is orthogonal to the non-degenerate subspace $\gen{x,y}$.
Now, if $\dim(\Rad(W)) = 2$, then $\dim(W) = 4$ and
\[2  = \dim(\Rad(W)) \leq \dim(W^\perp) = n - \dim(W) = 1.\]
This is a contradiction which leads to $\dim(\Rad(W)) = 1$.
We conclude then that $\dim(W) = 3$ and that
\[ \dim(W^\perp / \dim(\Rad(W^\perp))) = 5 - 3 - 1 = 1 > 0.\]
Hence there exists a vector $u\in W^\perp$ with $|u|\neq 0$.
This yields a triangulation of $C$ as in Figure \ref{triang2}, where $v = \gen{u}$.
In particular, $C$ is null-homotopic.
\end{proofofcase}

Before moving to the other case, note that if $q = 2$ then $q^2-1 = q+1$.
This implies that if $u\in \GF{2^2}^n$, then 
\[ |u| = \sum_{i=1}^n u_i^{q+1} = \sum_{i\tq u_i\neq 0} 1.\]
Hence $u$ is non-degenerate if and only if it has an odd number of non-zero entries, if and only if $|u| = 1$.

Now we prove that the second case above cannot arise.

\medskip

\noindent\textbf{Case 2.} If $\KK = \GF{2^2}$ then $\dim(S^\perp/\Rad(S^\perp)) \neq 3$.

\begin{proofofcase}[Proof of Case 2.]
By the way of contradiction, suppose that $\dim(S^\perp/\Rad(S^\perp)) = 3$.
Let $r = \dim(\Rad(S))$, which is $0$ or $1$.
Then $\dim(S) = 2$ implies that $3 = n - 2 - r$, that is, $n = 5 + r$.
Hence either $(n,r) = (5,0)$ or else $(n,r) = (6,1)$.
The latter case is the exceptional case excluded in the hypotheses of the lemma.
Therefore we are in the situation $(n,r) = (5,0)$.
However, we will see that $r \geq 1$ and thus arrive to a contradiction.

Write $z = ax + \tilde{z}$, where $a\in\KK^{\times}$, $\tilde{z}\in \gen{x}^\perp$.
Also note that $\tilde{z}\neq 0$ since $v_1\neq v_3$ by hypothesis on $C$.
Since $x,z$ are non-degenerate and $\KK = \GF{2^2}$, we have
\[ 1 = |z| = |x| + |\tilde{z}| = 1 + |\tilde{z}|.\]
This proves that $|\tilde{z}|=0$, i.e. $\tilde{z}$ is isotropic.
Therefore $0 \neq \tilde{z}\in \Rad(S)$, a contradiction.
\end{proofofcase}

We have shown that in any case, a $4$-gon is null-homotopic (except when $n=6$ and $q=2$).
This concludes the proof of this lemma.
\end{proof} 

\begin{figure}[ht]
\centering
\includegraphics[scale=0.4]{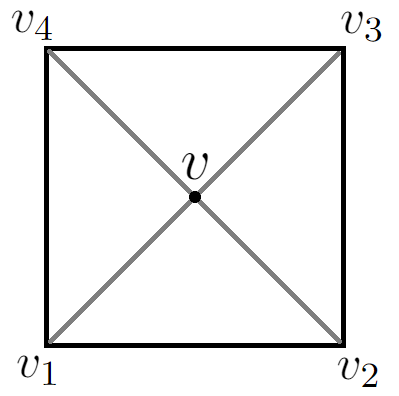}
\caption{Triangulation by ``starring'', with a  vertex $v$ orthogonal to all the vertices of the square.}
\label{triang2}
\end{figure}

Next we move to the case of $5$-gons.
We prove that a $5$-gon cycle is homotopic to a product of $4$-gons, squares and triangles.
Then by the previous lemma and the remarks at the beginning of this section, we will be able to conclude that $5$-gons are null-homotopic.

\begin{lemma}
\label{lemmaTriangulablePentagons}
Let $V$ be a unitary space over a field $\KK$ of finite dimension $n\geq 5$, with $(n,\KK)\neq (6,\GF{2^2})$.
Then every closed cycle of length $5$ in $\G(V)$ is null-homotopic.
\end{lemma}

\begin{proof}
Similar as in Lemma \ref{lemmaTriangulableSquares}, we can suppose that $C:v_1,v_2,v_3,v_4,v_5,v_1$ is a $5$-gon.
Let $S = v_1+v_3+v_4$.
Then $\dim(S)\leq 3$, $\dim(\Rad(S)) \leq 1$ and
\[ \dim(S^\perp/ \Rad(S^\perp)) = n - \dim(S) - \dim(\Rad(S)) \geq 5 - 3 - 1 > 0.\]
By Proposition \ref{dimensionTheorem}(7), there exists a vector $u\in S^\perp$ with $|u|\neq 0$.
For $v = \gen{u}$, we get a decomposition of $C$ into two squares and one triangle, as shown in Figure \ref{pentagon}.
Hence, by Lemma \ref{lemmaTriangulableSquares}, we conclude that $C$ is null-homotopic, except if $n = 6$ and $\KK = \GF{2^2}$.
\end{proof}


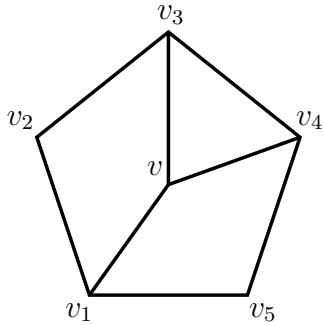
\begin{figure}[ht]
\centering
\begin{picture}(0,0)%
\includegraphics[scale=0.7]{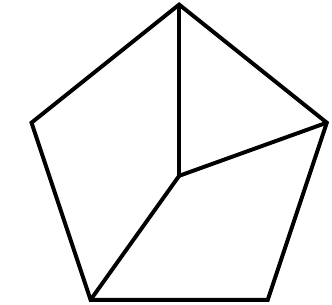}%
\end{picture}%
\setlength{\unitlength}{4144sp}%
\begingroup\makeatletter\ifx\SetFigFont\undefined%
\gdef\SetFigFont#1#2#3#4#5{%
  \reset@font\fontsize{#1}{#2pt}%
  \fontfamily{#3}\fontseries{#4}\fontshape{#5}%
  \selectfont}%
\fi\endgroup%
\begin{picture}(2550,2310)(1111,-2344)
\put(1100,-1300){\makebox(0,0)[lb]{\smash{{\SetFigFont{12}{14.4}{\rmdefault}{\mddefault}{\updefault}{\color[rgb]{0,0,0}$v_2$}%
}}}}
\put(2830,-1300){\makebox(0,0)[lb]{\smash{{\SetFigFont{12}{14.4}{\rmdefault}{\mddefault}{\updefault}{\color[rgb]{0,0,0}$v_4$}%
}}}}
\put(2000,-670){\makebox(0,0)[lb]{\smash{{\SetFigFont{12}{14.4}{\rmdefault}{\mddefault}{\updefault}{\color[rgb]{0,0,0}$v_3$}%
}}}}
\put(1450,-2450){\makebox(0,0)[lb]{\smash{{\SetFigFont{12}{14.4}{\rmdefault}{\mddefault}{\updefault}{\color[rgb]{0,0,0}$v_1$}%
}}}}
\put(2550,-2450){\makebox(0,0)[lb]{\smash{{\SetFigFont{12}{14.4}{\rmdefault}{\mddefault}{\updefault}{\color[rgb]{0,0,0}$v_5$}%
}}}}
\put(1940,-1600){\makebox(0,0)[lb]{\smash{{\SetFigFont{12}{14.4}{\rmdefault}{\mddefault}{\updefault}{\color[rgb]{0,0,0}$v$}%
}}}}
\end{picture}%

 \medskip
\caption{Pentagon cut into two squares and a triangle.}
\label{pentagon}
\end{figure}

By Lemmas \ref{lemmaTriangulableSquares} and \ref{lemmaTriangulablePentagons}, and the discussion at the beginning of this section, we conclude:

\begin{corollary}
\label{corollarySimplyConnectedDim5}
If $V$ is a unitary space of dimension $n\geq 5$ over $\KK$, with $(n,\KK)\neq (6,\GF{2^2})$, then $\F(V)$ is simply connected.
\end{corollary}
%

\begin{example}
\label{exampleDim6GF2}
Let $V$ be a unitary space of dimension $6$ over $\GF{2^2}$.
Computations in GAP show that $\pi_1(\F(V))$ is isomorphic to the Klein four-group $C_2\times C_2$.
By Lemma \ref{doubleJumpDimensionInField2}, $\F(V)$ collapses to a $6-3=3$-dimensional complex, but this is not homotopic to a wedge of spheres.
\end{example}

\begin{example}
\label{exampleDim4Field2}
Let $V$ be a unitary space of dimension $4$ over $\GF{2^2}$.
Then $\F(V)$ is homotopy equivalent to a $4-3=1$ dimensional complex (by Lemma \ref{doubleJumpDimensionInField2}) and its reduced Euler characteristic is $-81$ by Lemma \ref{eulerCharacteristicFrameComplex}.
Since $\F(V)$ is connected by Theorem \ref{connectedFramePoset}, we conclude that $\F(V)$ is homotopy equivalent to a wedge of $81$ spheres of dimension $1$.
\end{example}

\begin{example}
\label{exampleDim4Field3}
Let $V$ be a unitary space of dimension $4$ over $\GF{3^2}$.
Then $\F(V)$ collapses to a $(n-1)-1=2$-dimensional subcomplex.
Its reduced Euler characteristic is $9044$.
Its second homology group is free abelian of rank $9114$, and its degree $1$ homology group is free abelian of rank $70$.
In particular, $\F(V)$ is not simply connected.
\end{example}

Now we have all the ingredients to prove Theorem \ref{simplyConnectedFramePoset}.

\begin{proof}[Proof of Theorem \ref{simplyConnectedFramePoset}]
Assertion (i) follows from Corollary \ref{corollarySimplyConnectedDim5}.
Assertion (ii) follows from Example \ref{exampleDim6GF2}.
Assertion (iii) follows from Examples \ref{exampleDim4Field2} and \ref{exampleDim4Field3}.
Assertion (iv) follows from Remark \ref{rk:propertyT}.
Finally, assertion (v) follows from Proposition \ref{lowDimensionalCases}.
\end{proof}

Therefore, in order to fully answer the question of when $\F(V)$ is simply connected, it remains to analyze the case $\dim(V) = 4$ and $|\KK| \geq 4^2$.

\bigskip

\noindent\textbf{Question.} For which fields $\KK$ is $\F(V)$ simply connected when $\dim(V)=4$?

\bigskip

The following proposition reduces the study of the above question to understanding when $4$-gons are null-homotopic.
To simplify the computations, we work under (ON).

\begin{proposition}
Let $V$ be a unitary space of dimension $n=4$ over $\KK$ satisfying (ON).
Then every cycle of length $5$ in $\F(V)$ is homotopic to a concatenation of squares.
Therefore, $\pi_1(\F(V)) = 1$ if and only if $4$-gons are null-homotopic.
\end{proposition}

\begin{proof}
Let $C:v_0,v_1,v_2,v_3,v_4,v_0$ be a cycle of length $5$.
As usual, we can suppose that $C$ is a $5$-gon.
We show then that $C$ is homotopic to a concatenation of squares.
Take generators $x_i$ of $v_i$.
The main assumption (ON) allows us to pick the $x_i$ such that $|x_i| = 1$.
Define $\alpha_{ij} = \Psi(x_i,x_j)$ and $\alpha_{ij}^* = \alpha_{ij}\alpha_{ji}$.
Then $\alpha_{ii} = 1$ for all $i$ and $\alpha_{ij} = 0$ if $j=i\pm 1$ modulo $5$.

For $0\leq i\leq 4$, define $S_i := \gen{x_i,x_{i+2},x_{i+3}}$ (with indexes modulo $5$).
It is not hard to show that $\dim(S_i) = 2$ or $3$ since $C$ is a $5$-gon.
If $S_i^\perp / \Rad(S_i) \neq 0$ then $S_i^\perp$ contains a non-degenerate vector $x\in S_i^\perp$.
This provides a decomposition of $C$ into two squares and a triangle, similarly as in Figure \ref{pentagon}.

Therefore we suppose that $S_i^\perp = \Rad(S_i)$ for all $i$.
This is equivalent to $S_i^\perp = \gen{ax_i + b x_{i+2} + c x_{i+3}}\neq 0$.
And indeed this is equivalent to having a non-trivial solution for the coefficients $a,b,c$ above subject to the condition that the vector $ax_i + b x_{i+2} + c x_{i+3}$ is orthogonal to $x_i$, $x_{i+2}$ and $x_{i+3}$.
This means that the following determinant is $0$:
\[ \det \left(\begin{matrix}
1 & \alpha_{(i+2)i} & \alpha_{(i+3)i}\\ 
\alpha_{i(i+2)} & 1 & 0\\ 
\alpha_{i(i+4)} & 0 & 1
\end{matrix}\right) = 1 - \alpha_{i(i+2)}^* - \alpha_{i(i+3)}^*. \]
Equating this determinant to $0$ for every $i=0,1,2,3,4$, we get the following equations:
\[\begin{cases}
1 = \alpha_{02}^* +  \alpha_{03}^*,\\
1 = \alpha_{13}^* +  \alpha_{14}^*,\\
1 = \alpha_{24}^* +  \alpha_{02}^*,\\
1 = \alpha_{03}^* +  \alpha_{13}^*,\\
1 = \alpha_{14}^* +  \alpha_{24}^*.
\end{cases}\]
From these conditions we see that $2 \alpha_{ij}^* = 1$ for all $j = i+2$ or $i+3$ (modulo $5$).
In particular, $\chara(\KK)\neq 2$ and $\alpha_{ij}^* = 2^{-1}$ for the indexes $i,j = i+2$ or $i+3$ modulo $5$.

Take $T = \gen{x_0,x_1,x_2}$ and let $x_2' := x_2 - \alpha_{20} x_0$.
Then
\begin{equation}
    \label{eq:x2primeNorm}
    |x_2'| = |x_2| - \alpha_{20}\alpha_{02} - \alpha_{20}\alpha_{02} + \alpha_{20}\alpha_{02} |x_0| = 1 - \alpha^*_{02} = 1-2^{-1} = 2^{-1} \neq 0. 
\end{equation}
This implies that $T$ is non-degenerate of dimension $3$, so there exists $x_5\in T^\perp$ with $|x_5| = 1$.
If $v_5:=\gen{x_5}$, we see that $C$ is homotopic to the pentagon $C':v_0,v_5,v_2,v_3,v_4,v_0$.
We prove then that $C'$ is homotopic to a concatenation of squares.
See Figure \ref{fig:pentagon1}.

\begin{figure}[ht]
    \centering
    \includegraphics[scale=0.4]{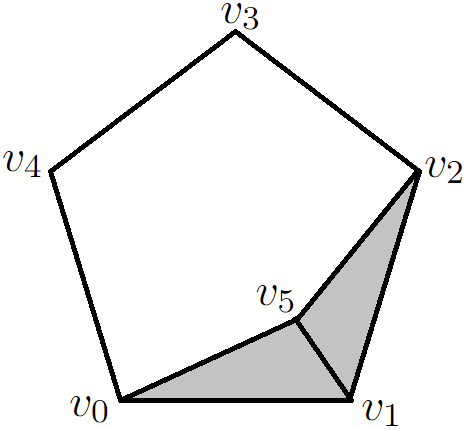}
    \caption{Addition of extra vertex $v_5 = \gen{x_5}$ orthogonal to $v_0,v_1,v_2$.}
    \label{fig:pentagon1}
\end{figure}

Similar as in the beginning of the proof, we show that $S' = \gen{x_0,x_5,x_3}$ is non-degenerate, i.e. $\Rad(S') = 0$, and hence that $C'$ can be decomposed into two squares and a triangle (see Figure \ref{pentagon}, and see also Figure \ref{fig:pentagon2}).
This will conclude the proof of the first part of the proposition.

\begin{figure}[ht]
    \centering
    \includegraphics[scale=0.4]{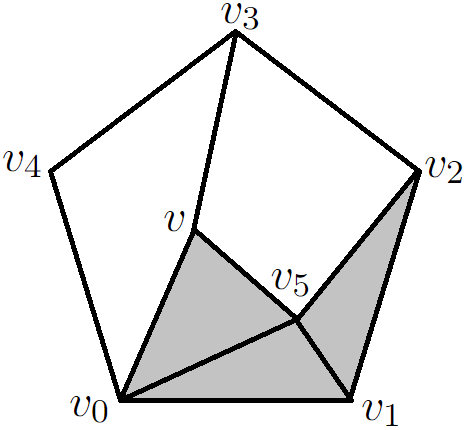}
    \caption{$5$-gon divided into filled triangles and squares, so it is homotopic to concatenation of squares.}
    \label{fig:pentagon2}
\end{figure}

Let $a\in \KK$ be such that $y:=ax_2'$ satisfies $|y| = 1$.
That is, $a\tau(a) = |x_2'|^{-1} = 2$ by \eqref{eq:x2primeNorm}.
Hence $\{x_0,x_1,y,x_5\}$ is an orthonormal basis of $V$.
We have that $x_3 = \alpha_{30} x_0 + \alpha_{31} x_1 + \Psi(x_3,y) y + \alpha_{53} x_5$.
Then
\[ S' = \gen{x_0,x_5,\alpha_{31} x_1 + \Psi(x_3,y) y} = \gen{x_0,x_5} \oplus \gen{\alpha_{31} x_1 + \Psi(x_3,y) y}, \]
where $\gen{x_0,x_5}$ and $\gen{\alpha_{31} x_1 + \Psi(x_3,y) y}$ are orthogonal.
In order to prove that $\Rad(S') = 0$, it is enough to show that $|\alpha_{31} x_1 + \Psi(x_3,y) y|\neq 0$.
But this follows by direct computation:
\begin{align*}
    |\alpha_{31} x_1 + \Psi(x_3,y) y| & = \alpha_{31}^* + \Psi(x_3,y)\Psi(y,x_3)\\
    & = 2^{-1} + \Psi(x_3,ax_2')\Psi(ax_2',x_3)\\
    & = 2^{-1} + a\tau(a) \Psi(x_3,x_2 - \alpha_{20} x_0)\Psi(x_2 - \alpha_{20} x_0,x_3)\\
    & = 2^{-1} + 2 \Psi(x_3,-\alpha_{20}x_0) \Psi(-\alpha_{20}x_0,x_3)\\
    & = 2^{-1} + 2(-\alpha_{02}) \alpha_{30} (-\alpha_{20}) \alpha_{03}\\
    & = 2^{-1} + 2 \alpha_{02}^* \alpha_{03}^*\\
    & = 2^{-1} + 2 \cdot 2^{-1} \cdot 2^{-1} = 1.
\end{align*}
This proves that $\Rad(S') = 0$.

We have shown that in any case, a (based) pentagon is homotopic to a concatenation of (tailed) squares and filled triangles.
In particular, we can conclude that $\pi_1(\F(V)) = 1$ if and only if $4$-gons are null-homotopic.
\end{proof}

We close this section with an application of this result to the fundamental group of $\redS(V)$.
Recall that by previous results (see \cite{BeS,Das,DGM}), for a finite field $\KK = \GF{q^2}$, $\redS(V)$ is simply connected if (and only if) $\dim(V)\geq 5$, or $\dim(V) = 4$ with $q>3$.
In the following corollary, we show that we can cover the case $\dim(V) \geq 5$ of this result by using our methods in conjunction with a version of
Quillen's fiber-Theorem.

\begin{corollary}
\label{corollarySVSimplyConnected}
Let $V$ be a unitary space of dimension $n\geq 4$ over a field $\KK$ distinct from $\GF{2^2}$.
Then $\phi: \F(V)^{n-2}\to \redS(V)$ is a $1$-equivalence.
The same conclusion holds for $\KK = \GF{2^2}$ if in addition $n\geq 5$.

In particular, if $n\geq 5$ then $\redS(V)$ is simply connected.
And if $n = 4$ and $\F(V)$ is simply connected (so $|\KK|\geq 4^2$), then $\redS(V)$ is simply connected.
\end{corollary}

\begin{proof}
We apply Proposition \ref{quillensFiberTheoremGeneral} to
the map $\phi:\F(V)^{n-2} \to \redS(V)$ with $m=0$.
Let $S\in \redS(V)$, and observe that $\phi^{-1}(\redS(V)_{\leq S}) = \F(S)^{n-2} = \F(S)$ and $\redS(V)_{>S} = \redS(S^\perp)$.
Let $Z_S := \F(S) * \redS(S^\perp)$.
We will prove that $Z_S$ is $0$-connected for any $S$.

We prove first that $Z_S$ is $0$-connected when $\KK\neq\GF{2^2}$ and $n\geq 4$.
Under these assumptions, the following assertions hold:
\begin{itemize}
    \item If $\dim(S) = 1$, then $\F(S) = \{S\}$, so $Z_S$ is contractible.
    \item If $\dim(S) = 2$, then $\F(S)$ is homotopy discrete, i.e. $(-1)$-connected. Since $\codim(S) = n-2\geq 2$, $\redS(S^\perp)$ is non-empty.
    Thus $Z_S$ is $0$-connected.
    \item If $\dim(S) \geq 3$, then $\F(S)$ is connected by Theorem \ref{connectedFramePoset}, so $Z_S$ is $0$-connected.
\end{itemize}
By Proposition \ref{quillensFiberTheoremGeneral} we conclude that $\phi$ is a $1$-equivalence.

Now suppose that $\KK = \GF{2^2}$ and $n\geq 5$.
Note that $\F(S)$ is connected if $\dim(S)\neq 3$ by Theorem \ref{connectedFramePoset}.
If $\dim(S) = 3$, then $\F(S)\neq\emptyset$ and $\redS(S^\perp)$ is non-empty since $\dim(S^\perp) = n-3\geq 2$.
In any case, $Z_S$ is $0$-connected.
Thus $\phi$ is a $1$-equivalence.

Finally, if $\dim(V)\geq 5$, $\pi_1(\F(V)) = 1$ by Theorem \ref{simplyConnectedFramePoset}, and therefore we also get $\pi_1(\redS(V)) = 1$, that is, $\redS(V)$ is simply connected.
\end{proof}

\section{Eigenvalues of the graph $\G(V)$}
\label{sec:eigen}

In this section we focus on computing the eigenvalues of the adjacency matrix of the graph $\G(V)$, for a unitary space $V$ over a finite field $\KK = \GF{q^2}$.
These computations will lead to the proof of Theorem \ref{theoremEigenvalues}.
Recall that our aim is to apply this theorem to conclude that some homology groups of $\F(V)$ have vanishing free part, by invoking Garland's method (see Theorem \ref{theoremGarland}).

Based on computer experiment we guessed that 
the eigenvalues of the adjacency matrix of $\G(V)$ have a very simple form
and there are at most $4$ of them (depending on $q$ and $n$). Thus the minimal polynomial of the adjacency matrix of $\G(V)$ has degree $\leq 4$. 


In order to carry out the verification of the minimal polynomial, we compute the powers of the adjacency matrix up to $4$.
Recall that the entry in row $i$ and column $j$ of the $k$\textsuperscript{th}-power of an adjacency matrix counts the number of walks of length $k$ between the vertices $i$ and $j$ and $\G(V)$. 
Thus in the following proof  we can resort to the computation from Section \ref{sec:walk}.

\begin{proof}[Proof of Theorem \ref{theoremEigenvalues}]
Write $m(n,q) = X^4 + c_3 X^3 + c_2 X^2 + c_1 X + c_0$, where
\begin{align*}
c_0 & = -d_n q^{3n-7} (-1)^{n},\\
c_1 & = d_n q^{2n-4} + q^{3n-7}  (-1)^{n},\\
c_2 & = d_nq^{n-3}(-1)^{n} - q^{2n-4},\\
c_3 & = - d_n - q^{n-3}(-1)^{n},\\
c_4 & = 1.
\end{align*}
Let $A_n$ be the adjacency matrix of $\G(V)$.
If the claim  of the theorem is correct then $m(n,q)(A_n) = 0$.
Therefore, we show for all $q$ that the following matrix identity holds:
\begin{equation}
\label{matrixEquality}
 A_n^4 + c_3 A_n^3 + c_2 A_n^2 + c_1 A_n + c_0 \Id = 0.
\end{equation}
As mentioned before the proof, the entry in row $S$ and column $W$ of $A_n^k$ represents the number of walks of length $k$ between $S$ and $W$.
We have computed these coefficients for $k=1,2,3,4$ in Lemma \ref{walksLength1and2} and Theorems \ref{theoremwalks3Dim3orMore} and \ref{theoremwalks4}.
Moreover, we have seen that, to compute these coefficients, we need to consider four different cases:
\begin{enumerate}
\item[($=$)] $S = W$,
\item[($\perp$)] $S \perp W$,
\item[(ND)] $S + W$ non-degenerate, $S\neq W$, $S\not\perp W$,
\item[(D)] $S + W$ degenerate.
\end{enumerate}
In order to establish (\ref{matrixEquality}), we show that each entry of the matrix at the left-hand side of \eqref{matrixEquality} is zero.
Therefore, we need to prove that the following equality holds
for $\bullet$ being one of $=$, $\perp$, ND or D:  
\begin{eqnarray*}
    \sum_{k=0}^4 c_k l_k(\bullet) & = 0.
\end{eqnarray*}

At many points of the proof, we will use the list of identities given in Lemma \ref{lem:identities}.

\medskip

\noindent\textbf{Case 1.} $\sum_{k=0}^4 c_k l_k(=) = 0$.


\begin{proofofcase}[Proof of Case 1:]
We replace the values of the sum on the left-hand side:
\begin{align*}
    d_n & l_3(\perp) - (d_n + q^{n-3}(-1)^n)d_n l_2(\perp) + (d_nq^{n-3}(-1)^n - q^{2n-4})l_2(=) - d_n q^{3n-7}(-1)^n = \\
    & = d_n\left( d_n + (d_{n-1})^2 + d_{n-1}(d_n-d_{n-1}-I_{n-2}-1) + d^1_{n-1}I_{n-2} \right.\\
    & \qquad \left.  - (d_n + q^{n-3}(-1)^n)d_{n-1} + d_nq^{n-3}(-1)^n - q^{2n-4} -q^{3n-7}(-1)^n\right)\\
    & = d_n\left( (d_n-d_{n-1}-q^{2n-4})(1+q^{n-3}(-1)^n) + I_{n-2}(d^1_{n-1}-d_{n-1}) \right)\\
    & = d_n \left( (-q^{2n-5} + q^{n-3}(-1)^n)(1+q^{n-3}(-1)^n)\right.\\
    & \qquad \quad \left. + (q^{2n-5}-1  + q^{n-2}(-1)^n-q^{n-3}(-1)^n)q^{n-3}(-1)^n \right)\\
    & = 0.
\end{align*}
\end{proofofcase}

\noindent\textbf{Case 2.} $l_4(\perp) = - \sum_{k=0}^3 c_k l_k(\perp)$, that is, $\sum_{k=0}^4 c_k l_k(\perp) = 0$.

\begin{proofofcase}[Proof of Case 2:]
By Theorem \ref{theoremwalks4} and the relations between $l_3(\perp)$, $l_3(\text{ND})$ and $l_3(\text{D})$ in Table \ref{tab:walks3}, we have:
\begin{align*}
l_4 & (\perp) = (d_n + l_3(\perp))d_{n-1} + l_3(\text{D}) I_{n-2} + l_3(\text{ND})(d_n-1-d_{n-1}-I_{n-2})\\
& = (d_n + l_3(\text{ND})+q^{2n-4})d_{n-1} + \left(l_3(\text{ND}) + q^{2n-6}\right )I_{n-2}+ l_3(\text{ND})\left (d_n - 1 - d_{n-1} - I_{n-2}\right )\\
& = d_n d_{n-1} + q^{2n-4} d_{n-1} + q^{2n-6} I_{n-2} + d_n l_3(\text{ND})  - l_3(\text{ND}) \\
& = d_n l_3(\text{ND}) - q^{3n-8}(-1)^n  + q^{2n-6} + q^{2n-4} d_{n-1} + q^{2n-6} I_{n-2}\\
& = d_n l_3(\text{ND}) - q^{3n-8}(-1)^n + q^{2n-6} + q^{2n-4} d_{n-1} + q^{4n-10} - q^{2n-6} - d_{n-1}q^{2n-6}(q^2-1)\\
& = d_n l_3(\text{ND}) - q^{3n-8}(-1)^n  + q^{4n-10} +  d_{n-1}q^{2n-6}.
\end{align*}
On the other hand, the term $c_3 l_3(\perp)+c_2 l_2(\perp)+c_1 l_1(\perp) + c_0 l_0(\perp)$ is given by:
\begin{align*}
    - (d_n + & q^{n-3}(-1)^n)(l_3(\text{ND}) + q^{2n-4}) + (d_nq^{n-3}(-1)^n-q^{2n-4})d_{n-1} + d_n q^{2n-4}+q^{3n-7} = \\
    & = -d_n l_3(\text{ND}) - q^{n-3}(-1)^n l_3(\text{ND}) + d_n d_{n-1} q^{n-3}(-1)^n - d_{n-1}q^{2n-4}\\
    & = -d_n l_3(\text{ND}) -q^{n-3}(-1)^n(d_n d_{n-1} + q^{3n-8}(-1)^n-q^{2n-6})\\
    & \qquad + d_n d_{n-1} q^{n-3}(-1)^n - d_{n-1}q^{2n-4}\\
    & = -d_n l_3(\text{ND}) - q^{4n-11} + q^{3n-9}(-1)^n - d_{n-1}q^{2n-4}\\
    & = - l_4(\perp) + \left(q^{4n-10}-q^{4n-11} + q^{3n-9}(-1)^n-q^{3n-8}(-1)^n - d_{n-1}q^{2n-6}(q^2-1) \right)\\
    & = - l_4(\perp) + (q-1)\left( q^{4n-11} - q^{3n-9}(-1)^n - q^{3n-9}(q^{n-2}-(-1)^n) \right)\\
    & = -l_4(\perp).
\end{align*}
\end{proofofcase}

\noindent \textbf{Case 3.} $l_4(\text{ND}) = - \sum_{k=0}^3 c_k l_k(\text{ND})$, that is, $\sum_{k=0}^4 c_k l_k(\text{ND}) = 0$.
Here $q\neq 2$ by Lemma \ref{remarkCaseqEquals2}.

\begin{proofofcase}[Proof of Case 3:]
By Theorem \ref{theoremwalks4}:
\begin{align*}
l_4(\text{ND}) & = l_3(\perp)d_{n-1} + l_3(\text{ND})\left(d_n-d_{n-1}\right ) + (q^{2n-5} - q^{n-3}(-1)^n)\left(l_3(\text{D}) - l_3(\text{ND}) \right)\\
& = l_3(\perp)d_{n-1} + ( l_3(\perp) - q^{2n-4})(d_n-d_{n-1}) + (q^{2n-5} - q^{n-3}(-1)^n)q^{2n-6}\\
& =  l_3(\perp)d_n - q^{2n-4}(d_n-d_{n-1}) + (q^{2n-5} - q^{n-3}(-1)^n)q^{2n-6}.
\end{align*}
The term $-c_3 l_3(\text{ND})-c_2 l_2(\text{ND})-c_1 l_1(\text{ND}) - c_0 l_0(\text{ND})$ is given by:
\begin{align*}
(d_n + q^{n-3}(-1)^n) l_3 & (\text{ND}) - (d_nq^{n-3}(-1)^n-q^{2n-4}) d_{n-1} = \\
& = (d_n + q^{n-3}(-1)^n) (l_3(\perp) - q^{2n-4}) - (d_nq^{n-3}(-1)^n-q^{2n-4}) d_{n-1} \\
& = d_n l_3(\perp) - q^{2n-4}d_n + q^{n-3}(-1)^n(d_n d_{n-1} + q^{3n-8}(-1)^n - q^{2n-6})\\
&\qquad - (d_nq^{n-3}(-1)^n-q^{2n-4}) d_{n-1}\\
& = d_n l_3(\perp) - q^{2n-4}(d_n - d_{n-1}) + q^{n-3}(-1)^n( q^{3n-8}(-1)^n - q^{2n-6})\\
& = d_n l_3(\perp) - q^{2n-4}(d_n - d_{n-1}) + q^{2n-6}( q^{2n-5} - q^{n-3}(-1)^n)\\
& = l_4(\text{ND}).
\end{align*}
\end{proofofcase}

\noindent\textbf{Case 4.}  $l_4(\text{D}) = - \sum_{k=0}^3 c_k l_k(\text{D})$, that is, $\sum_{k=0}^4 c_k l_k(\text{D}) = 0$.

\begin{proofofcase}[Proof of Case 4:]
By Theorem \ref{theoremwalks4}:
\begin{align*}
l_4(\text{D}) & = l_3(\perp)d^1_{n-1} + l_3(\text{ND}) \left(d_n-d^1_{n-1}\right)+ q^{2n-5} (l_3(\text{D})-l_3(\text{ND}))\\
& = (l_3(\text{ND}) + q^{2n-4}) d^1_{n-1} + l_3(\text{ND}) \left(d_n-d^1_{n-1}\right)+ q^{2n-5} q^{2n-6}\\
& = d_n l_3(\text{ND}) + d^1_{n-1} q^{2n-4} + q^{4n-11}.
\end{align*}
On the other hand, $-c_3 l_3(\text{D})-c_2 l_2(\text{D})-c_1 l_1(\text{D}) - c_0 l_0(\text{D})$ is given by:
\begin{align*}
(d_n + q^{n-3}&(-1)^n) l_3(\text{D}) - (d_n q^{n-3}(-1)^n - q^{2n-4}) d^1_{n-1}\\
& = (d_n + q^{n-3}(-1)^n)(l_3(\text{ND}) + q^{2n-6}) - (d_n q^{n-3}(-1)^n - q^{2n-4}) d^1_{n-1}\\
& = d_n l_3(\text{ND}) + d^1_{n-1} q^{2n-4} + d_n q^{2n-6}  + q^{n-3}(-1)^n(d_n d_{n-1} + q^{3n-8}(-1)^n)\\
& \qquad - d_n d^1_{n-1} q^{n-3}(-1)^n\\
& = l_4(\text{D}) + d_n q^{2n-6} + d_n d_{n-1} q^{n-3}(-1)^n - d_n d^1_{n-1} q^{n-3}(-1)^n\\
& = l_4(\text{D}) + d_n q^{n-3}(-1)^n \left( q^{n-3}(-1)^n + d_{n-1} - d^1_{n-1}\right)\\
& = l_4(\text{D}).
\end{align*}
\end{proofofcase}

Up to this point, we have proved that $m(n,q)$ is a monic polynomial that annihilates the adjacency matrix of $\G(V)$.
As a consequence, the minimal polynomial divides $m(n,q)$.
Moreover, $\mu_1=d_n$, the degree of the graph, is a root of the minimal polynomial.

On the other hand, recall that (see \cite[Lemma 8.12.1]{GR}) a connected graph of diameter $d$ has at least $d+1$ eigenvalues.
Hence, by Lemma \ref{walksLength1and2} and Theorem \ref{theoremwalks3Dim3orMore} we have
\begin{itemize}
    \item if $n\geq 4$, $\G(V)$ has at least $3$ eigenvalues, 
    \item if $n= 3$ with $q > 2$, $\G(V)$ has at least $4$ eigenvalues.
\end{itemize}
Therefore $m(n,q)$ is the minimal polynomial of $A_n$ when $n=3$ and $q > 2$.

Suppose that $n \geq 4$ and $q>2$.
Then $\G(V)$ is connected and the number of common neighbours between two different non-adjacent vertices is given by the values $l_2(\text{ND}) = d_{n-1}$ and $l_2(\text{D}) = d^1_{n-1}$ in view of Lemma \ref{walksLength1and2}.
Since both cases arise and $d_{n-1}\neq d^1_{n-1}$, we see that $\G(V)$ is a connected regular graph which is not strongly regular. 
A classical characterization \cite[Lemma 10.2.1]{GR} of strongly regular graphs shows then that the adjacency matrix of a connected regular but not strongly regular graph must have at least $4$ eigenvalues.
This shows that $m(n,q)$ is the minimal polynomial when $n \geq 3$ and $q > 2$.

We move now to the case $q = 2$ and $n \geq 4$.
Here the case $l_2(\text{ND})$ does not arise by Lemma \ref{remarkCaseqEquals2}.
Since the number of common neighbors of two adjacent vertices is
always $d_{n-1}$ it follows that $\G(V)$ is connected regular graph of diameter $2$ which is also strongly regular, so by \cite[Section 10.2]{GR} its adjacency matrix has exactly $3$ eigenvalues given by:
\begin{align*}
    \lambda_1 & = \text{ the degree } = d_n,\\
    \lambda_2 & =  \frac{1}{2}\left( d_{n-1}-d^1_{n-1} + \sqrt{(d_{n-1}-d^1_{n-1})^2+4(d_n-d^1_{n-1})}\right)\\
    & =  \frac{1}{2}\left( -q^{n-3}(-1)^n + \sqrt{q^{2n-6}+4(q^{2n-4}-q^{2n-5})}\right)\\
    & = \begin{cases}
    q^{n-3} & n \text{ even,}\\
    q^{n-2} & n \text{ odd.}
    \end{cases}\\
    \lambda_2 & =  \frac{1}{2}\left( d_{n-1}-d^1_{n-1} - \sqrt{(d_{n-1}-d^1_{n-1})^2+4(d_n-d^1_{n-1})}\right)\\
    & =  \frac{1}{2}\left( -q^{n-3}(-1)^n - \sqrt{q^{2n-6}+4(q^{2n-4}-q^{2n-5})}\right)\\
    & = \begin{cases}
    -q^{n-2} & n \text{ even,}\\
    -q^{n-3} & n \text{ odd.}
    \end{cases}
\end{align*}

When $n = 3$ and $q = 2$, the computation is straightforward from Lemma \ref{walksLength1and2}, Lemma \ref{remarkCaseqEquals2} and the identities in Appendix \ref{appendixCountingVectorsAndSubspaces}.

\end{proof}

Next, we compute the multiplicities of the eigenvalues.

\begin{corollary}
\label{multiplicitiesEigenvaluesDimAtLeast3}
Let $V$ be a unitary space of dimension $n\geq 3$ over $\GF{q^2}$.
Let
\[ \alpha_2:= \frac{d_nd_{n+1}}{2 q^{2n-3}}\frac{ (q^2-q-1-(-1)^n)}{(q-1)}; \quad \alpha_3:= \frac{d_n d_{n-1}}{ q^{2n-8}(q-1)},\]
\[\alpha_4:=\frac{d_n d_{n+1}}{2 q^{2n-3}} \frac{(q^2-q - 1 + (-1)^n)}{(q-1)}.\]
Then the multiplicities of the eigenvalues of the adjacency matrix of $\G(V)$ are given in Table \ref{tab:multiplicities}.

\begin{table}[ht]
    \centering
    \begin{tabular}{|c|c|c|c|c|}
    \hline
         & $\mu_1 = d_n$ & $\mu_2 = q^{n-2}$ & $\mu_3 = q^{n-3}(-1)^n$ & $\mu_4 = -q^{n-2}$ \\
    \hline
        $q\neq 2, n\geq 3$ & $1$ & $\alpha_2$ & $\alpha_3$ & $\alpha_4$\\
        $q = 2, n$ even & $1$ & $0$ & $\alpha_3$ & $\alpha_2+\alpha_4$\\
        $q = 2, n>3$ odd & $1$ & $\alpha_2+\alpha_4$ & $\alpha_3$ & $0$\\
        $q = 2, n = 3$ & $4=1 + \alpha_2+\alpha_4$ & $0$ & $8 = \alpha_3$ & $0$\\
    \hline
    \end{tabular}
    
    \medskip
    
    \caption{Multiplicities of the eigenvalues of $\G(V)$ for $n\geq 3$.}
    \label{tab:multiplicities}
\end{table}



\end{corollary}

\begin{proof}
Suppose first that $q\neq 2$.
By Theorem \ref{theoremEigenvalues}, the eigenvalues are:
\[\mu_1 = d_n,\quad \mu_2=q^{n-2},\quad \mu_3=q^{n-3}(-1)^n = \mu_2 q^{-1}(-1)^n,\quad \mu_4 =-q^{n-2} = -\mu_2.\]
Write $\mu :=q^{n-2}=\mu_2$.
Let $a_i$ be the multiplicity of $\mu_i$.
Note that $a_1=1$ since $\G(V)$ is connected by Theorem \ref{connectedFramePoset}.
Let $A_n$ be the adjacency matrix of $\G(V)$.
Then the following relations show that $a_i = \alpha_i$ for $i=2,3,4$:
\begin{equation*}
\label{eqMultiplicities}
    \begin{cases}
    1 + a_2 + a_3 + a_4 = \dim(A_n) = d_{n+1},\\
    d_n + a_2 \mu + a_3 \mu q^{-1}(-1)^n + a_4(-\mu) = \Tr(A_n) = 0,\\
    d_n^2 + a_2\mu^2 + a_3\mu^2q^{-2} + a_4\mu^2 = \Tr(A_n^2) = d_{n+1}d_n.
    \end{cases}
\end{equation*}
We leave the simple but tedious calculation proving that the solution takes
the asserted form to the reader.

The case $q = 2$ follows from an analogous computation.

\end{proof}

\begin{remark}
Since $\GU(V)$ acts on the eigenspaces of the adjacency matrix, we get representations of the unitary groups on complex vector spaces whose dimensions are given in the previous theorem.
It would be interesting to analyze which irreducible representations arise.
\end{remark}

\subsection{The low dimensional case $n=2$}

For the sake of completeness, we include the computation of the eigenvalues of the graph  $\G(V)$ when $V$ is a unitary space of dimension $n=2$ over $\GF{q^2}$.
In this case, the eigenvalues are $\mu_1 = 1$ and $\mu_4=-1$.
The conclusion is straightforward from the fact that $\G(V)$ has $q(q-1)/2$ connected components, each of them consisting of a single edge.

\begin{proposition}
\label{walksAndEigenvaluesLowDim2}
Let $V$ be a unitary space of dimension $n=2$ over $\GF{q^2}$.
Then Table \ref{tab:caseDimension2} records the number of walks between vertices $S,W\in\G(V)$.
\begin{table}[ht]
    \centering
    \begin{tabular}{|c|c|c|c|}
    \hline
                & $l_0(S,W)$ & $l_1(S,W)$ & $l_2(S,W)$ \\
    \hline
        $S = W$ & $1$        & $0$        & $1$ \\
    $S \perp W$ & $0$        & $1$        & $0$ \\
$S \neq W$, $S\not \perp W$ & $0$        & $0$        & $0$ \\
    \hline
    \end{tabular}
    
    \medskip
    
    \caption{Counting of walks in dimension $n = 2$.}
    \label{tab:caseDimension2}
\end{table}
In particular, we see that the minimal polynomial $m(2,q)$ of the adjacency matrix of $\G(V)$ is \[m(2,q) = X^2-1 = (X-\mu_1)(X-\mu_4),\]
where $\mu_1=d_n=d_2=1$ and $\mu_4=-1$.

Moreover, the multiplicity of $\mu_1$ equals the number of connected components, which is $q(q-1)/2$,
so the multiplicity of $\mu_4$ is $d_3 - q(q-1)/2 = q(q-1)/2$.
\end{proposition}

\begin{proof}
These computations are straightforward.
\end{proof}

\section{Application of Garland's method}
\label{sec:garland}

In this section we use the computation of eigenvalues of $\G(V)$ from Theorem \ref{theoremEigenvalues}, along with the list of connected cases from Theorem \ref{connectedFramePoset}, to deduce vanishing results about the homology groups with coefficients in a field $\kk$ of characteristic $0$ of the frame complex $\F(V)$ in the case of a finite field $\KK = \GF{q^2}$.
We will arrive at these conclusions by applying Garland's method.

We begin by stating the version of Garland's method given in \cite{BS} (the original idea can be found in \cite{Garland}).

\begin{theorem}
[{Garland's method}]
\label{theoremGarland}
Let $K$ be a $d$-dimensional simplicial complex and let $0 \leq i \leq  d-1$ be fixed.
Suppose that for every $(i-1)$-simplex $\sigma$ of $K$ we have that:
\begin{itemize}
\item[(G1)] $\Lk_K(\sigma)$ is connected of dimension $\geq 1$;
\item[(G2)] if $\lambda_{\min}$ denotes the smallest non-zero eigenvalue of the normalized Laplacian of the $1$-skeleton of $\Lk_K(\sigma)$, then $\lambda_{\min} >\frac{i}{i+1}$.
\end{itemize}
Then $\tilde{H}_i(K;\kk) = 0$ for any field $\kk$ of characteristic $0$.
\end{theorem}

Here by normalized Laplacian of a graph $\G$ we mean the matrix $\Id - D^{-1} A$ with rows and columns indexed by the
vertices of $\G$, where $D$ is the 
diagonal matrix with diagonal given by the degrees of the vertices of $\G$, and $A$ is the adjacency matrix of $\G$.

\bigskip

We apply Garland's method to the frame complex $\F(V)$ of a unitary space $V$ of dimension $n$ over some $\GF{q^ 2}$.
Denote by $A_n$ the adjacency matrix of $\G(V)$ and let $D_n=d_n\Id$.

From now on, we assume that $n = \dim(V) \geq 3$, and $q > 2$ (we treat the case $q = 2$ in the second part of this section).

Let $\sigma\in \F(V)$ be an $i$-frame for some $i\leq n-3$.
Then $\gen{\sigma}^\perp$ has dimension $n-i\geq 3$, and by Theorem \ref{connectedFramePoset} it follows that the link 
$\Lk_{\F(V)}(\sigma) \cong \F(\gen{\sigma}^\perp)$ is connected of dimension $\geq 1$. This shows that
condition (G1) of Theorem \ref{theoremGarland} is satisfied.

As a further consequence, the eigenvalues of the adjacency matrix of $\G(\gen{\sigma}^\perp)$ coincide with the eigenvalues of $A_{n-i}$.
By Theorem \ref{theoremEigenvalues}, these are: $d_{n-i}$, $q^{n-i-2}$, $(-1)^{n-i}q^{n-i-3}$, $-q^{n-i-2}$.

On the other hand, the eigenvalues of the normalized Laplacian $\Id -D_{n-i}^{-1}A_{n-i}$, are of the form $\lambda = 1 - (d_{n-i})^{-1}\mu$, for $\mu$ in the spectrum of $A_{n-i}$.
We see then that $\lambda$ takes the following values:
\[ 0,\ 1-(d_{n-i})^{-1}q^{n-i-2},\ 1-(d_{n-i})^{-1}(-1)^{n-i}q^{n-i-3},\ 1+(d_{n-i})^{-1}q^{n-i-2}.\]
Elementary arguments yield that the smallest positive eigenvalue $\lambda_{\min}$ is
\begin{align*}
    1-(d_{n-i})^{-1}q^{n-i-2} & = 1 - \frac{(q+1) q^{n-i-2}}{q^{n-i-2}(q^{n-i-1}-(-1)^{n-i-1})}\\
    & = 1-\frac{q+1}{q^{n-i-1}-(-1)^{n-i-1}}.
\end{align*}
Now $\lambda_{\min}$
satisfies (G2) 
if and only if:
\[  1 - \frac{q+1}{q^{n-i-1}-(-1)^{n-i-1}} = \lambda_{\min}> \frac{i}{i+1} = 1 - \frac{1}{i+1},\]
which is equivalent to
\begin{equation*}
     \frac{q^{n-i-1}-(-1)^{n-i-1}}{q+1}-i-1 > 0.
\end{equation*}

It is straightforward to verify that the left-hand side of the preceding inequality is indeed monotone in $i$.

\begin{lemma}
[{Monotone-bound}]
\label{lemmaMonotoneBound}
Let $0\leq i < j \leq n-3$ and $q\geq 2$.
Then
\begin{equation}
    \label{monotoneCondition}
    \frac{q^{n-i-1} - (-1)^{n-i-1}}{q+1} - i - 1 > \frac{q^{n-j-1} - (-1)^{n-j-1}}{q+1} - j - 1.
\end{equation}

\end{lemma}



The following lemma now summarizes what the argument above shows.

\begin{lemma}
\label{lem:garlandconsequence}
Let $V$ be a unitary space of dimension $n$
over $\GF{q^2}$ for some $q \neq 2$ and $\kk$ a field of
characteristic $0$. Then 
\begin{itemize}
    \item[(i)] For $0 \leq i \leq n-3$ we have that $\F(V)$ is $\kk$-homologically $i$-connected if
\begin{equation*}
    \label{generalBound}
     \frac{q^{n-i-1}-(-1)^{n-i-1}}{q+1}-i-1 > 0.
\end{equation*}
   \item[(ii)] For $3 \leq j \leq n$ we have that $\F(V)$ is
   $\kk$-homologically $(n-j)$-connected if 
  \begin{equation}
\label{altGeneralBound}
    P_j(q) := \frac{q^{j-1} - (-1)^{j-1}}{q+1} + j - 1 > n.
\end{equation}
\end{itemize}
\end{lemma}
\begin{proof}
   Part (i) follows from Lemma \ref{lemmaMonotoneBound} together with the arguments preceding the lemma.
   Part (ii) follows directly from (i) when replacing $i$ by
   $n-j$.
\end{proof}



Theorem \ref{coroHomologyCohenMacaulay} is now a special case.

\begin{proof}[Proof of Theorem \ref{coroHomologyCohenMacaulay}]
Applying Lemma \ref{lem:garlandconsequence} to $j=n-3$ shows that $\F(V)$ is $\kk$-homo\-lo\-gically $(n-3)$-connected if $n < q+1$. 

It remains to show that $\widehat{\F}(V)$ is a Cohen-Macaulay poset over $\kk$. This follows from the first part of the
theorem together with Lemma \ref{lemmaHatPoset} applied to 
open intervals in $\widehat{\F}(V)$. 
\end{proof}

We now state another set of homological connectivity results:

\begin{theorem}
\label{thm:otherconresults}
Let $V$ be a unitary space of dimension $n$
over $\GF{q^2}$ and $\kk$ a field of
characteristic $0$. Then:

\begin{itemize}
\item[(i)] If $2\leq n < q+1$, then $\F(V)$ is $\kk$-homologically $(n-3)$-connected.

\item[(ii)] If $q\geq 3$, $4 \leq n < q^2-q+4$ then 
$\F(V)$ is $\kk$-homologically $(n-4)$-connected.

\item[(iii)] If $n \geq 4$ and $q \geq 3$ then:
\begin{itemize}
    \item[(a)] If $n\in \{4,5,6\}$ and $n\leq q$, then $\F(V)$ is $\kk$-homologically $(n-3)$-connected.
    
    \item[(b)] If $n\in \{4,5,6\}$ and $q < n$, then $\F(V)$ is $\kk$-homologically $(n-4)$-connected.
    
    \item[(c)] If $n\geq 7$, then $\F(V)$ is $\kk$-homologically $[n/2]$-connected.
\end{itemize}

\end{itemize}
\end{theorem}

\begin{proof}

\begin{itemize}
    \item[(i)] If $n = 2$, then $\F(V)$ is non-empty, and in particular $(-1)$-connected for any $q\geq 2$.
    For $n \geq 3$, and hence also $q\geq 3$, the claim is part of
Theorem \ref{coroHomologyCohenMacaulay}.
\item[(ii)]
This is an immediate consequence of Lemma \ref{lem:garlandconsequence}(ii) for $j = 4$.
\item[(iii)] 
Claims (a) and (b) are straightforward consequences of (i) and (ii). 

Next consider claim (c) and hence assume that $n\geq 7$.
Then it is routine to check that $[n/2] \leq n-3$.
Write $n = 2n'+e$, where $e \in \{0,1\}$.
So $[n/2] = n'$ and $j:=n - [n/2] = n' + e$.
Then, in view of the discussion in Lemma \ref{lem:garlandconsequence}, we show that \eqref{altGeneralBound} holds.
We compute:
\begin{align*}
    \frac{q^{j-1} - (-1)^{j-1}}{q+1} + j - 1 - n  & = \frac{q^{n'+e -1} + (-1)^{n'+e}}{q+1} + n'+e -1 - n\\
    & \geq \frac{q^{n' + e-1} - 1}{q+1} - n' - 1\\
    & = \frac{q^{n' + e-1} - 1 - (q+1)(n'+1)}{q+1}.
\end{align*}
To conclude the proof, we show that the numerator of the latter fraction
\begin{equation}
\label{positiveNumerator}
    q^{n' + e-1} - 1 - (q+1)(n'+1)
\end{equation}
is positive.
Define $h(x) = q^x-1 - (q+1)(x+2)$.
From basic analysis it can be seen that $h(x) > 0$ for any $x\geq 3$ and $q\geq 3$.
Let $x := n'+e -1$, which is at least $3$ by the
condition $n\geq 7$.
Then:
\[ q^{n'+e-1} - 1 -(q+1)(n'+1) = q^{x} - 1 -(q+1)(x+2-e) = h(x) + e (q+1) > 0.\]
Therefore, \eqref{positiveNumerator} is positive for any $q\geq 3$.
This implies that \eqref{altGeneralBound} holds, so Lemma \ref{lem:garlandconsequence}(ii) applies.
\end{itemize} 
\end{proof}

\begin{remark}
Indeed, if $q\geq 11$ and $n\geq 8$, then $\F(V)$ is $\kk$-homologically $(n-[n/4]-1)$-connected.
This follows by direct application of the bound (\ref{altGeneralBound}).

Write $n = 4n'+e$, $e\in \{0,1,2,3\}$, and let $j = [n/4]+1=n'+1$.
Then, for $n'\geq 2$, we have:
\begin{align*}
    \frac{q^{j-1}-(-1)^{j-1}}{q+1}+j-1-n & = \frac{q^{n'}-(-1)^{n'}}{q+1}-3n'-e\\
    & \geq q^{n'-2}(q-1)-3n'-3\\
    & \geq 11^{n'-2}10-3n'-3 >0.
\end{align*}
\end{remark}

It would be interesting to see if this observed
$\kk$-homological connectivity lifts to the homotopy level.
For example, if $n\geq 7$ and $q\geq 3$, we wonder if it is true that $\F(V)$ is $[n/2]$-connected.

Note that when $q = 2,3$ and $n = 4$, we have $\pi_1(\F(V))\neq 1$ and indeed $H_1(\F(V),\kk)\neq 0$.
So the excluded cases in the above theorem constitute exceptions, and the bound given by Garland's method is sharp there.
We did not check the other cases, but we expect to encounter similar situations.

A very interesting scenario arises when $n = 7$ and $q = 3$:

\begin{example}
Let $n = 7$, $q = 3$ and $\kk$ a field of characteristic $0$.
Here the reduced Euler characteristic of $\F(V)$ is $1582997389326080$.
Since Theorem \ref{thm:otherconresults}(iii) yields  $\kk$-homological $[n/2] = 3$-connectivity for $\F(V)$
and fields $\kk$ of characteristic $0$, we see that a possible contribution to non-zero free homology has to come from the degrees $4$ and $5$.
Indeed the positivity of the Euler characteristic shows that $H_4(\F(V),\kk)\neq 0$.
Nevertheless, it is not know whenever $H_5(\F(V),\kk)$ vanishes or not.
\end{example}

\begin{example}
Let $n = 6$, $q = 3$ and $\kk$ a field of characteristic $0$.
Now $\tilde{\chi}(\F(V)) = 19557643832$.
Since by Theorem \ref{thm:otherconresults}(ii) we get that $\F(V)$ is $2$-connected, we see that a contribution to the positivity of the Euler characteristic can only come from the homology group of degree $n-2=4$ (and indeed a ``negative'' contribution can only come from $H_3(\F(V),\kk)$.)
Then we conclude that $H_4(\F(V),\kk) \neq 0$.
But we cannot decide if $H_3(\F(V),\kk)$ vanishes or not.
\end{example}


Now we deal with the case $q = 2$.
In order the emphasize the structure of the polynomials in $q = 2$ that will appear, we will always write $q$ instead of $2$.

Recall that, since $q=2$, $\F(V)$ has the homotopy type of a complex of dimension at most $n-3$.
We assume now that $n\geq 4$.
The eigenvalues for the adjacency matrix are $d_n, (-1)^n q^{n-3}, -(-1)^nq^{n-2}$.
Let $0\leq i \leq n-4$, and the link of an $i$-frame corresponds to the frame complex over a unitary space of dimension $n-i\geq 4$ over $\GF{q^2}$.
Now the eigenvalues of the normalized Laplacian of the $1$-skeleton of the link of an $i$-frame are
\[ 0, 1 - (d_{n-i})^{-1} (-1)^{n-i} q^{n-i-3}, 1 + (d_{n-i})^{-1} (-1)^{n-i} q^{n-i-2}.\]
The smallest non-zero eigenvalue above now depends on the parity of $n-i$.
So we divide the analysis in cases:

\medskip

\noindent\textbf{Case $n-i$ even.}

Here the smallest eigenvalue for the link is 
\[ 1 - (d_{n-i})^{-1} (-1)^{n-i} q^{n-i-3}  = 1 - (d_{n-i})^{-1} q^{n-i-3} = 1 - \frac{(q+1)}{q(q^{n-i-1} + 1)} .\]
And we should ask this eigenvalue to be strictly greater than $\frac{i}{i+1} = 1 - \frac{1}{i+1}$.
This bound holds if and only if
\begin{equation*}
     \frac{q(q^{n-i-1}+1)}{q+1} - i - 1 > 0 .
\end{equation*}

\medskip

\noindent\textbf{Case $n-i$ odd.}

Now the smallest eigenvalue for the link is 
\[ 1 + (d_{n-i})^{-1} (-1)^{n-i} q^{n-i-2}  = 1 - (d_{n-i})^{-1} q^{n-i-2} = 1 - \frac{q+1}{q^{n-i-1} -1} .\]
Then this eigenvalue is strictly greater than $\frac{i}{i+1} = 1 - \frac{1}{i+1}$ if and only if 
\begin{equation*}
     \frac{q^{n-i-1}-1}{q+1}  - i - 1 > 0 .
\end{equation*}

For $q = 2$, $n\geq 4$ and $0\leq i\leq n-4$, define
\begin{equation*}
    Q_n(q,i) := \begin{cases}
         \frac{q(q^{n-i-1}+1)}{q+1} - i - 1 & n-i \text{ is even},\\
         \frac{q^{n-i-1}-1}{q+1}  - i - 1 & n-i \text{ is odd}.
    \end{cases}
\end{equation*}

We get an analogous monotone-bound as in the case $q\geq 3$:

\begin{lemma}
\label{lemmaMonotoneBound2Case}
Let $q = 2$, and let $0\leq i \leq j \leq n-4$.
Then $Q_n(q,i) \geq Q_n(q,j)$.

In particular, if $Q_n(q,j) > 0$ then $\F(V)$ is $\kk$-homologically $j$-connected.
\end{lemma}

\begin{proof}
The consequence about homological connectedness 
follows from the monotony of the bound and from Garland's Theorem \ref{theoremGarland}.

Now we show the monotony condition $Q_n(q,i) \geq Q_n(q,j)$, or equivalently that $Q_n(q,i) - Q_n(q,j)\geq 0$.
Note that it is enough to consider the case $i = j-1$.
Then we divide in cases:

\medskip

\noindent\textbf{Case 1:} if $n-j$ is even, then $Q_n(q,j) = Q_n(q,j+1)$.
\begin{proofofcase}[Proof of Case 1.]
Then $n-i=n-j+1$ is odd, and so:
\begin{align*}
    Q_n(q,i) - Q_n(q,j) & = Q_n(q,j-1) - Q_n(q,j)\\
    & = \frac{q^{n-j+1-1}-1}{ q+1} - (j-1)-1 - \left( \frac{q(q^{n-j-1}+1)}{q+1} - j-1 \right)\\
    & = \frac{q^{n-j} -1 - q^{n-j} - q + (q+1) }{q+1}\\
    & = 0.
\end{align*}
So $Q_n(q,i) - Q_n(q,j) = 0$ for this case.
\end{proofofcase}


\noindent\textbf{Case 2:} if $n-j$ is odd, then $Q_n(q,j) < Q_n(q,j+1)$.

\begin{proofofcase}[Proof of Case 2.]
Now $n-i=n-j+1$ is even.
We compute:
\begin{align*}
    Q_n(q,i) - Q_n(q,j) & = Q_n(q,j-1) - Q_n(q,j)\\
    & = \frac{q(q^{n-(j-1)-1}+1}{q+1} - (j-1)-1 - \left( \frac{q^{n-j-1}-1}{q+1} - j-1 \right)\\
    & = \frac{q^{n-j+1} + q - q^{n-j-1} + 1 + (q+1) }{q+1}\\
    & = q^{n-j-1} + 2 > 0.
\end{align*}
So $Q_n(q,i) - Q_n(q,j) > 0$ for this case.
\end{proofofcase}

This finishes the proof of the lemma.
\end{proof}

\begin{remark}
\label{remarkBoundPolynomial2Case}
We see that if we instead take $j = n-i$, $4\leq j\leq n$, then the bound $Q_n(q,n-j) > 0$ is equivalent to:
\begin{equation*}
    \begin{cases}
    \frac{q(q^{j-1}+1)}{q+1} + j-1 > n & \text{ if $j$ is even,}\\
    \frac{q^{j-1}-1}{q+1} + j-1 > n & \text{ if $j$ is odd.}
    \end{cases}
\end{equation*}
Analogously, we define $P_j(q) := Q_n(q,n-j) + n$.
Replacing with $q = 2$, we get:
\begin{equation*}
P_j(2) = \begin{cases}
\frac{2^j-1}{3} + j & \text{ if $j$ is even,}\\
\frac{2^{j-1}-1}{3} + j-1 & \text{ if $j$ is odd.}
\end{cases}
\end{equation*}
Note that, if $j$ is even, then $P_j(2) = P_{j+1}(2)$ in view of the above equation (see Case 1 in the proof of Lemma \ref{lemmaMonotoneBound2Case}).

Some quick computations give:
\begin{align*}
    P_4(2) = P_5(2) & = 9,\\
    P_6(2) = P_7(2) & = 27,\\
    P_8(2) = P_9(2) & = 93,\\
    P_{10}(2) = P_{11}(2) & = 351.
\end{align*}
\end{remark}

The proof of the following theorem is now straightforward from the above computations.

\begin{theorem}
\label{coroGarlandCase2}
Let $q = 2$, $n\geq 4$, $V$ a unitary space of dimension $n$ over $\GF{2^2}$, and $\kk$ a field of characteristic $0$.
Let $4\leq j\leq n$.
Then $\F(V)$ is $\kk$-homologically $(n-j)$-connected if $P_j(2) = \frac{2^{j}-1}{3} + j > n$ and $j$ is even.
If $j$ is odd and $P_j(2) = P_{j-1}(2) > n$ then $\F(V)$ is $\kk$-homologically $(n-j+1)$-connected.

In particular, if $n\geq 11$ or $n=7,8$, then $\F(V)$ is $\kk$-homologically $[n/2]$-connected.
\end{theorem}

In the following example, we give a concrete application of Garland's method in the frame complex $\F(V)$ for $q=2$ and $n=6$.
In this case, the degree $1$ homology group has trivial free part by Garland's Theorem, but a non-zero torsion part.

\begin{example}
Recall from Theorem \ref{simplyConnectedFramePoset} that if $q=2$ and $n = 6$, then $\pi_1(\F(V)) = C_2\times C_2$.
Here, for $j = 4$, $P_j(2) = 9 > 6$, so $\F(V)$ is $\kk$-homologically $(n-j) = 2$-connected by Corollary \ref{coroGarlandCase2}.
That is, $H_1(\F(V),\kk) = 0 = H_2(\F(V),\kk)$ for any field $\kk$ of characteristic $0$.
However, we have $H_1(\F(V),\ZZ) = \ZZ_2\oplus \ZZ_2 \neq 0$.
This shows that Garland's method definitely cannot rule out torsion in homology.
\end{example}

\section{Further applications}
\label{sec:further}

In this section, we give some applications of the results on the frame complex to the homology and homotopy type of the poset of non-degenerate subspaces, the poset of orthogonal decompositions, and the Quillen poset of non-trivial elementary abelian $p$-subgroups of a unitary group.
At the end of the section, we include also some additional comments on related applications.
From now on, $\kk$ will denote a field of characteristic $0$.

\subsection{Poset of non-degenerate subspaces}
In this subsection, we apply the result of the previous section to the poset $\redS(V)$ of proper non-zero non-degenerate subspaces of a unitary space $V$.

\begin{proposition}
\label{coro:applicationOtherPosets}
Let $V$ a unitary space of dimension $n\geq 2$ over $\GF{q^2}$.
The following assertions hold:
\begin{itemize}
    \item[(i)] If $n < q+1$ then $\redS(V)$ is Cohen-Macaulay over $\kk$.
    
    \item[(ii)] If $q+1 \leq n < 2 (q+1)$, then $\redS(V)$ is $\kk$-homologically $(n-4)$-connected. 
    
    \item[(iii)] If $q\geq 3$ and $n \geq 7$ then $\redS(V)$ is $\kk$-homologically $[n/2]$-connected.
\end{itemize}
\end{proposition}

\begin{proof}
If $q = 2$, item (i) is clear and item (ii) follows from Corollaries \ref{corollarySVConnected} and  \ref{corollarySVSimplyConnected}.
Therefore, in the remaining of the proof we suppose that $q\geq 3$.
Note that item (i) follows from Theorem \ref{coroHomologyCohenMacaulay} and Proposition \ref{prop:CMimplicationFrames}.

These assertions (ii) and (iii) will follow by a standard application of Quillen's fiber-Theorem to the map $\phi:\F(V)^{n-2}\to \redS(V)$ that sends a frame to its span, in conjunction with Theorem \ref{thm:otherconresults}.


Assume that $q+1 \leq n < 2(q+1)$.
Note that in particular $n < P_4(q)$.
If $s = \dim(S)$ for $S\in \redS(V)$, then either $s < q+1$ or $n-s < q+1$.
In the former case, $\F(S)$ is $\kk$-homologically $(s-3)$-connected and $\redS(S^\perp)$ is at least $\kk$-homologically $(n-s-4)$-connected by induction (since $(n-s) < 2(q+1)$).
Thus the join $\F(S) * \redS(S^\perp)$ is $\kk$-homologically $(s-3)+(n-s-4)+2=(n-5)$-connected.
Now suppose that $n-s<q+1$.
Then $s < n < 2(q+1)$, so $s < 2(q+1) \leq P_4(q)$ for any $q$.
Then, $\F(S)$ is $\kk$-homologically $(s-4)$-connected while $\redS(S^\perp)$ is $\kk$-homologically $(n-s-3)$-connected by item (1).
Thus the join $\F(S) * \redS(S^\perp)$ is $\kk$-homologically $(n-5)$-connected.
By Quillen's fiber-Theorem, $\phi$ is a $\kk$-homologically $(n-4)$-equivalence, and since $\F(V)$ is $\kk$-homologically $(n-4)$-connected, we conclude that $\redS(V)$ is $\kk$-homologically $(n-4)$-connected.
This finishes the proof of  item (ii).

The proof of item (iii) follows similar ideas, but it involves a lengthy case-by-case analysis.
We sketch the main argument of the proof.

The aim is to show that the join $\F(S) * \redS(S^\perp)$ is at least $\kk$-homologically $([n/2]-1)$-connected, and then conclude as above by Quillen's fiber-Theorem that $\phi$ is a $\kk$-homologically $[n/2]$-equivalence.
In conjunction with Theorem \ref{thm:otherconresults}(iii), this will provide the desired conclusion.

We apply Theorem \ref{thm:otherconresults}(iii) to $\F(S)$ when $s:=\dim(S) \geq 7$ to get $[s/2]$-connectivity on $\F(S)$, and induction on $\redS(S^\perp)$ when $n-s \geq 7$.
It can be seen that $[s/2] + [(n-s)/2] \geq [n/2]-1$, so the join $\F(S) *  \redS(S^\perp)$ is at least $\kk$-homologically $([n/2]+1)$-connected when $s,n-s\geq 7$.
However, when $s \leq 6$ or $n-s\leq 6$, we have to look for alternative ways of showing connectivity.
We proceed by applying (i) and (ii) of this theorem and Theorem \ref{thm:otherconresults}.

If $s\leq 6$ and $n-s\leq 6$, we perform a case-by-case analysis.
Note that if $q\geq 7$, then $s,n-s < q+1$, so $\F(S)*\redS(S^\perp)$ is $\kk$-homologically $(s-3)+(n-s-3)+2=n-4$ connected, and $n-4 \geq [n/2]$ by the hypothesis $n\geq 7$.
So the case-by-case inspection should be carried out for $q =3,4,5$.
The idea is to check if the bounds $s,n-s < q+1$ and $s,n-s <2(q+1)$ hold in order to apply (i) and (ii) of this theorem and Theorem \ref{thm:otherconresults}(i)
and (ii).
If $q = 4,5$, we will always get that $\F(S) * \redS(S^\perp)$ is at least $\kk$-homologically $([n/2]-1)$-connected.
However, if $q = 3$, this argument may fail only in the case $s=4$ and $n=8$: a priori we just get that $\F(S) * \redS(S^\perp)$ is $\kk$-homologically $2$-connected, but we need to achieve $3$-connected.
Nevertheless, we claim that this is indeed the case.
The classical formula for the homology of a join of spaces yields
\[H_3(\F(S)*\redS(S^\perp),\kk) \groupiso H_1(\F(S),\kk)\otimes H_1(\redS(S^\perp),\kk), \]
since both $\F(S)$ and $\redS(S^\perp)$ are connected.
We have that $\pi_1(\redS(S^\perp)) \groupiso C_3\times C_3$ by the main result of \cite{Horn}, so $H_1(\redS(S^\perp),\kk) = 0$ since $\kk$ is a field of characteristic $0$.
Therefore we conclude that $\F(S)*\redS(S^\perp)$ is $\kk$-homologically $3$-connected, and $3 = [n/2]-1$.

Next, if $s \leq 6$ and $n-s\geq 7$, we apply induction to $\redS(S^\perp)$ and Theorem \ref{thm:otherconresults}(i) and (ii) to $\F(S)$.
Similarly, for $s \geq 7$ and $n-s\leq 6$ we apply Theorem \ref{thm:otherconresults}(iii) to $\F(S)$, and items (i) and (ii) of this theorem to $\redS(S^\perp)$.
In any situation, we will get that $\F(S)*\redS(S^\perp)$ is at least $\kk$-homologically $([n/2]-1)$-connected.
\end{proof}

For $q = 2$, a similar proof as above, with case-by-case inspection in low dimensions and invoking Lemma \ref{lemmaMonotoneBound2Case} and Quillen's fiber-Theorem, shows that:

\begin{corollary}
Let $V$ be a unitary space of dimension $n$ over $\GF{2^2}$. Then:
\begin{enumerate}
    \item If $n \in \{ 5,6 \}$, $\redS(V)$ is $\kk$-homologically $1$-connected.
    \item If $n\in \{7,8,9,10\}$, $\redS(V)$ is $\kk$-homologically $([n/2]-1)$-connected.
    \item If $n\geq 11$, $\redS(V)$ is $\kk$-homologically $[n/2]$-connected.
\end{enumerate}
\end{corollary}

The following remark shows that the bound on $n$ and $q$ for homotopical Cohen-Macau\-layness given in \cite{DGM} can be improved when asking for 
Cohen-Macaulay\-ness over fields $\kk$ of characteristic $0$.
We do not know whether the same conclusions apply for homotopy Cohen-Macaulayness.

Recall from the results of \cite{DGM} that if $2^{n-2}<\frac{q^2}{q+1}$ then $\redS(V)$ is homotopically Cohen-Macaulay.
This inequality is equivalent to:
\begin{equation}
    \label{boundDGM}
    n < \frac{\ln(q^2) - \ln(q+1)}{\ln(2)} + 2.
\end{equation}
Now, from basic analysis, it is not hard to check that for any $q\geq 2$ we have:
\[ \frac{\ln(q^2) - \ln(q+1)}{\ln(2)} + 2 < q + 1.\]
Moreover, the quotient $(q+1)/\left( \frac{\ln(q^2) - \ln(q+1)}{\ln(2)} + 2\right)$ goes to infinity as $q$ goes to infinity.
Therefore the bound provided in Proposition \ref{coro:applicationOtherPosets}(i) covers more cases on Cohen-Macaulayness over $\kk$ than the bound \eqref{boundDGM}.
For example, for $q = n = 4$ or $q = n = 5$, \eqref{boundDGM} fails while $n < q+1$ holds.

\subsection{Poset of elementary abelian $p$-subgroups}
In this subsection we give some applications of our results to achieve $\kk$-homological connectivity 
for two posets:

\begin{itemize}
    \item The poset $\D(V)$ of all orthogonal decompositions of $V$ and $\mathring{\D}(V) = \D(V)-\{V\}$ was introduced in Section \ref{sec:basics}. 
    \item For a prime $p$ and a finite group $G$,
    let $\Ap(G)$ be the \textit{poset of all non-trivial elementary
    abelian $p$-subgroups} of $G$.
\end{itemize}

Even though refinement is the natural order
    on $\D(V)$, in our applications we will often use $\D(V)^ \op$ which is the same poset but with the
    reversed order relation. Of course, the order complexes of the two posets are identical.
    
The next proposition summarizes the relation between these two posets, which was noticed first by Aschbacher-Smith in \cite[Section 4]{AS93}, as well as consequences on the $\kk$-homological connectivity of these posets.

\begin{proposition}
\label{coro:applicationDecompApPosets}
Let $n\geq 3$, $q\geq 2$, and $V$ a unitary space of dimension $n$ over $\GF{q^2}$.
Suppose that $p$ is a prime number dividing $q+1$, and let $Z$ denote the unique central subgroup of order $p$ in $\GU_n(q)$.
Then for a field $\kk$ of characteristic $0$ the following assertions hold:
\begin{itemize}
    \item[(i)] By \cite{AS93}, there is an inclusion map of posets $\iota : \mathring{\D}(V)^{\op} \hookrightarrow \Ap(\GU_n(q))_{>Z}$ together with a poset-retraction $r$ such that $\iota \circ r \geq \Id$ and $r \circ \iota = \Id$.
    In particular, $\iota$ is a homotopy equivalence.
    
    \item[(ii)] If $n < q+1$ then $\D(V)$ and $\Ap(\GU_n(q))_{>Z}$ are Cohen-Macaulay over $\kk$.
    
    \item[(iii)] If $q+1 \leq n < 2 (q+1)$ then $\mathring{\D}(V)$ and $\Ap(\GU_n(q))_{ > Z}$ are $\kk$-homologically $(n-4)$-connected.
    
    \item[(iv)] If $q\geq 3$ and $n\geq 7$ then $\mathring{\D}(V)$ and $\Ap(\GU_n(q))_{ > Z}$ are $\kk$-homologically $[n/2]$-connected.
\end{itemize}
\end{proposition}

\begin{proof}
The proof of (i) is based on the construction given in \cite[pp. 514-515]{AS93}.
We review their argument and fill in the details left to the reader in \cite{AS93}.

Fix a generator $\lambda$ of $Z$.
We have a function $\iota:\mathring{\D}(V)^{\op} \to \Ap(\GU_n(q))_{ > Z}$ that maps an orthogonal decomposition $d = \{V_1, \ldots , V_r\}$ to the elementary abelian $p$-subgroup $A_1 \oplus \ldots \oplus A_r$, where $A_i$ is generated by the order-$p$ unitary-map $a_i$ which acts as multiplication by $\lambda$ on $V_i$ and as the identity on $V_i^{\perp}$.
Then $\iota$ is a well-defined poset map: since any orthogonal decomposition $\pi \in \mathring{\D}(V)$ yields a non-central elementary abelian $p$-subgroup $\iota(\pi)$ of rank $|\pi|$, and $\iota$ is order-preserving.

On the other hand, if $A\in \Ap(\GU_n(q))_{ > Z}$, then there exists an orthogonal decomposition $V_1\oplus \ldots \oplus V_r$ of $V$ corresponding to the weight spaces of $A$.
This allows us to define a retraction $r: \Ap(\GU_n(q))_{ > Z} \to \mathring{\D}(V)^{\op}$ (note that if $A < B$ then $B$ gives rise to a finer decomposition $r(B)$, so $r(A) \leq r(B)$ and $r$ is a poset map).
It is not hard to check that $\iota \circ r(A) \geq A$ and $r\circ \iota(\pi) = \pi$.
This shows that $\mathring{\D}(V)^{\op}$ is a strong homotopy retract of $\Ap(\GU_n(q))_{ > Z}$ and $\iota \circ r\geq \Id$, $r \circ \iota = \Id$.
Then $\iota$ is an inclusion map of posets and $\iota,r$ are homotopy equivalences (see \cite[p.103]{Qui78}).
This finishes the proof of (i).

Using the considerations above we can also show that $\Ap(\GU_n(q))_{ > Z}$ is pure of dimension $n-2$.
Let $A\in \Ap(\GU_n(q))_{ > Z}$ be a maximal element.
Then $\iota \circ r(A) \geq A$ implies that $\iota \circ r(A) =A$.
So, if $r(A) < \pi'$, $A = \iota \circ r(A) < \iota(\pi')$ since $\iota$ is injective.
The maximality of $A$ implies that we cannot have $r(A) < \pi'$, so $r(A)$ is a maximal element of $\mathring{\D}(V)^{\op}$ (and necessarily a full frame, i.e. an orthogonal decomposition of size $n$).
This means that $A = \iota \circ r(A)$ is an elementary abelian $p$-subgroup of rank $n$.
Thus $\Ap(\GU_n(q))_{ > Z}$ is pure of dimension $n-2$.

 On the other hand, by Lemma \ref{lem:fiber}, there exists a continuous map
 between the geometric realizations $|\redS(V)| \to |\mathring{\D}(V)|$ which induces an epimorphism in every homotopy and homology group (with any coefficient ring).
Since $\mathring{\D}(V)^{\op} \simeq \Ap(\GU_n(q))_{>Z}$, by applying Proposition \ref{coro:applicationOtherPosets}(iii), we derive the conclusions of (iii).
In a similar way, we get the conclusions of item (iv).

Next, we suppose that $n < q+1$ and show that item (ii) holds.
Then $\mathring{\D}(V)$ is Cohen-Macaulay over $\kk$ by Proposition \ref{coro:applicationOtherPosets} and Corollary \ref{coro:decomp}.
By the explanation given in the previous paragraph, we also see that $\Ap(\GU_n(q))_{>Z}$ is spherical over $\kk$ of dimension $n-2$.
To conclude Cohen-Macaulayness over $\kk$ for $\Ap(\GU_n(q))_{>Z}$, we have to show that every interval is spherical over $\kk$.

We have three kinds of intervals in the poset $\Ap(\GU_n(q))_{ > Z}$. Regard $Z$ as the minimum $\hat{0}$.
Then we have:
\begin{itemize}
    \item The interval $(\hat{0},\hat{1}) = \Ap(\GU_n(q))_{ > Z}$ is spherical over $\kk$ since $\mathring{\D}(V)^{\op}$ is.
    \item If $B\in \Ap(\GU_n(q))_{ > Z}$ and $A\in \Ap(\GU_n(q))_{ \geq Z}$, then the interval $(A,B)$ is just the poset of proper subspaces of $B/A$, which is spherical.
    \item If $B\in \Ap(\GU_n(q))_{ > Z}$, then we have to show that the interval $$(B,\hat{1}) = \Ap(\GU_n(q))_{ > B}$$ is spherical over $\kk$.
    Clearly, it has the correct dimension.
    On the other hand, 
    \[ \iota^{-1}( \Ap(\GU_n(q))_{ \geq B} ) = \mathring{\D}(V)^{\op}_{\geq r(B) }.\]
    If we have $B = \iota \circ r(B)$, then 
    \[ \iota^{-1}( \Ap(\GU_n(q))_{ > B} ) = \mathring{\D}(V)^{\op}_{> r(B) }\]
    which is spherical over $\kk$.
    And indeed in this case, the restriction $\iota:\mathring{\D}(V)^{\op}_{>r(B)}\to \Ap(\GU_n(q))_{>B}$ is a homotopy equivalence with retraction given by the restriction $r:\Ap(\GU_n(q))_{>B} \to \mathring{\D}(V)^{\op}_{>r(B)}$.
    Thus the upper-interval $\Ap(\GU_n(q))_{>B}$ is spherical over $\kk$.
    
    Now assume that $B < \iota \circ r(B)$.
    In this case, we get that $B_1:= \iota\circ r(B)\in \Ap(\GU_n(q))_{ > B}$, and so, this upper-interval is conically-contractible via the homotopy $T \leq \iota\circ r(T) \geq B_1$.
    In particular, $\Ap(\GU_n(q))_{ > B}$ is spherical.
\end{itemize}
We have proved that every interval in $\Ap(\GU_n(q))_{ > Z}$ is spherical over $\kk$.
This shows that $\Ap(\GU_n(q))_{ > Z}$ is Cohen-Macaulay over $\kk$.
This concludes the proof of (ii) and hence the proof of the proposition.
\end{proof}

Note that the proofs of Proposition \ref{coro:applicationDecompApPosets} (ii) and (iii) on the homological connectivity of the poset $\Ap(\GU_n(q))_{>Z}$ 
can be extended to derive homotopical connectivity, sphericity, or Cohen-Macaulayness, once we know that $\redS(V)$ has the corresponding property.
That is, the above proof also shows that:

\begin{corollary}
\label{coro:connectionCM}
Let $V$ be a unitary space of dimension $n\geq 3$ over a finite field $\GF{q^2}$ and let $p$ be a prime dividing $q+1$.
The following hold:
\begin{enumerate}
    \item $\mathring{\D}(V)$ is Cohen-Macaulay (over $\kk$), if and only if $\Ap(\GU_n(q))_{>Z}$ is Cohen-Macaulay (over $\kk$).
    \item If $\redS(V)$ is ($\kk$-homologically) $m$-connected, spherical (over $\kk$), or Cohen-Macaulay (over $\kk$), then $\mathring{\D}(V)$ and $\Ap(\GU_n(q))_{>Z}$ are ($\kk$-homologically) $m$-connected, spherical (over $\kk$), or Cohen-Macaulay (over $\kk$), resp.
\end{enumerate}
\end{corollary}
\begin{proof}
Item (1) is deduced from the proof of Proposition \ref{coro:applicationDecompApPosets}.
Item (2) follows from item (1), Proposition \ref{prop:decomp} and Corollary \ref{coro:decomp}.
\end{proof}

In the following example, we apply Proposition \ref{coro:applicationOtherPosets} to show that the top-degree homology groups of $\mathring{\D}(V)$ and $\Ap(\GU_n(q))_{>Z}$ are non-zero when $q=3$, $n=7$ and $p = 2$.

\begin{example} \label{ex:qdp}
Let $q = 3$, $n = 7$, $j=4$, and $V$ a unitary space of dimension $n$ over $\GF{q^2}$.
Then $n < P_4(q) = q(q-1)+4$ and $q+1 < n < 2(q+1)$.
By item (iii) of the previous proposition, we see that $\mathring{\D}(V)$ is $\kk$-homologically $n-4= 3$-connected.
Since $\mathring{\D}(V)$ has dimension $7-2=5$, this means that the homology groups that contribute to the reduced Euler characteristic are $H_5(\mathring{\D}(V),\kk)$ and $H_4(\mathring{\D}(V),\kk)$.
Now, since $\tilde{\chi}(\mathring{\D}(V)) = -507209080872632320$, we see that necessarily $H_5(\A_2(\GU_n(q))_{>Z},\kk) \groupiso H_5(\mathring{\D}(V),\kk) \neq 0$.

On the other hand, by Proposition \ref{propQuillenPosetIsomorphisms},
\[ \A_2(\GU_n(q))_{>Z} \equiv \A_2(\PGU_n(q)) = \A_2(\PSU_n(q)).\]
Thus we conclude that
\[0\neq H_5(\A_2(\GU_n(q))_{>Z},\kk) \equiv H_5(\A_2(\PSU_n(q)),\kk)=H_5(\A_2(\PGU_n(q)),\kk).\]
This means that $\PSU_n(q)$ and $\PGU_n(q)$ satisfy the \textit{Quillen dimension property} as defined in \cite{AS93} (see Conjecture 4.1 in \cite{AS93}).
\end{example}

For the following proposition, we adopt the usual convention that $\PSL^{+1}_n(q) = \PSL_n(q)$, $\PSL^{-1}_n(q) = \PSU_n(q)$, and similarly with the possibilities $\SL^{\pm 1}, \GL^{\pm 1}, \PGL^{\pm 1}$.

\begin{proposition}
\label{propQuillenPosetIsomorphisms}
Let $n\geq 2$, $q$ a prime power, $\epsilon \in \{\pm 1\}$ and $p$ a prime number such that $p\mid q- \epsilon$ but $p\nmid n$.
Then we have natural poset isomorphisms:
\[ \A_p(\PSL^{\epsilon}_n(q)) \equiv \A_p(\PGL^{\epsilon}_n(q)) \equiv \A_p(\SL^{\epsilon}_n(q)) \equiv \A_p(\GL^{\epsilon}_n(q))_{>Z}, \]
where $Z$ is the unique cyclic subgroup of order $p$ contained in the center of $\GL^\epsilon_n(q)$.

In particular, $m_p(\PSL^{\epsilon}_n(q)) = m_p(\PGL^{\epsilon}_n(q)) = m_p(\SL^{\epsilon}_n(q)) = m_p(\GL^{\epsilon}_n(q))-1=n-1$.
\end{proposition}

\begin{proof}
Recall that $|\PGL^{{\epsilon}}_n(q):\PSL^{{\epsilon}}_n(q)| = \gcd(n,q-{\epsilon})$, which has order prime to $p$ by hypothesis.
This shows that $\Ap(\PGL^{{\epsilon}}_n(q)) = \Ap(\PSL^{{\epsilon}}_n(q))$.

We also have that $\PSL^{\epsilon}_n(q) = \SL^{\epsilon}_n(q)/Z^*$, where $Z^*$ is a central subgroup (of scalar matrices) of order $\gcd(n,q-{\epsilon})$.
Since $Z^*$ is central and $p$ does not divide its order, we get that the quotient map $\SL^{\epsilon}_n(q)\to \PSL^{\epsilon}_n(q)$ induces indeed an isomorphism of posets
\[ \Ap(\SL^{\epsilon}_n(q))\to \Ap(\PSL^{\epsilon}_n(q)) = \Ap(\PGL^{\epsilon}_n(q)).\]
Therefore it remains to show that $\Ap(\GL^{\epsilon}_n(q))_{>Z}\equiv \Ap(\SL^{\epsilon}_n(q))$.

Let $L:=\SL^{\epsilon}_n(q)$.
We prove the following claim.

\bigskip

\noindent\textbf{Claim.} If $A\in \Ap(\GL^{\epsilon}_n(q))_{>Z}$ then $A = Z\times (A\cap L)$, and $A\cap L > 1$.

\begin{proofofcase}[Proof of Claim.]
Note that if $\lambda \in Z \cap L$, then $\lambda^n = 1$ and $\lambda^p = 1$, so $1=\lambda^{\gcd(n,p)} = \lambda$.
This shows that $Z\times (A\cap L)$ is indeed a direct product.
Now we prove that $A = Z\times (A\cap L)$.
Let $a\in A$.
Consider a combination $1 = k n+ l p$, and set $z:=\det(a)^{k}\Id_n \in Z$.
Thus $\det(z) = \det(a)$, $z\in Z\leq A$ and $z^{-1}a\in A\cap L$.
This proves that $a\in Z\times (A\cap L)$.

Finally, $A > Z$ implies that $A\cap L > 1$.
This finishes the proof of the claim.
\end{proofofcase}

Now we can produce the desired isomorphism:
\begin{align*}
r: \Ap(\GL^{\epsilon}_n (q))_{>Z} & \to \Ap(\SL^{\epsilon}_n(q)),\\
    A & \mapsto r(A) = A\cap L\\
i: \Ap(\SL^{\epsilon}_n (q)) & \to \Ap(\GL^{\epsilon}_n(q))_{>Z},\\
    B & \mapsto i(B) = Z\times B.
\end{align*}
Clearly $r,i$ are well-defined poset maps by the above Claim.
We also get that:
\[ ri(B) = r(Z\times B) = r(Z \times B\cap L) = B\cap L = B,\]
\[ ir(A) = i(A\cap L) = Z\times A\cap L = A.\]
This shows that $r,i$ are isomorphisms (one inverse of the other).
This finishes the proof of the isomorphisms.

Finally, the assertion on the $p$-ranks follows by computing the height of these posets and noting that $m_p(\GL^{\epsilon}_n(q)) = n$.
This concludes the proof of the proposition.
\end{proof}


\begin{example}
Let $n = 3$, $q$ a prime power, and $p$ a prime dividing $q+1$ with $p\neq 3$.
Note that this forces $q\neq 2$.

Let $V$ be a unitary space of dimension $3$ over $\GF{q^2}$, and denote by $Z$  the unique central cyclic subgroup of order $p$ of $\GU(V)$.
Observe that $\redF(V) \cong \mathring{\D}(V)^{\op}$.
Thus, by Corollary \ref{coro:applicationDecompApPosets}(1) and Proposition \ref{propQuillenPosetIsomorphisms},
 \[ \redF(V) =\mathring{\D}(V)^{\op} \simeq \Ap(\GU(V))_{>Z} = \Ap(\SU(V)) = \Ap(\PGU(V)) = \Ap(\PSU(V)).\]
By Proposition \ref{lowDimensionalCases}, these $\Ap$-posets are homotopy equivalent to a bouquet of
\begin{equation}
    \label{spheresApDim1}
    \frac{1}{3} (q^6 - 2 q^5 - q^4 + 2 q^3 - 3 q^2 + 3)
\end{equation}
sphere of dimension $1$.

On the other hand, the value in equation \eqref{spheresApDim1} equals $-\tilde{\chi}(\Ap(\PSU(V)))$.
By standard arguments, since $\Ap(\PSU(V))$ has dimension $1$, this Euler characteristic can also be computed as follows:
\[ \tilde{\chi}(\Ap(\PSU(V))) = -1 + \#\{ A\leq \PSU(V) \tq |A| = p\} - p \cdot\# \{A\in \Ap(\PSU(V))\tq |A|=p^2\}.\]
Thus, the number of elements of order $p$ in $\PSU(V)$ is
\[ (p-1)\cdot\left( p \cdot \# \{A\in \Ap(\PSU(V))\tq |A|=p^2\} - \frac{1}{3} (q^6 - 2 q^5 - q^4 + 2 q^3 - 3 q^2) \right).\]
\end{example}

\subsection{Final remarks}

We close this section by including some final comments.

\begin{lemma}[Contractibility of the orbit-space]
  Let $V$ be a finite dimensional unitary space. Then the orbit
  space $|\F(V)|/\GU(V)$ of the geometric realization $|\F(V)|$
  of $\F(V)$ is contractible.
  
  In particular, $\tilde{H}_*(\F(V),\kk)^{\GU(V)} \groupiso \tilde{H}_*(|\F(V)|/\GU(V),\kk) = 0$, where $\kk$ is a field of characteristic $0$ or prime to $|\GU(V)|$.
  That is, $\tilde{H}_*(\F(V),\kk)$ does not contain the trivial representation.
  \end{lemma}
  \begin{proof}
This is
a straightforward application of the fact that $\GU(V)$ acts flag-transitively on $\F(V)$ (see Remark \ref{rem:elemproperties}(vi)).

Consider the first barycentric subdivision $\sd(\F(V))$ of $\F(V)$.
Observe that since the action of $\GU(V)$ is flag-transitive on $\F(V)$,
the action of $\GU(V)$ on $\sd(\F(V))$ is regular in the sense of \cite[Definition 1.2]{Bredon}.
In particular, we have a transitive action of $\GU(V)$ on the maximal simplices of $\sd(\F(V))$.
Thus, the \textit{orbit-complex} $\sd(\F(V))/\GU(V)$ as defined in \cite[p.117]{Bredon} is a simplex of dimension $\dim(V)-1$.
From this, we conclude that:
\[ |\F(V)|/\GU(V) = |\sd(\F(V))|/\GU(V) = |\sd(\F(V))/\GU(V)|\simeq *,\]
where the second equality holds since the action is regular.

The In particular part follows from well-known general facts about the
action of automorphism groups on homology.
\end{proof}

Note that $\SU(V)$, $\PSU(V)$ and $\PGU(V)$ are also flag-transitive on $\F(V)$, so in the previous lemma we get the same conclusions if we replace $\GU(V)$ by one of these groups.

\medskip

The following remark shows that $\redF(V)$ is not Cohen-Macaulay if the dimension is big enough.

\begin{remark}
\label{rk:nonCM}
    Let $q\geq 3$ be a prime power and let $n \geq 2$.
    If $\redF(V)$ was a wedge of spheres of the maximal possible dimension, namely $n-2$, then its reduced Euler characteristic would have the sign of $(-1)^n$.
    However, we show that, in general, this is not the case since the sign of $\tilde{\chi}(\redF(V))$ is $(-1)^{n-1}$ if $q(q-1)+1 \leq n \leq q^3(q-1)+q$.
    Furthermore, $H_{n-3}(\F(V),\kk)\neq 0$ where $\kk$ is a field of characeristic $0$.
    
    Let $f_m$ denote the number of $m$-frames, with $f_0 = 1$.
    We exhibit values of $n$ for which the inequalities
    \begin{equation}
        \label{eq:inqualitiesWrongSign}
        f_{m+1} - f_{m} >0 \quad \text{ and }
    \quad f_n-f_{n-1}+f_{n-2}-f_{n-3} > 0
    \end{equation}
    hold for all $0\leq m\leq n-3$.
    
    For the first inequality of \eqref{eq:inqualitiesWrongSign}, note that
    \begin{align*}
        f_{m+1} - f_m & = \frac{|\GU_n(q)|}{(q+1)^{m+1}(m+1)! |\GU_{n-m-1}(q)|} - \frac{|\GU_n(q)|}{(q+1)^{m}m! |\GU_{n-m}(q)|}\\
        & = \frac{q^{\binom{n}{2}} {\prod_{i=n-m+1}^n (q^i-(-1)^i)} }{(q+1)^m (m+1)! q^{ \binom{n-m-1}{2} }} \left( \frac{(q^{n-m} - (-1)^{n-m})}{q+1} - \frac{m+1}{q^{n-m-1}}\right).
    \end{align*}
    Thus $f_{m+1} - f_m > 0$ if and only if 
    \begin{equation}
        \label{eq:inequalitiesWrongSignConsecutive}
        h_q(m) := \frac{q^{n-m-1}(q^{n-m} - (-1)^{n-m})}{q+1} - (m+1)>0.
    \end{equation}
    Since $h_q(m)$ is monotone nonincreasing on $m$, $h_q(m) > 0$ if and only if $h_q(n-3)>0$, if and only if $q^2 (q(q-1)+1) + 2 > n$.

    A similar computation shows that $f_n-f_{n-1}+f_{n-2}-f_{n-3} > 0$ if and only if
    \[ q^2(q(q-1)+1)n > n(n-2) + q^3(q(q-1)+1)(q-1).\]
    This inequality holds, for instance, if $q(q-1)+1 \leq n \leq q^3(q-1)+q$.
    Moreover, it is not hard to conclude that this in fact implies $H_{n-3}(\F(V),\kk)\neq 0$.
    
    Finally, if \eqref{eq:inqualitiesWrongSign} holds, the following formula shows that the sign of the Euler characteristic is $(-1)^{n-1}$:
    \begin{align*}
       (-1)^{n-1} \tilde{\chi}(\F(V)) 
       & = f_n-f_{n-1}+f_{n-2}-f_{n-3} + \begin{cases}
           1 + \sum_{k=1}^{n/2-2} f_{2k}-f_{2k-1} & \text{$n$ even,}\\
           \sum_{k=0}^{(n-1)/2-2} f_{2k+1}-f_{2k} & \text{$n$ odd.}
       \end{cases}
    \end{align*}
    In summary, we have proved that, if $n \geq q(q-1)+1$ then:
    \begin{enumerate}
        \item $\widetilde{H}_{n-3}(\F(V),\kk)\neq 0$ if $q^2(q(q-1)+1)n > n(n-2) + q^3(q(q-1)+1)(q-1)$, which in particular holds for $n \leq q^3(q-1)+q$.
        \item If $n \leq q^3(q-1)+q$ then $\tilde{\chi}(\F(V))$ has sign $(-1)^{n-1}$.
        \item In particular, $\redF(V)$ is not Cohen-Macaulay over any field.
    \end{enumerate}
\end{remark}

\medskip

We end up with some remarks on the fundamental group of the frame complex.

\begin{remark}
The simple connectivity of $\F(V)$ is related to amalgams of the special unitary group: since $\SU(V)$ is flag-transitive on $\F(V)$ (which we can regard as a Tits geometry), we can take the \textit{amalgam of parabolics}.
This is the amalgam of the stabilizers of the non-empty sub-flags of a fixed maximal flag.
A maximal flag is just a chain of size $n$ of the form $(\{V_1\} <  \{V_1,V_2\} < \ldots < \{V_1,\ldots,V_n\})$, where $\{V_1,\ldots,V_n\}$ is a maximal frame of $V$.
When $\F(V)$ is connected, $\SU(V)$ is generated by the stabilizers of these sub-flags by Lemma 1.4.2 of \cite{Ivanov}.
Moreover, by Tits' Lemma, $\F(V)$ is simply connected if and only if the natural epimorphism from the universal enveloping group of the amalgam of parabolics is indeed an isomorphism (that is, the amalgam is $\SU(V)$).
In general, the ``only if'' part of the equivalence is used.
Nevertheless, for the case $n = 4$ and $q > 3$, one might attempt to apply the ``if'' part of Tits' Lemma to get simple connectivity of $\F(V)$.
An analogous situation arises if we consider $\redF(V)$.
\end{remark}

\begin{remark}
\label{rk:propertyT}
On the other hand, by the results of \cite{Zuk} we see that for $n=4$ and $q>3$, i.e. $n<q+1$, $\pi_1(\F(V))$ satisfies Kazhdan Property T.
\end{remark}

\appendix

\section{Counting formulas for finite Hermitian spaces}
\label{appendixCountingVectorsAndSubspaces}

In this appendix we collect some well-known formulas counting various types of objects in unitary spaces.

Let $V$ be a unitary space of dimension $n$ over $\GF{q^2}$.
Recall that
\[ |\GU(V)| = |\GU_n(q)| = q^{ \binom{n}{2} } \prod_{i=1}^n (q^i-(-1)^i).\]

By Witt's lemma, every isometry $\alpha:S\to T$ between subspaces $S,T\leq V$ extends to an isometry of $V$ (see for example Section 20 of \cite{AscFGT}).
This implies that $\GU(V)$ acts transitively on the set
\[ \S_m(V) := \{ S \in \S(V) \tq \dim(S) = m \}. \]

Denote by $\F(V)_m$ the set of $m$-frames of $V$.
Fix $\sigma\in\F(V)_m$.
The stabilizer of $\sigma$ under the action of $\GU(V)$ is
\[\Stab_{\GU(V)}(\sigma) \groupiso \left( \GU_1(q) \wr \Sym_m \right) \times \GU_{n-m}(q).\]
Here $\Sym_m$ denotes the symmetric group on $m$ letters, and $A\wr \Sym_m$ is the wreath product of the group $A$ by $\Sym_m$.

Then, we have an isomorphism of $\GU(V)$-sets:
\[ \F(V)_m \equiv \GU(V) / \Stab_{\GU(V)}(\sigma).\]
In particular we get:
\begin{equation}
\label{eqNumberMFrames}
|\F(V)_m| = \frac{|\GU_n(q)|}{(q+1)^m m! |\GU_{n-m}(q)|}.
\end{equation}
From the above equation we can compute the Euler characteristic of $\F(V)$.
Note that $|\F(V)_0| = 1$, which corresponds to the unique $0$-frame.

\begin{lemma}
\label{eulerCharacteristicFrameComplex}
Let $V$ be a unitary space of dimension $n$ over $\GF{q^2}$.
Then
\[\tilde{\chi}\left(\F(V)\right) = \sum_{m=0}^n (-1)^{m+1} |\F(V)_m| = \sum_{m=0}^n (-1)^{m+1} \frac{|\GU_n(q)|}{(q+1)^m m! |\GU_{n-m}(q)|}.\]
\end{lemma}

From now on, assume that $V$ is equipped with a possibly degenerate Hermitian form.
Table \ref{tab:generalCountingIsotropicAndNonDegenerate} lists formulas for the number of isotropic vectors/subspaces and
non-degenerate subspaces.
Recall that
\[ d_{n+1} = \frac{q^{n-1}(q^{n}-(-1)^{n})}{q+1} \]
records the number of $1$-dimensional non-degenerate subspaces in a unitary space of dimension $n$.




\begin{table}[ht]
    \centering
    \begin{tabular}{|c|c|}
    \hline
        Number of &  \\
    \hline
    &\\
        $I^R_n = $ non-zero isotropic vectors & $q^{2n}-1-q^{2R} \cdot d_{n-R+1}\cdot(q^2-1)$\\
        $1$-dim totally isotropic subspaces & $\frac{q^{2n}-1}{q^2-1}-q^{2R}\cdot d_{n-R+1}$\\
        non-isotropic vectors & $q^{2R}\cdot d_{n-R+1}\cdot(q^2-1)$\\
        $d^R_{n+1} = \,1$-dim non-degenerate subspaces & $q^{2R}\cdot d_{n-R+1}$\\
        &\\
        \hline
    \end{tabular}
    
    \medskip
    
    \caption{Counting vectors and $1$-dimensional subspaces in an $n$-dimensional $\GF{q^2}$-vector space with a Hermitian form whose radical has dimension $R\geq 0$.}
    \label{tab:generalCountingIsotropicAndNonDegenerate}
\end{table}

With the notation of Table \ref{tab:generalCountingIsotropicAndNonDegenerate}, $d_{n+1} = d^0_{n+1}$.
We also let $I_n := I_n^0$, and $d^1_1 := 0$.
From Table \ref{tab:generalCountingIsotropicAndNonDegenerate}, we derive some useful identities:

\begin{lemma} \label{lem:identities}
\begin{enumerate}[label=(\roman*)]
\item $I^k_n = q^{2n}-1 - d^k_{n+1}(q^2-1)$,
\item $I_{n-2} = q^{2n-5}-1 + (q-1)q^{n-3}(-1)^n$,
\item $\left(I_{n-1}-I_{n-2}\right)\frac{1}{q^2-1} = q^{2n-5} - q^{n-3}(-1)^{n} = d_{n-1}(q+1)$,
\item $\left(I_{n-1}-I^1_{n-2}\right)\frac{1}{q^2-1} = q^{2n-5}$,
\item $d^1_{n-1} = d_n - q^{2n-4} + q^{2n-5}$,
\item $d_n = q^{2n-4} - q d_{n-1}$,
\item $d_n - d_{n-1} = q^{2n-4} - q^{2n-5} + q^{n-3}(-1)^{n}$,
\item $d^1_{n-1} - d_{n-1} = q^{n-3}(-1)^{n}$.
\end{enumerate}
\end{lemma}

\end{document}